\documentclass[a4paper,12pt]{amsart}
\usepackage[utf8]{inputenc}
\usepackage[T1]{fontenc}
\usepackage[UKenglish]{babel}
\usepackage[a4paper,margin=24.5mm]{geometry}
\usepackage{verbatim}

\allowdisplaybreaks[4]
\usepackage{times}
\usepackage{dsfont,mathrsfs}
\usepackage{amsmath}
\usepackage{amsthm}
\usepackage{amssymb}
\usepackage{amsfonts}
\usepackage{latexsym}
\usepackage{booktabs}
\usepackage{mathtools}
\usepackage{yhmath}
\usepackage{float}

\usepackage[colorlinks,pagebackref]{hyperref}
\usepackage{xcolor}

\newtheorem{theorem}{Theorem}[section]
\newtheorem{lemma}{Lemma}[section]

\newtheorem{corollary}[theorem]{Corollary}
\newtheorem{assumption}{Assumption}

\theoremstyle{definition}
\newtheorem{definition}{Definition}[section]

\newtheorem{remark}{Remark}[section]
\numberwithin{equation}{section}
\renewcommand{\labelenumi}{\roman{enumi})}
\renewcommand\theenumi\labelenumi
\renewcommand{\leq}{\leqslant}
\renewcommand{\le}{\leqslant}
\renewcommand{\geq}{\geqslant}
\renewcommand{\ge}{\geqslant}

\newcommand{\Be}{\begin{equation}}
\newcommand{\Ees}{\end{equation*}}
\newcommand{\Bes}{\begin{equation*}}
\newcommand{\Ee}{\end{equation}}

\newcommand{\R}{\mathbb{R}}
\newcommand{\E}{\mathbb{E}}

\newcommand{\PP}{\mathbb{P}}

\newcommand{\mcl}{\mathcal}

\newcommand{\wtalpha}{\widetilde{\alpha}}
\newcommand{\wtsigma}{\widetilde{\sigma}}
\newcommand{\wtepsilon}{\widetilde{\epsilon}}
\newcommand{\wtX}{\widetilde{X}}
\newcommand{\wtY}{\widetilde{Y}}

\newcommand{\wtS}{\widetilde{S}}
\newcommand{\wtA}{\widetilde{A}}
\newcommand{\wtK}{\widetilde{K}}
\newcommand{\wtchi}{\widetilde{\chi}}
\newcommand{\whX}{\widehat{X}}
\newcommand{\hP}{\bar{P}}
\newcommand{\whP}{\widetilde{P}} %
\newcommand{\whQ}{\widehat{Q}}
\newcommand{\olQ}{\overline{Q}}

\newcommand{\dtv}{d_{\mathrm{TV}}}
\newcommand{\op}{\mathrm{op}}
\newcommand{\dif}{\mathrm{d} }
\newcommand{\eup}{\mathrm{e}}

\DeclareMathOperator{\sgn}{sgn}
\DeclareMathOperator{\Leb}{Leb}
\newcommand*\abs[1]{\left\lvert#1\right\rvert}
\newcommand*\norm[1]{\left\lVert#1\right\rVert}

\usepackage{marginnote}
\begin{document}
\title[Stable CLT in TV distance]{Stable central limit theorem in total variation distance}

\date{}

\author{Xiang Li}
\address[X. Li]{1. Department of Mathematics, Faculty of Science and Technology, University of Macau, Macau S.A.R., China. 2. Zhuhai UM Science \& Technology Research Institute, Zhuhai, China.}
\email{yc07904@um.edu.mo}

\author{Lihu Xu}
\address[L. Xu]{1. Department of Mathematics, Faculty of Science and Technology, University of Macau, Macau S.A.R., China. 2. Zhuhai UM Science \& Technology Research Institute, Zhuhai, China.}
\email{lihuxu@um.edu.mo}

\author{Haoran Yang}
\address[H. Yang]{1. School of Mathematical Sciences, Peking University, Beijing, China. 2. Beijing International Center for Mathematical Research (BICMR), Peking University, Beijing, China.}
\email{yanghr@pku.edu.cn}

\thanks{}

\subjclass{Primary 60E07, 60F05; Secondary 60J05}

\keywords{stable central limit theorem, total variation distance, optimal convergence rate, measure decomposition, backward induction on $\alpha$}

\begin{abstract}
Under certain general conditions, we prove that the stable central limit theorem holds in the total variation distance and get its optimal convergence rate for all $\alpha \in (0,2)$. Our method is by two measure decompositions, one step estimates, and a very delicate induction with respect to $\alpha$. One measure decomposition is light tailed and borrowed from \cite{BC16}, while the other one is heavy tailed and indispensable for lifting convergence rate for small $\alpha$. The proof is elementary and composed of the ingredients at the postgraduate level.  Our result clarifies that when $\alpha=1$ and $X$ has a symmetric Pareto distribution, the optimal rate is $n^{-1}$ rather than $n^{-1} (\ln n)^2$ as conjectured in literatures.   
\end{abstract}
\maketitle

\section{Introduction and Main Results}

It is well known that Kolmogorov distance is an important measurement for probability distributions and has enormous applications in statistics, e.g., empirical process and Kolmogorov-Smirnov statistic. However, Kolmogorov distance only works well for one dimensional probability. For multidimensional probability distributions, one often uses the total variation (TV) distance as a replacement, it is much stronger than Kolmogorov distance and therefore much harder to be handled.  

  
  The one dimensional normal central limit theorem (CLT) in TV distance was first initialized in \cite{Prohorov1952} by Prohorov, while the development of the multidimensional normal CLT is much more involved, see \cite{SM1962,Bhattacharya1968} for the CLT on special sets and \cite{BR2010}  for a review in this direction. Surprisingly, the existence and the optimal convergence rate of the multidimensional normal CLT were figured out by Bally and Caramellino \cite{BC16} in 2016.  
  
  
For heavy tailed random variables without second moment, the stable CLT holds as a natural replacement of the normal one. Its convergence rate in Kolmogorov distance was intensively studied by Hall in 1970s and 1980s, see \cite{H8102,H8109}. Although Hall obtained in \cite{H8109}  two sided bounds, they do not imply an optimal rate. In 1999, Kuske and Keller first provided in \cite{KK2000} an optimal rate in Kolmogorov distance for $\alpha \in (1,2)$ as the heavy tailed random variables have symmetric Pareto distribution. Stable CLT in 1-Wasserstein distance was recently studied in a series of papers, see \cite{Xu201902,Xu2021} for $\alpha \in (1,2)$, \cite{Xu2022} for $\alpha \in (0,1]$, and \cite{Xu2023}  for the multidimensional case.

To the best of our knowledge, it has not been known whether the stable CLT holds in the TV distance, not to mention its convergence rate. In this paper, we show that as long as the heavy tailed random variables are i.i.d. with a distribution which falls in the domain of normal attraction of stable distribution and is locally lower bounded by the Lebesgue measure, then the stable CLT holds in TV distance with a certain convergence rate. We further prove that this rate is optimal by showing a lower bound for the symmetric Pareto distributed random variables. When the dimension is one, our result immediately implies the optimal convergence rate in Kolmogorov distance.

Our result clarifies that when $\alpha=1$ and $X$ has a symmetric Pareto distribution, the optimal rate is $n^{-1}$ rather than $n^{-1} (\ln n)^2$ as conjectured in literatures.  



 \subsection{Preliminary and Notations}
 Throughout the paper, $\R^d$ denotes the $d$-dimensional Euclidean space, with norm $\abs{\cdot}$ and scalar product $\langle \cdot, \cdot \rangle$. The open ball centered at $x \in \R^d$ with a radius of $R > 0$ is denoted by $B (x, R) = \{ y \in \R^d \colon \abs{y - x} < R \}$.
For $r, s \in \R$, we denote
$r \lor s = \max \{r, s\}$, $r \land s = \min \{r, s\}$, and $[r] = \max \{ z \in \mathbb{Z} \colon z \leq r \}$.

Denote by $\mathcal{B}_{b} (\R^d)$ the set of bounded measurable functions $f \colon \mathbb{R}^{d}\to \mathbb{R}$. If $f$ is sufficiently smooth, $\partial_i f = \partial f / \partial x_i$, $i = 1, \dots, d$ denote its partial derivatives, and $\nabla f$ denotes the gradient of $f$. Given $v \in \R^d$, the directional derivative of $f$ in the direction $v$ is denoted by $\nabla_v f (x)  = \langle \nabla f (x) , v \rangle$, $x \in \R^d$. For $\kappa = 1, 2, \dots$, we denote by $\nabla^{\kappa} f \colon \R^d \to \R^{d^{\otimes \kappa}}$ the $\kappa$-th order gradient of $f$, and by convention $\nabla^{0} f = f$. We further denote
\begin{align*}
    \mathcal{C}_{b}^{\kappa} (\R^d) = \left\{ f \colon \mathbb{R}^{d}\to \mathbb{R}; f, \nabla f, \dots, \nabla^{\kappa} f \text{ are continuous and bounded} \right\},
\end{align*}
and for $f \in \mathcal{C}_{b}^{\kappa} (\R^d)$,
\begin{align*}
    \norm{\nabla^{\kappa} f}_{\op, \infty} = \sup \left\{ \abs{\nabla_{v_1} \dots \nabla_{v_\kappa} f (x)} \colon x, v_1, \dots, v_\kappa \in \R^d; \abs{v_1}, \dots, \abs{v_\kappa} \leq 1 \right\}.
\end{align*}

For two probability measures $\mu$ and $\nu$ defined on $\mathbb{R}^{d}$, the total variation distance between them is given by 
\begin{align*}
    \dtv(\mu,\nu)=\inf_{\pi\in \Pi (\mu,\nu)} \left\{ \int_{\mathbb{R}^{d}\times \mathbb{R}^{d}} \mathbf{1}_{\{x\neq y\}} \pi(\mathrm{d}x,\mathrm{d}y) \right\},
\end{align*}
where $\Pi (\mu,\nu)$ is the class composed of all couplings of $\mu$ and $\nu$. It is well known that, by Kantorovich-Rubinstein theorem \cite{Villani2003},
\begin{align} \label{e:KR-TV}
    \dtv(\mu,\nu)=\sup_{f\in \mathcal{B}_{b}(\R^d), \, \norm{f}_\infty \leq 1} \abs{ \int_{\R^d} f(x) \mu (\dif x) - \int_{\R^d} f(x) \nu (\dif x)}.
\end{align}

Finally, we remark that $C$ denotes a positive constant which may be different even in a single chain of inequalities.
\begin{definition}
     Let $\mathbb{S}^{d-1}$ be the unit sphere in $\mathbb{R}^{d}$, i.e., $\mathbb{S}^{d-1}=\{ u\in \mathbb{R}^{d} \colon |u|=1 \}$ and $\alpha\in (0,2)$. We say that a random vector $Y$ has a multivariate strictly $\alpha$-stable distribution with spectral measure $\nu$ (finite measure on $\mathbb{S}^{d-1}$) denoted as $S_{\alpha}(\nu)$, if the characteristic function of $Y$ is given by
    \begin{align*}
        \mathbb{E}\left[e^{i\langle \lambda,Y \rangle} \right]
        = \exp\left( -\int_{\mathbb{S}^{d-1}} \psi_\alpha(\langle \lambda,\theta \rangle) \nu(\dif \theta)\right), \quad \lambda\in \mathbb{R}^{d},
    \end{align*}
    where $\psi_\alpha: \R \rightarrow \mathbb C$ is defined as 
    \begin{align*}
    \psi_\alpha(t) = \begin{cases}
        \abs{t}^{\alpha}\left(1-i \sgn (t) \tan \frac{\pi \alpha}{2}\right), &\alpha \in (0, 2) \setminus \{1\}; \\
\abs{t} \left(1+i  \frac{2}{\pi} \sgn (t) \ln \abs{t} \right), &\alpha=1.
\end{cases}
\end{align*}
\end{definition}
\begin{remark}
   In this paper we always assume $\int_{\mathbb{S}^{d-1}} \theta \nu(\dif \theta)=\mathbf{0}$ when $\alpha=1$ to avoid a singularity which is not typical. Without loss of generality, we also assume $\nu$ to be a probability measure on $\mathbb{S}^{d-1}$ throughout this paper.
\end{remark}

    Let $(\widehat{Y}_t)_{t \geqslant 0}$ be a $d$-dimensional Lévy process satisfying
    \begin{align*}
        \mathbb{E}\left[e^{i\langle \lambda, \widehat{Y}_t \rangle} \right]=\exp\left(-t\int_{\mathbb{S}^{d-1}}\psi_\alpha(\langle \lambda, \theta \rangle ) \nu(\dif \theta) \right).
    \end{align*}
    According to \cite[Theorem 14.10]{Ken1999}, the above characteristic function also has the following Lévy–Khintchine representation:
    \begin{align} \label{def:mY}
        \mathbb{E}\left[e^{i\langle \lambda, \widehat{Y}_t \rangle} \right]
        =\exp\left(- d_{\alpha}t \int_{\mathbb{S}^{d-1}} \nu(\dif \theta ) \int_{0}^{+\infty} \left(e^{i\langle \lambda, r\theta\rangle}-1-ik_{\alpha}(r)\langle \lambda,r\theta \rangle  \right) \frac{\dif r}{r^{1+\alpha}}\right),
    \end{align}
    where $k_\alpha(r)=\mathbf{1}_{ \{ \alpha=1, r \in(0,1] \} }+\mathbf{1}_{ \{ \alpha \in(1,2) \} }$ and $d_{\alpha}=\left(\int_{0}^{+\infty} \frac{1-\cos y}{y^{1+\alpha}} \, \dif  y\right)^{-1}$.

It is clear that $\widehat{Y}_{1}\sim S_{\alpha}(\nu)$ and $\widehat{Y}_{t}\stackrel{d}{=}t^{\frac{1}{\alpha}}Y$ (i.e., $\widehat{Y}_{t}$ and $t^{\frac{1}{\alpha}}Y$ have the same distribution). By \cite[Theorem 3.3.3]{Applebaum2009}, one can verify that the generator of $\widehat{Y}_{t}$ is given by
\begin{align} \label{mYgenerator}
    \mathcal{L}^{\alpha,\nu}f(x)=d_{\alpha} \int_{\mathbb{S}^{d-1}} \nu(\dif \theta) \int_{0}^{+\infty} \frac{f(x+r\theta)-f(x)-k_{\alpha}(r)\langle \nabla f(x), r\theta\rangle}{r^{1+\alpha}} \dif r, \ \ \ f \in \mathcal C_b^2(\mathbb R^d). 
\end{align}

\begin{definition}[Domain of attraction]
    For $\alpha\in (0,2)$ and a spectral measure $\nu$, we say that a distribution $\mu$ is in the domain of attraction of the strictly stable law $S_\alpha (\nu)$ if there exists a sequence of location vectors $(b_n)_{n \geqslant 1}$ and a slowly varying function $h \colon [0,+\infty)\to (0,+\infty)$ (i.e., for any $k>0$, $\lim_{x\to +\infty}h(kx)/h(x)=1$), such that
\begin{align*}
	\frac{\sum_{i = 1}^n X_i - b_n}{n^{1 / \alpha}h(n)} \Rightarrow S_\alpha (\nu),
\end{align*}
where $(X_n)_{n \geqslant 1}$ is a sequence of i.i.d.\ random variables with common distribution $\mu$.
\end{definition}
\begin{remark}
     When $d=1$, a distribution $\mu$ is said to be in the domain of normal attraction of $S_{\alpha}(\nu)$, if there exists a positive constant $c$ such that
    \begin{align*}
	\frac{\sum_{i = 1}^n X_i - b_n}{cn^{1 / \alpha}} \Rightarrow S_\alpha (\nu).
\end{align*}
  It has been shown in \cite[Theorem 2.6.7]{Ibragimov1971} that a distribution is in the domain of normal attraction of the stable law $S_\alpha (\nu)$ if and only if its corresponding distribution function is of the form 
\begin{align} \label{Xcdf}
    F (x)=\left[ 1-\frac{A+\epsilon(x)}{x^{\alpha}}\nu(\{1\}) \right] \mathbf{1}_{[0, +\infty)}(x)+\left[ \frac{A+\epsilon(x)}{(-x)^{\alpha}}\nu(\{-1\}) \right] \mathbf{1}_{(-\infty, 0)}(x),
\end{align}
where $A>0$ and $\epsilon: \mathbb{R} \rightarrow \mathbb{R}$ is a bounded measurable function vanishing at infinity.
\end{remark}


\subsection{Main results and the strategy of the proofs}
Let us first state our assumptions before giving the main results.  The first assumption, Assumption \ref{A:mX}, is a   {multi-dimensional} version of normal domain of attraction with an additional condition $\gamma>0$. 
\begin{assumption} \label{A:mX}
    Let $\alpha\in (0,2)$ and $\nu$ be a probability measure on $\mathbb{S}^{d-1}$. Suppose $X$ is a $d$-dimensional random vector such that for any $r > 0$ and $B\in \mathscr{B}(\mathbb{S}^{d-1})$, 
    \begin{align} \label{def:mX}
        \mathbb{P}\left(\abs{X} \ge r, \, \frac{X}{\abs{X}}\in B \right)=\int_{B} \frac{A+\epsilon(r,\theta)}{r^{\alpha}} \nu(\dif \theta),
    \end{align}
    where $A>0$ and $\epsilon \colon \mathbb{R_+} \times \mathbb{S}^{d-1} \to \mathbb{R}$ is a bounded measurable function satisfying 
    \begin{align*}
        \sup_{\theta\in \mathbb{S}^{d-1}} \abs{\epsilon(r,\theta)}\le K ( 1 \wedge r^{-\gamma} ),
    \end{align*}
    for some $K,\gamma>0$.
\end{assumption}

\begin{definition}
    A probability distribution $\mu$ in $\mathbb{R}^d$ is said to be locally lower bounded by the Lebesgue measure, if there exists a constant $\varepsilon_{0}>0$ and an open set $D_{0} \subseteq \mathbb{R}^d$ such that
    \begin{align}
        \mu(E) \geq \varepsilon_0 \Leb (E \cap D_0), \quad \forall E \in \mathscr{B} (\mathbb{R}^d ),
    \end{align}
    where $\Leb (\cdot)$ denotes the Lebesgue measure.
\end{definition}

\begin{assumption} \label{A:nu}
    Suppose that the spectral measure $\nu$ on $\mathbb{S}^{d-1}$ satisfies one of the following assumptions,

    \noindent(i) $\nu$ possesses a density function $\varrho (\theta)$ with respect to the surface measure on $\mathbb{S}^{d-1}$, and there exists a constant $\varrho_0 > 1$ such that
    \begin{align*}
        \varrho_0^{-1} \leq \varrho (\theta) \leq \varrho_0, \quad \forall \theta \in \mathbb{S}^{d-1}.
    \end{align*}

    \noindent(ii) $\nu$ is symmetric, i.e., $\nu (B) = \nu (-B)$ for any $B \in \mathscr{B}(\mathbb{S}^{d-1})$ (the class of Borel measurable sets on $\mathbb{S}^{d-1}$). And there exist constants $c > 0$ and $\eta \in (d - \alpha, d] \cap [1, +\infty)$ such that
    \begin{align*}
        \nu ( B_{\mathbb{S}^{d-1}} (\theta, R) ) \leq c R^{\eta - 1}, \quad \forall \theta \in \mathbb{S}^{d-1}, \, R \in (0, 1),
    \end{align*}
    where $B_{\mathbb{S}^{d-1}} (\theta, R) = \{ \vartheta \in \mathbb{S}^{d-1} \colon \abs{\vartheta - \theta} < R \}$.
\end{assumption}

Let $X_{1}, X_{2}, \dots$ be i.i.d.\ copies of $X$, and $\mu_n^\alpha$ denotes the distribution of $S_n$ defined by
\begin{align*}
    S_{n} = \frac{1}{n^{1 / \alpha} \sigma} \sum_{i=1}^{n}(X_{i}-\omega_{n,\alpha}),
\end{align*}
where 
\begin{align} \label{mu}
    \sigma=\left(A \alpha \int_{0}^{+\infty} \frac{1-\cos y}{|y|^{1+\alpha}} \, \dif y\right)^{\frac{1}{\alpha}}, \qquad
    \omega_{n,\alpha}= \begin{cases} \mathbb{E} X_{1}, & \alpha \in(1,2); \\ \mathbb{E}\left[X_{1} \mathbf{1}_{(0, \sigma n]}\left(\left|X_{1}\right|\right)\right], & \alpha=1; \\
    \mathbf{0}, & \alpha \in(0,1). \end{cases}
\end{align}

\begin{theorem} \label{theorem1}
    Assume the common distribution of $X_1, X_2, \dots$ is locally lower bounded by the Lebesgue measure and satisfies Assumption \ref{A:mX} with $\alpha \in (0,2)$, $\gamma > 0$. Suppose that $\nu$ satisfies Assumption \ref{A:nu}. Then there exists a constant $C$ not depending on $n$, such that
    
    \noindent (i) For $\alpha \in (1, 2)$,
    \begin{align*}
        \dtv (\mu_n^\alpha, S_\alpha (\nu)) 
        &\leqslant C \left[ n^{- \frac{2 - \alpha}{\alpha}} + n^{-\frac{\gamma}{\alpha}} (\ln n)^{\mathbf{1}_{ \{ \gamma = 2 - \alpha \} }} \right], \quad \forall n \geq 1.
    \end{align*}
    
    \noindent (ii) For $\alpha = 1$ and $\nu$ satisfying $\int_{\mathbb{S}^{d-1}} \theta \nu ( \dif \theta ) = \mathbf{0}$,
    \begin{align*}
        \dtv (\mu_n^1, S_1 (\nu)) 
        \leqslant C \left[ n^{-1} + n^{- \gamma} (\ln n)^{\mathbf{1}_{ \{ \gamma = 1 \} }} \right], \quad \forall n \geq 1.
    \end{align*}
    
    \noindent (iii) For $\alpha \in (0, 1)$,
    \begin{align*}
        \dtv (\mu_n^\alpha, S_\alpha (\nu)) \leqslant C \left[ n^{-1} + n^{-\rho_{\alpha, \gamma}} (\ln n)^{\mathbf{1}_{ \{ \gamma = 1 - \alpha \} }} \right], \quad \forall n \geq 1.
    \end{align*}
    where
    \begin{align*}
        \rho_{\alpha, \gamma} = \begin{cases}
            \frac{1 - \alpha}{\alpha}, &\gamma \in (1 - \alpha, +\infty); \\
            \frac{\gamma}{1 - \gamma}, &\gamma \in (0, 1 - \alpha].
        \end{cases}
    \end{align*}
    If the distribution of $X$ is symmetric in the sense that $\int_{\mathbb{S}^{d-1}} \theta \epsilon (r, \theta) \nu (\dif \theta) = \mathbf{0}$, $\forall r > 0$, the conclusion still holds with $\rho_{\alpha, \gamma}$ replaced by $\frac{\gamma}{\alpha \lor (1 - \gamma)}$, i.e.
    \begin{align*}
        \dtv (\mu_n^\alpha, S_\alpha (\nu)) \leqslant C \left[ n^{-1} + n^{-\frac{\gamma}{\alpha \lor (1 - \gamma)}} (\ln n)^{\mathbf{1}_{ \{ \gamma = 1 - \alpha \} }} \right], \quad \forall n \geq 1.
    \end{align*}
\end{theorem}

The following theorem shows that the order of convergence rate in the above theorem is optimal.

\begin{theorem} \label{corollary1}
		Suppose $(X_n)_{n \geqslant 1}$ are i.i.d.\ $d$-dimensional random vectors with Pareto distribution of index $\alpha \in (0, 2)$, i.e., the density function of $X_n$ is
		\begin{align*}
		  p(x) = \frac{\alpha \Gamma (\frac{d}{2} + 1)}{\pi^{d / 2} d \abs{x}^{d + \alpha}} \mathbf{1}_{[1, +\infty)} (\abs{x}),
		\end{align*}
		where $\Gamma (z) = \int_0^{+\infty} t^{z-1} e^{-t} \, \dif t$ denotes the gamma function. Then for any $\alpha \in (0,2)$ and $d \geq 1$, there exist constants $C, c > 0$ such that
		\begin{align*}
            c n^{-\frac{(2 - \alpha) \land \alpha}{\alpha}}
            \leq \dtv ( \mu_n^\alpha, S_\alpha (\nu_0) )
            \leq C n^{-\frac{(2 - \alpha) \land \alpha}{\alpha}}, \quad \forall n \geq 1,
		\end{align*}
		where $\nu_0$ denotes the uniform probability measure on $\mathbb{S}^{d-1}$.
\end{theorem}

Let us now sketch out the strategy for the proof of Theorem \ref{theorem1}. Let $Y_{1}, Y_{2}, \dots$ be $d$-dimensional i.i.d.\ random vectors with common distribution $S_{\alpha}(\nu)$. For $1 \leq m \leq n$, define linear operators $P_m$ and $Q_m$ with
\begin{align} \label{def:PQ}
    P_m f (x) = \E f \left( x + \frac{1}{n^{1 / \alpha}} \sum_{i=1}^{m} Y_i \right), \quad
    Q_m f (x) = \E f \left( x + \frac{1}{n^{1 / \alpha}\sigma} \sum_{i=1}^{m} \left( X_{i} - \omega_{n,\alpha} \right) \right), 
\end{align}
where $f \in \mathcal B_b(\mathbb R^d)$.  Thanks to \eqref{e:KR-TV},  it suffices to show that for all $f \in \mathcal B_b(\R^d)$, 
$$\left|P_n f(\mathbf{0})-Q_n f(\mathbf{0})\right| \le d_n \|f\|_\infty,$$
where $d_n$ is the rate defined in Theorem \ref{theorem1}. To this end, we decompose $P_n f(\mathbf{0})-Q_n f(\mathbf{0})$ as 
$$P_n f(\mathbf{0})-Q_n f(\mathbf{0})=\sum_{k=0}^{n-1} Q_{n-k-1} (Q_1-P_1) P_{k} f(\mathbf{0})=I_1+I_2,$$
 where $I_1=\sum_{k=0}^{n'} Q_{n-k-1} (Q_1-P_1) P_{k} f(\mathbf{0})$, $I_2=\sum_{k=n'+1}^{n-1} Q_{n-k-1} (Q_1-P_1) P_{k} f(\mathbf{0})$, and $n'$ is a large number. 
 The problem is now how to estimate the terms  $Q_{n-k-1} (Q_1-P_1) P_{k} f(\mathbf{0})$ in the sum, we call their bounds as \emph{one step estimates}. 
 $(Q_1-P_1)$ can be taken as a   {pseudo-}differential operator with an order less than four,   {so} bounding $P_n f(\mathbf{0})-Q_n f(\mathbf{0})$ is essentially a gradient estimate problem. 

\begin{figure}[htb]
    \centering
    \includegraphics[height=8cm]{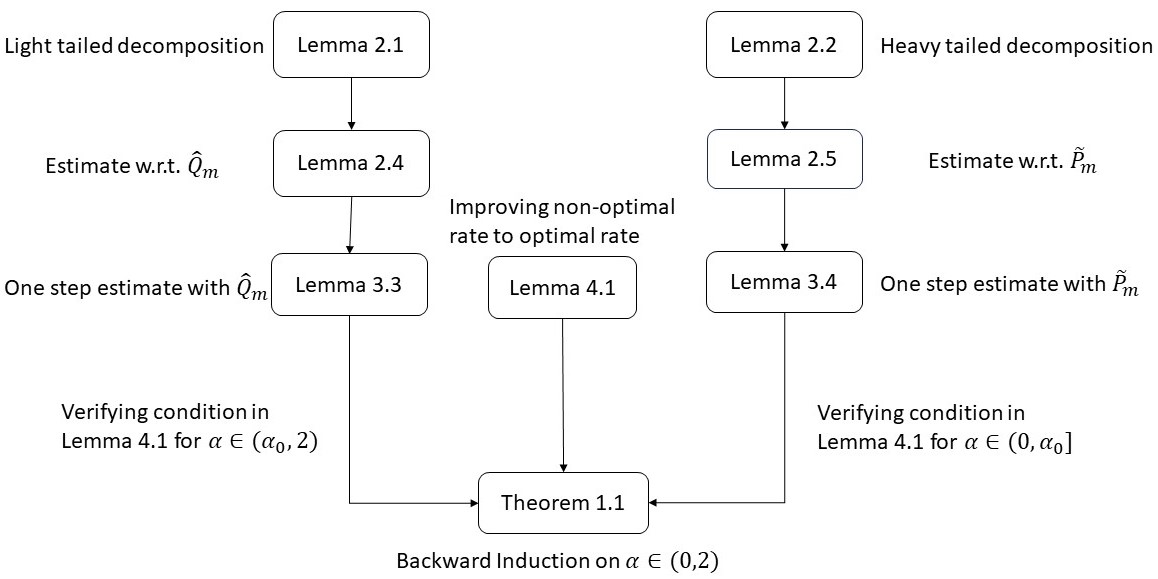}
    \caption{Roadmap for the proof of Theorem \ref{theorem1}}
    \label{fig:enter-label}
\end{figure}
 
Thanks to the smoothening effect of $P_k$ for a large $k$, $I_2$ is easy to handle. In order to bound the terms with a small $k$ in $I_1$, we need to dig out the smoothening effect of $Q_k$ by two measure decompositions, one is a light tailed decomposition and the other is a heavy tailed one. 
More precisely, the light tailed decomposition is borrowed from \cite{BC16}, in which one decomposes the random variable $X$ as a stochastic combination of $\whX$ and $U$ with $\whX$ having smooth density and compact support and $U$ being the remainder. Using this decomposition, we prove the crucial one step estimate for $\alpha \in (\alpha_0,2)$ with $\alpha_0$ being some constant larger than $1$. However, as $\alpha \in (0,\alpha_0]$, the light tailed decomposition does not work, we need to replace the above light tailed $\whX$ with a heavy tailed $\wtX$, which has heavy tailed behaviour like $\abs{x}^{-(d+\wtalpha)}$ with some $\wtalpha \in (\alpha, 2)$. 
Another key ingredient for bounding $I_1$ is the following identity:  
\begin{align*}
    Q_k (Q_1 - P_1) f= (Q_m - \whP_m)^{\otimes r} (Q_1 - P_1) Q_{k - rm} f + \sum_{j=0}^{r-1} (Q_m - \whP_m)^{\otimes j} Q_{k - (j+1) m} (Q_1 - P_1) \whP_m f,
\end{align*}    
where $\whP_m$ is defined by \eqref{e:whPm} {below}, and has the same smoothening effect as $P_m$. In contrast, Bally and Caramellino \cite{BC16} used Malliavin calculus to overcome the same problems. 


\vspace{8 pt}
\noindent
{\bf Acknowledgements}: We would like to gratefully thanks to Qi-Man Shao and Feng-Yu Wang for numerous very helpful comments and suggestions on early drafts.

\section{Measure decomposition and the related estimates} \label{sec2}
Recall the two family of operators $(P_m)_{1 \le m \le n}$ and $(Q_m)_{1 \le m \le n}$ defined in \eqref{def:PQ}. 
It is easy to see that $\norm{P_m f}_\infty \leq \norm{f}_\infty$, $\norm{Q_m f}_\infty \leq \norm{f}_\infty$ for $f \in \mathcal{B}_{b} (\R^d)$. Thanks to the independence, it is easy to verify that for $k, l \geq 1$,
\begin{align*}
    P_k P_l = P_{k+l}, \qquad Q_k Q_l = Q_{k+l}, \qquad P_k Q_l = Q_l P_k.
\end{align*}
Moreover, for any $f \in \mathcal{C}_{b}^1 (\R^d)$,
\begin{align*}
    \partial_i (P_k f) = P_k (\partial_i f), \qquad \partial_i (Q_l f) = Q_l (\partial_i f), \qquad i = 1, \dots, d.
\end{align*}
\vskip 2mm

\subsection{Measure decomposition} Under Assumption \ref{A:mX}, the distribution of $X$ may not be absolutely continuous with respect to the Lebesgue measure, which makes it difficult to estimate the gradient of $Q_m f$ directly. Inspired by the decomposition in \cite{BC16}, we consider the following decomposition of the law of $X$. 

\begin{lemma} \label{lemma:decom1}
    Suppose that the distribution of $X$ is locally lower bounded by the Lebesgue measure. Then there exists independent random variable $\chi$ and random vectors $\whX$, $U$ such that 
    
    \noindent (i) $(1 - \chi) \whX + \chi U$ possesses the same distribution as $X$,

    \noindent (ii) $\chi \sim B(1, p)$ for some $p \in (0,1)$, i.e., $\mathbb{P}\left(\chi=1 \right) = 1 - \mathbb{P}\left(\chi=0 \right) = p$,
    
    \noindent (iii) $\whX$ has the following density function
    \begin{align} \label{eq:pdf_tX}
        p_{\whX} (z) = c \eup^{- \frac{1}{\tau^2 - \abs{z-a}^2}} \mathbf{1}_{B (a, \tau)} (z), \qquad
        c = \left( \int_{B (a, \tau)} \eup^{- \frac{1}{\tau^2 - \abs{z-a}^2}} \, \dif  z \right)^{-1},
    \end{align}
    with some $\tau > 0$, $a \in \R^d$.
\end{lemma}

The proof of   {the} above lemma is postponed to Appendix. We shall use the smooth effect of $\whX$ to study the regularity property of $Q_m$.

As a consequence of   {the} above lemma, there exist independent Bernoulli random variables $\chi_i \sim B(1, p)$, $p \in (0,1)$ and random vectors $\whX_i$, $U_i$, $i = 1, \dots, m$ such that $(1 - \chi_i) \whX_i + \chi_i U_i$ possesses the same distribution as $X_i$, and the common density function of $\whX_i$ is displayed in \eqref{eq:pdf_tX}.

For $\omega_{n,\alpha}$ defined in \eqref{mu}, we denote
\begin{align*}
    \xi = \sum_{i = 1}^m (1 - \chi_i), \qquad
    U = \frac{1}{n^{1 / \alpha}\sigma} \sum_{i=1}^{m} \left( \chi_i U_i - \omega_{n,\alpha} \right),
\end{align*}
then it follows that
\begin{align*}
    Q_{m} f (x)
    &= \E f \left( x + \frac{1}{n^{1 / \alpha}\sigma} \sum_{i=1}^{m} (1 - \chi_i) \whX_i + \frac{1}{n^{1 / \alpha}\sigma} \sum_{i=1}^{m} \left( \chi_i U_i - \omega_{n,\alpha} \right) \right) \\
    &= \E f \left( x + \frac{1}{n^{1 / \alpha}\sigma} \sum_{i=1}^{\xi} \whX_i + U \right).
\end{align*}
Further denote
\begin{align*}
    \whQ_m f (x) = \E \left[ f \left( x + \frac{1}{n^{1 / \alpha}\sigma} \sum_{i=1}^{\xi} \whX_i + U \right) \mathbf{1}_{\{\xi \geq \frac{1-p}{2} m \}} \right].
\end{align*}

The measure decomposition in Lemma \ref{lemma:decom1} is not enough for deriving a gradient estimate for our induction. Alternatively, we derive another measure decomposition in the following lemma, which takes the heavy tail property into account. 
\begin{lemma} \label{lemma:decom2}
    Suppose that $X$ is a $d$-dimensional random vector satisfying Assumption \ref{A:mX} with parameters $\alpha \in (0, 2)$ and $\gamma > 0$, i.e.
    \begin{align*}
        \mathbb{P}\left(\abs{X} \ge r, \, \frac{X}{\abs{X}}\in B \right) = \int_{B} \frac{A+\epsilon(r,\theta)}{r^{\alpha}} \nu(\dif \theta), \quad \forall r > 0, \; B \in \mathscr{B}(\mathbb{S}^{d-1}),
    \end{align*}
    with $A$, $K > 0$ and $\epsilon \colon \mathbb{R_+}\times \mathbb{S}^{d-1}\to \mathbb{R}$ satisfying
    \begin{align*}
        \sup_{\theta\in \mathbb{S}^{d-1}} \abs{\epsilon(r,\theta)} \le K ( 1 \wedge r^{-\gamma} ).
    \end{align*}
    Then for any $\wtalpha \in (\alpha, 2)$, there exists independent Bernoulli random variable $\wtchi \sim B(1, q)$, $q \in (0, 1)$ and random vectors $\wtX$, $V$ such that $(1 - \wtchi) \wtX + \wtchi V$ possesses the same distribution as $X$, in which $\wtX$ satisfies Assumption \ref{A:mX} with parameters $\wtalpha$ and $\gamma$, to be concrete,
    \begin{align*}
        \mathbb{P} \left( \lvert \wtX \rvert \ge r, \, \frac{\wtX}{\lvert \wtX \rvert}\in B \right)=\int_{B} \frac{\wtA + \wtepsilon(r,\theta)}{r^{\wtalpha}} \nu(\dif \theta), \quad \forall r > 0, \; B \in \mathscr{B}(\mathbb{S}^{d-1}),
    \end{align*}
    for some $\wtA$, $\wtK > 0$ and $\wtepsilon \colon \mathbb{R_+} \times \mathbb{S}^{d-1}\to \mathbb{R}$ satisfying
    \begin{align*}
        \sup_{\theta\in \mathbb{S}^{d-1}} \abs{\wtepsilon(r,\theta)} \le \wtK ( 1 \wedge r^{-\gamma} ).
    \end{align*}
    Furthermore, the distribution of $\wtX$ is locally lower bounded by the Lebesgue measure if and only if   {so is} $X$.
\end{lemma}

The proof of   {the} above lemma is postponed to Appendix.

As a consequence of   {the} above lemma, there exist independent Bernoulli random variables $\wtchi_i \sim B(1, q)$, $q \in (0, 1)$ and random vectors $\wtX_i$, $V_i$, $i = 1, \dots, m$ such that $(1 - \wtchi_i) \wtX_i + \wtchi_i V_i$ possesses the same distribution as $X_i$, and $\wtX_i$ satisfies Assumption \ref{A:mX} with parameters $\wtalpha$ and $\gamma$.

Denote
\begin{align*}
    \wtsigma &= \left(\wtA \wtalpha \int_{\mathbb{R}} \frac{1-\cos y}{\abs{y}^{1+\wtalpha}} \, \dif  y\right)^{\frac{1}{\wtalpha}}, &
    \omega_{n,\wtalpha} &= \begin{cases}
        \E \wtX_1, &\wtalpha \in (1, 2); \\
        \E \left[\wtX_1 \mathbf{1}_{(0, \wtsigma n]} \left( \lvert \wtX_{1} \rvert \right) \right], &\wtalpha = 1; \\
        \mathbf{0}, &\wtalpha \in (0, 1),
    \end{cases} \\
    \zeta &= \sum_{i = 1}^m (1 - \wtchi_i), &
    V &= \frac{1}{n^{1 / \alpha}\sigma} \sum_{i=1}^{m} \left( \wtchi_i V_i - \omega_{n,\alpha} \right) + \frac{\omega_{\zeta,\wtalpha}}{n^{1 / \alpha} \sigma} \zeta.
\end{align*}

  Parallel to the semigroup $\widehat{Q}_m$ with respect to $\{\widehat{X}_k\}_{k \ge 1}$, it is natural to define a semigroup with respect to $\{\widetilde X_k\}_{k \ge 1}$. Since the distributions between $\frac{1}{\sqrt n} \sum_{k=1}^n \widetilde X_k$ and $\wtY \sim S_{\wtalpha} (\nu)$ are very close, we can avoid introducing this semigroup and use the one of $\wtY$ as a replacement.  More precisely, {let $\wtY \sim S_{\wtalpha} (\nu)$ independent of $\wtchi_i$, $V_i$, $i = 1, \dots, m$, we further denote
\begin{align} \label{e:whPm}
    \whP_m f (x) = \E \left[ f \left( x + \frac{\zeta^{1 / \wtalpha} \wtsigma}{n^{1 / \alpha} \sigma} \wtY + V \right) \mathbf{1}_{\{\zeta \geq \frac{1-q}{2} m \}} \right].
\end{align}}

\subsection{Related estimates} We derive the gradient estimates for $P_m f$, $\whQ_m f$ and $\whP_m f$, and give the error estimates of $Q_m f-\whP_m f$. 

\begin{lemma} \label{lemma:P}
    Let $P_{m}$ be defined in \eqref{def:PQ}, then
    \begin{align*}
        \norm{\nabla^{\kappa} P_{m}f}_{\op,\infty}\le C\norm{f}_{\infty} \left( \frac{m}{n} \right)^{-\frac{\kappa}{\alpha}}, \quad \forall f\in \mathcal{C}^{4}_{b}\left( \R^{d}\right), \; \kappa=1,\dots,4.
    \end{align*}
\end{lemma}

\begin{proof}
    Denote the density function of $\widehat{Y}_t$ by $p_{\widehat{Y}_t}$. According to \cite{Bogdan2020,Chen2016}, either of the two conditions in Assumption \ref{A:nu} implies that
    \begin{align*}
        \int_{\R^d} \abs{\nabla p_{\widehat{Y}_t} (y)} \dif  y \leq C t^{-\frac{1}{\alpha}}, \quad \forall t>0,
    \end{align*}
    so it follows from integration by parts that, for any $v \in \mathbb{R}^{d}$ and $|v|=1$, 
    \begin{align} \label{eq:lemmaPpr1}
        \abs{\E\left[ \nabla_{v} f(\widehat{Y}_{t})\right]}
        =\abs{ \int_{\mathbb{R}^{d}}\nabla_{v} f(y) p_{\widehat{Y}_t} (y)\dif y }
        =\abs{ \int_{\mathbb{R}^{d}} f(y) \nabla_{v}p_{\widehat{Y}_t} (y)\dif y }
        \le C t^{-\frac{1}{\alpha}} \norm{f}_{\infty}.
    \end{align}
    Let $\widehat{Y}_{t}^{\prime}$ be an independent copy of $\widehat{Y}_{t}$. It is clear that $\widehat{Y}_s + \widehat{Y}_t^{\prime}$ and $\widehat{Y}_{s+t}$ possess the same distribution, so the previous result derives that
    \begin{align*}
        \norm{\nabla \E f (\cdot + \widehat{Y}_{s+t})}_{\op,\infty}
        &= \sup_{x,v \in \mathbb{R}^{d}, |v|=1} \abs{\E \left[\nabla_{v} f (x + \widehat{Y}_{s+t}) \right]}= \sup_{x,v \in \mathbb{R}^{d}, |v|=1} \abs{\E \left[\nabla_{v} f (x + \widehat{Y}_s + \widehat{Y}_t^{\prime}) \right]}.
    \end{align*}
    For any fixed $x \in \R^d$ and $s \geq 0$, denote $\hat{f}_x (y) = \E f(x+\widehat{Y}_{s} + y)$, $y \in \R^d$, then \eqref{eq:lemmaPpr1} shows that
    \begin{align*}
        \abs{\E \left[\nabla_{v} f (x + \widehat{Y}_s + \widehat{Y}_t^{\prime}) \right]}
        = \abs{\E\left[ \nabla_{v} \hat{f}_x (\widehat{Y}_{t}')\right]}
        \leq C t^{-\frac{1}{\alpha}} \norm{\hat{f}_x}_\infty
        = C t^{-\frac{1}{\alpha}} \norm{\E f (\cdot + \widehat{Y}_s)}_{\infty}.
    \end{align*}
    So we have
    \begin{align*}
        \norm{\nabla \E f (\cdot + \widehat{Y}_{s+t})}_{\op,\infty}
        \leq C t^{-\frac{1}{\alpha}} \norm{\E f (\cdot + \widehat{Y}_s)}_{\infty}, \quad \forall s \geq 0, \; t > 0.
    \end{align*}
    For higher order gradients, applying above inequality repeatedly yields that, for $\kappa=1,\dots,4$, 
    \begin{align*}
        \norm{\nabla^{\kappa} \E \left[f(\cdot+\widehat{Y}_{t}) \right]}_{\op,\infty}
        \leq C \left( \frac{t}{\kappa} \right)^{-\frac{1}{\alpha}} \norm{\nabla^{\kappa-1} \E \left[f(\cdot+\widehat{Y}_{\frac{\kappa - 1}{\kappa} t}) \right]}_{\op,\infty} 
        \leq \cdots 
        \leq C \norm{f}_\infty t^{-\frac{\kappa}{\alpha}}.
    \end{align*}
    For the special case, notice that $\sum_{i=1}^m Y_i / n^{1 / \alpha}$ and $\widehat{Y}_{m / n}$ possess the same distribution. At the same time, by (\ref{def:PQ}), we have
    \begin{align*}
        \norm{\nabla^{\kappa} P_{m}f }_{\op,\infty}
        =\norm{\nabla^{\kappa} \E\left[ f\left(\cdot+ \frac{1}{n^{1 / \alpha}} \sum_{i=1}^m Y_i  \right)\right] }_{\op,\infty}
        =\norm{\nabla^{\kappa} \E\left[ f\left(\cdot+\widehat{Y}_{\frac{m}{n}} \right)\right] }_{\op,\infty},
    \end{align*}
    so the proof is finished by taking $t=m / n$.
\end{proof}

With the results of measure decomposition in Lemma \ref{lemma:decom1} and \ref{lemma:decom2}, we have the following gradient estimates of $Q_m f$.

\begin{lemma} \label{lemma:Q1}
    With above definitions, if the common distribution of $X_1, X_2, \dots$ is locally lower bounded by the Lebesgue measure, then there exist $C, C' > 0$ such that for any $f \in \mathcal{C}_{b}^4 (\R^d)$,
    \begin{align*}
        \norm{(Q_m - \whQ_m) f}_\infty \leq \norm{f}_\infty \eup^{- C' m}, \; \;
        \norm{\nabla^{\kappa} \whQ_m f}_{\op,\infty} \leq C \norm{f}_\infty n^{\frac{\kappa}{\alpha}} m^{-\frac{\kappa}{2}}, \; \; \kappa = 0, 1, \dots, 4.
    \end{align*}
\end{lemma}

\begin{proof}
Hoeffding's inequality shows that
\begin{align*}
    \PP \left( \xi < \frac{1 - p}{2} m \right)
    = \PP \left( \sum_{i = 1}^m (\chi_i - p) > \frac{1-p}{2} m \right)
    \leq \eup^{- \frac{(1 - p)^2}{2} m},
\end{align*}
then for any $x \in \R^d$,
\begin{align*}
    \abs{(Q_m - \whQ_m) f (x)}
    = \abs{ \E \left[ f \left( x + \frac{1}{n^{1 / \alpha}\sigma} \sum_{i=1}^{\xi} \whX_i + U \right) \mathbf{1}_{\{\xi < \frac{1-p}{2} m \}} \right] }
    \leq \norm{f}_\infty \eup^{- \frac{(1 - p)^2}{2} m},
\end{align*}
and the first result follows.\par
For the second result, when $\kappa=0$, the definition of $\whQ_{m}f$ shows that $\lVert \whQ_m f \rVert_\infty \leq \norm{f}_\infty$. To estimate $\nabla^{\kappa} \whQ_m f$ for $\kappa = 1, \dots, 4$, we need some preparation. For $l = \left[ \frac{1 - p}{8} m \right] \geq 1$, denote
\begin{align*}
    \olQ_l f (x)
    &= \E f \left( x + \frac{1}{n^{1 / \alpha} \sigma} \sum_{i = 1}^l \whX_i \right) \\
    &= \int_{\R^d} \dots \int_{\R^d} f \left( x + \frac{1}{n^{1 / \alpha} \sigma} \sum_{i = 1}^l z^{(i)} \right) \prod_{i = 1}^l p_{\whX} (z^{(i)}) \, \dif  z^{(1)} \dots \dif  z^{(l)}.
\end{align*}
Recall that $p_{\whX}$ defined in \eqref{eq:pdf_tX} vanishes outside the ball $B (a, \tau)$, so it follows from integration by parts that
\begin{align*}
    \partial_k \olQ_l f (x)
    &= \int_{\R^d} \dots \int_{\R^d} \partial_k f \left( x + \frac{1}{n^{1 / \alpha} \sigma} \sum_{i = 1}^l z^{(i)} \right) \prod_{i = 1}^l p_{\whX} (z^{(i)}) \, \dif  z^{(1)} \dots \dif  z^{(l)} \\
    &= - \sigma n^{\frac{1}{\alpha}} \int_{\R^d} \dots \int_{\R^d} f \left( x + \frac{1}{n^{1 / \alpha} \sigma} \sum_{i = 1}^l z^{(i)} \right) \partial_k p_{\whX} (z^{(j)}) \prod_{i = 1, \, i \ne j}^l p_{\whX} (z^{(i)}) \, \dif  z^{(1)} \dots \dif  z^{(l)} \\
    &= - \sigma n^{\frac{1}{\alpha}} \E \left[ f \left( x + \frac{1}{n^{1 / \alpha} \sigma} \sum_{i = 1}^l \whX_i \right) h_k (\whX_j) \right],
\end{align*}
where $k = 1, \dots, d$, $j = 1, \dots, l$,
\begin{align*}
    h_k (z) = \frac{\partial_k p_{\whX} (z)}{p_{\whX} (z)} \mathbf{1}_{B (a, \tau)} (z)
    = - \frac{2 (z_k - a_k)}{\left( \tau^2 - \abs{z-a}^2 \right)^2} \mathbf{1}_{B (a, \tau)} (z), \quad z = (z_1, \dots, z_d) \in \R^d.
\end{align*}
It follows that
\begin{align*}
    \abs{\partial_k \olQ_l f (x)}
    &= \abs{- \sigma n^{\frac{1}{\alpha}} l^{-1} \E \left[ f \left( x + \frac{1}{n^{1 / \alpha} \sigma} \sum_{i = 1}^l \whX_i \right) \sum_{j = 1}^l h_k (\whX_j) \right] } \\
    &\leq \norm{f}_\infty \sigma n^{\frac{1}{\alpha}} l^{-1} \E \abs{ \sum_{j = 1}^l h_k (\whX_j) }
    \leq \norm{f}_\infty \sigma n^{\frac{1}{\alpha}} l^{-1} \sqrt{ \E \abs{ \sum_{j = 1}^l h_k (\whX_j) }^2 },
\end{align*}
for any $x \in \R^d$. Notice that
\begin{gather*}
    \E h_k (\whX_j)
    = - 2c \int_{B (a, \tau)} \frac{z_k - a_k}{\left( \tau^2 - \abs{z -a}^2 \right)^2} \eup^{- \frac{1}{\tau^2 - \lvert z -a \rvert^2}} \, \dif  z
    = 0, \\
    \E h_k (\whX_j)^2
    = 4c \int_{B (a, \tau)} \frac{(z_k - a_k)^2}{\left( \tau^2 - \abs{z -a}^2 \right)^4} \eup^{- \frac{1}{\tau^2 - \lvert z -a \rvert^2}} \, \dif z
    < + \infty,
\end{gather*}
and $\whX_1, \dots, \whX_l$ are independent, so for any $k = 1, \dots, d$,
\begin{align*}
    \norm{\partial_k \olQ_l f}_\infty
    &\leq \norm{f}_\infty \sigma n^{\frac{1}{\alpha}} l^{-1} \sqrt{ \E \sum_{j = 1}^l h_k (\whX_j)^2 + \E \sum_{i \ne j} h_k (\whX_i) h_k (\whX_j) } \\
    &= \norm{f}_\infty \sigma n^{\frac{1}{\alpha}} l^{-1} \sqrt{ l \E h_k (\whX_1)^2 } \leq C \norm{f}_\infty n^{\frac{1}{\alpha}} l^{-\frac{1}{2}}.
\end{align*}
Combining $l = \left[ \frac{1 - p}{8} m \right]$ we have
\begin{align} \label{eq:lemmaQ1pr1}
    \norm{\nabla \olQ_l f}_{\op,\infty}
    \leq C \norm{f}_\infty n^{\frac{1}{\alpha}} l^{-\frac{1}{2}}
    \leq C \norm{f}_\infty n^{\frac{1}{\alpha}} m^{-\frac{1}{2}}.
\end{align}

Now we are ready to estimate $\nabla^{\kappa} \whQ_m f$, $\kappa = 1, \dots, 4$. Recall $\kappa l = \kappa \left[ \frac{1 - p}{8} m \right] \leq \frac{1-p}{2} m$, for any  $x, v_{1},\dots,v_{\kappa} \in \mathbb{R}^{d}$, $\abs{v_{1}} = \dots = \abs{v_{\kappa}} = 1$,
\begin{align} \label{eq:lemmaQ1pr2}
    \begin{split}
    &\mathrel{\phantom{=}} \nabla_{v_{1}} \dots \nabla_{v_{\kappa}} \whQ_m f (x) \\
    &= \E \left[ \nabla_{v_{1}} \dots \nabla_{v_{\kappa}} f \left( x + \frac{1}{n^{1 / \alpha}\sigma} \sum_{i=1}^{\xi} \whX_i + U \right) \mathbf{1}_{\{\xi \geq \frac{1-p}{2} m \}} \right] \\
    &= \E \left[ \nabla_{v_{1}} \dots \nabla_{v_{\kappa}} f \left( x + \frac{1}{n^{1 / \alpha}\sigma} \sum_{i=1}^{\kappa l} \whX_i + \frac{1}{n^{1 / \alpha}\sigma} \sum_{i = \kappa l + 1}^{\xi} \whX_i + U \right) \mathbf{1}_{\{\xi \geq \frac{1-p}{2} m \}} \right] \\
    &= \E \left[ \olQ_l^{\otimes \kappa} (\nabla_{v_{1}} \dots \nabla_{v_{\kappa}} f) \left( x + \frac{1}{n^{1 / \alpha}\sigma} \sum_{i = \kappa l + 1}^{\xi} \whX_i + U \right) \mathbf{1}_{\{\xi \geq \frac{1-p}{2} m \}} \right],
    \end{split}
\end{align}
where in the last equality we use the independence between $(\xi, U)$ and $\whX_i$, $i = 1, \dots, m$. By the commutativity of $\olQ_l$ and derivative operator, \eqref{eq:lemmaQ1pr1} derives that
\begin{align*}
    \norm{\olQ_l^{\otimes \kappa} ( \nabla_{v_{1}} \dots \nabla_{v_{\kappa}} f )}_{\infty} 
    &= \norm{ \nabla_{v_{1}} \olQ_l \left[ \olQ_l^{\otimes (\kappa - 1)} ( \nabla_{v_{2}} \dots \nabla_{v_{\kappa}} f) \right]}_\infty \\
    &\leq \norm{ \nabla \olQ_l \left[ \olQ_l^{\otimes (\kappa - 1)} ( \nabla_{v_{2}} \dots \nabla_{v_{\kappa}} f) \right]}_{\op,\infty} \abs{v_1} \\
    &\leq C \norm{ \olQ_l^{\otimes (\kappa - 1)} ( \nabla_{v_{2}} \dots \nabla_{v_{\kappa}} f) }_\infty  n^{\frac{1}{\alpha}} m^{-\frac{1}{2}}.
\end{align*}
According to \eqref{eq:lemmaQ1pr2}, applying above inequality repeatedly implies
\begin{multline*}
    \abs{\nabla_{v_{1}} \dots \nabla_{v_{\kappa}} \whQ_m f (x)} 
    \leq \norm{\olQ_l^{\otimes \kappa} ( \nabla_{v_{1}} \dots \nabla_{v_{\kappa}} f )}_{\infty}
    \leq C \norm{ \olQ_l^{\otimes (\kappa - 1)} ( \nabla_{v_{2}} \dots \nabla_{v_{\kappa}} f) }_\infty  n^{\frac{1}{\alpha}} m^{-\frac{1}{2}} \\
    \leq \cdots \leq C \norm{f}_\infty n^{\frac{\kappa}{\alpha}} m^{-\frac{\kappa}{2}},
\end{multline*}
which shows the desired result
\begin{align*}
    \norm{\nabla^{\kappa} \whQ_m f}_{\op,\infty}
    &= \sup_{\substack{x, v_{1},\dots,v_{\kappa} \in \mathbb{R}^{d} \\ \abs{v_{1}} = \dots = \abs{v_{\kappa}} = 1}} \abs{\nabla_{v_{1}} \dots \nabla_{v_{\kappa}} \whQ_m f (x)}
    \leq C \norm{f}_\infty n^{\frac{\kappa}{\alpha}} m^{-\frac{\kappa}{2}}, \quad \kappa = 1, \dots, 4. \qedhere
\end{align*}
\end{proof}


The following lemma shows a better gradient estimate with the assumption \eqref{eq:inhyp} below. To be concrete, the gradient estimate $\lVert \nabla^{\kappa} \whQ_m f \rVert_{\op,\infty} \leq C \norm{f}_\infty n^{\frac{\kappa}{\alpha}} m^{-\frac{\kappa}{2}}$ in Lemma \ref{lemma:Q1} is replaced with $\lVert \nabla^{\kappa} \whP_m f \rVert_{\op,\infty} \leq C \norm{f}_\infty n^{\frac{\kappa}{\alpha}} m^{-\frac{\kappa}{\wtalpha}}$. This lemma paves a way for us to using backward induction on $\alpha$ in the proof of Theorem \ref{theorem1}.

\begin{lemma} \label{lemma:Q2}
    Assume the common distribution of $X_{1}, X_{2}, \dots$ is locally lower bounded by the Lebesgue measure, and satisfies Assumption \ref{A:mX} with parameters $\alpha \in (0, 2)$, $\gamma > 0$. Suppose that $\nu$ satisfies Assumption \ref{A:nu}. If
    \begin{align} \label{eq:inhyp}
        \dtv (\mu_n^{\wtalpha}, S_{\wtalpha} (\nu)) \leq c n^{- \delta}, \quad \forall n \geq 1,
    \end{align}
    holds for some $\wtalpha \in (\alpha, 2)$ and $c, \delta > 0$,   {then} there exists $C > 0$ such that for any $f \in \mathcal{C}_{b}^4 (\R^d)$,
    \begin{align*}
        \norm{(Q_m - \whP_m) f}_\infty \leq C \norm{f}_\infty m^{-\delta}, \; \;
        \norm{\nabla^{\kappa} \whP_m f}_{\op,\infty} \leq C \norm{f}_\infty n^{\frac{\kappa}{\alpha}} m^{-\frac{\kappa}{\wtalpha}}, \; \; \kappa = 0, 1, \dots, 4.
    \end{align*}
\end{lemma}

\begin{proof}
    Since $X_i$ and $(1 - \wtchi_i) \wtX_i + \wtchi_i V_i$ have the same distribution,
\begin{align*}
    Q_{m} f (x)
    &= \E f \left( x + \frac{1}{n^{1 / \alpha}\sigma} \sum_{i=1}^{m} (1 - \wtchi_i) \wtX_i + \frac{1}{n^{1 / \alpha}\sigma} \sum_{i=1}^{m} \left( \wtchi_i V_i - \omega_{n,\alpha} \right) \right).
\end{align*}
Hoeffding's inequality shows that
\begin{align*}
    \PP (\zeta < \frac{1-q}{2} m)
    = \PP \left( \sum_{i = 1}^m (\wtchi_i - q) > \frac{1-q}{2} m \right)
    \leq \eup^{- \frac{(1 - q)^2}{2} m}
    \leq C m^{-\delta},
\end{align*}
which implies that
\begin{align} \label{eq:Qfpart1}
    \abs{ \E \left[ f \left( x + \frac{1}{n^{1 / \alpha}\sigma} \sum_{i=1}^{m} (1 - \wtchi_i) \wtX_i + \frac{1}{n^{1 / \alpha}\sigma} \sum_{i=1}^{m} \left( \wtchi_i V_i - \omega_{n,\alpha} \right) \right) \mathbf{1}_{\{\zeta < \frac{1-q}{2} m \}} \right] }
    \leq C \norm{f}_\infty m^{-\delta}.
\end{align}
On $\{\zeta \geq \frac{1-q}{2} m \}$, define
\begin{align*}
    \wtS_\zeta = \frac{1}{\zeta^{1 / \wtalpha} \wtsigma} \sum_{i = 1}^{\zeta} \left( \wtX_i - \omega_{\zeta,\wtalpha} \right).
\end{align*}
Because $\wtX_i$, $i = 1, \dots, m$ are independent identically distributed and independent of $\wtchi_i$, $V_i$, $i = 1, \dots, m$, we have
\begin{align} \label{eq:Qfpart2}
    \begin{split}
    &\mathrel{\phantom{=}} \E \left[ f \left( x + \frac{1}{n^{1 / \alpha}\sigma} \sum_{i=1}^{m} (1 - \wtchi_i) \wtX_i + \frac{1}{n^{1 / \alpha}\sigma} \sum_{i=1}^{m} \left( \wtchi_i V_i - \omega_{n,\alpha} \right) \right) \mathbf{1}_{\{\zeta \geq \frac{1-q}{2} m \}} \right] \\
    &= \E \left[ f \left( x + \frac{1}{n^{1 / \alpha}\sigma} \sum_{i=1}^{\zeta} \wtX_i + \frac{1}{n^{1 / \alpha}\sigma} \sum_{i=1}^{m} \left( \wtchi_i V_i - \omega_{n,\alpha} \right) \right) \mathbf{1}_{\{\zeta \geq \frac{1-q}{2} m \}} \right] \\
    &= \E \left[ f \left( x + \frac{\zeta^{1 / \wtalpha} \wtsigma}{n^{1 / \alpha} \sigma} \wtS_\zeta + V \right) \mathbf{1}_{\{\zeta \geq \frac{1-q}{2} m \}} \right].
    \end{split}
\end{align}
Combining \eqref{eq:Qfpart1}, \eqref{eq:Qfpart2}, and the definitions of $Q_m$, $\whP_m$, we have
\begin{align*}
    &\mathrel{\phantom{=}} \abs{(Q_m - \whP_m) f (x)} \\
    &\leq \abs{ \E \left\{ \left[ f \left( x + \frac{\zeta^{1 / \wtalpha} \wtsigma}{n^{1 / \alpha} \sigma} \wtS_\zeta + V \right) - f \left( x + \frac{\zeta^{1 / \wtalpha} \wtsigma}{n^{1 / \alpha} \sigma} \wtY + V \right) \right] \mathbf{1}_{\{\zeta \geq \frac{1-q}{2} m \}} \right\} } + C \norm{f}_\infty m^{-\delta} \\
    &\leq \E \abs{ \E \left[ f \left( x + \frac{\zeta^{1 / \wtalpha} \wtsigma}{n^{1 / \alpha} \sigma} \wtS_\zeta + V \right) - f \left( x + \frac{\zeta^{1 / \wtalpha} \wtsigma}{n^{1 / \alpha} \sigma} \wtY + V \right) \middle\vert (\zeta, V) \right] \mathbf{1}_{\{\zeta \geq \frac{1-q}{2} m \}} } + C \norm{f}_\infty m^{-\delta},
\end{align*}
Since $\wtX_i$, $i = 1, \dots, m$ are independent of $(\zeta, V)$, \eqref{eq:inhyp} shows that on $\{\zeta \geq \frac{1-q}{2} m \}$,
\begin{align*}
    \abs{ \E \left[ f \left( x + \frac{\zeta^{1 / \wtalpha} \wtsigma}{n^{1 / \alpha} \sigma} \wtS_\zeta + V \right) - f \left( x + \frac{\zeta^{1 / \wtalpha} \wtsigma}{n^{1 / \alpha} \sigma} \wtY + V \right) \middle\vert (\zeta, V) \right] }
    \leq C \norm{f}_\infty \zeta^{-\delta}, \text{ a.s.},
\end{align*}
so
\begin{align*}
    \abs{(Q_m - \whP_m) f (x)}
    \leq C \norm{f}_\infty \E \left[ \zeta^{-\delta} \mathbf{1}_{\{\zeta \geq \frac{1-q}{2} m \}} \right] + C \norm{f}_\infty m^{-\delta}
    \leq C \norm{f}_\infty m^{-\delta}.
\end{align*}
On the other hand, for $\kappa = 0, 1, \dots, 4$, Lemma \ref{lemma:P} derives that on $\{\zeta \geq \frac{1-q}{2} m \}$,
\begin{align*}
    \norm{ \E \left[ \nabla^{\kappa} f \left( \cdot + \frac{\zeta^{1 / \wtalpha} \wtsigma}{n^{1 / \alpha} \sigma} \wtY + V \right) \middle\vert (\zeta, V) \right] }_{\op,\infty}
    \leq C \norm{f}_\infty n^{\frac{\kappa}{\alpha}} \zeta^{-\frac{\kappa}{\wtalpha}}, \text{ a.s.}.
\end{align*}
As $(\zeta, V)$ and $\wtY$ are independent, we have
\begin{align*}
    \norm{\nabla^{\kappa} \whP_m f }_{\op,\infty}
    &=\norm{\nabla^{\kappa} \E \left[ f \left( \cdot + \frac{\zeta^{1 / \wtalpha} \wtsigma}{n^{1 / \alpha} \sigma} \wtY + V \right) \mathbf{1}_{\{\zeta \geq \frac{1-q}{2} m \}} \right]}_{\op,\infty}\\
    &\leq \E \abs{\norm{ \E \left[ \nabla^{\kappa} f \left( \cdot + \frac{\zeta^{1 / \wtalpha} \wtsigma}{n^{1 / \alpha} \sigma} \wtY + V \right) \middle\vert (\zeta, V) \right]}_{\op,\infty} \mathbf{1}_{\{\zeta \geq \frac{1-q}{2} m \}} } \\
    &\leq C \norm{f}_\infty n^{\frac{\kappa}{\alpha}} \E \left[ \zeta^{-\frac{\kappa}{\wtalpha}} \mathbf{1}_{\{\zeta \geq \frac{1-q}{2} m \}} \right] \\
    &\leq C \norm{f}_\infty n^{\frac{\kappa}{\alpha}} m^{-\frac{\kappa}{\wtalpha}}. \qedhere
\end{align*}
\end{proof}

\section{One step estimates}

In this section, we first estimate one step error $(Q_1 - P_1) f$ in Lemma \ref{lemma:one step}. Combining with the gradient estimates derived in   {the previous section}, Lemma \ref{lemma4} and \ref{lemma5} estimate $( Q_1 - P_1 ) P_k f$ and $( Q_1 - P_1 ) Q_k f$ respectively.

\begin{lemma}\label{lemma:one step}
    Suppose that the common distribution of $X_1, X_2, \dots$ satisfies Assumption \ref{A:mX} with parameters $\alpha \in (0, 2)$, $\gamma > 0$. There exists a constant $C$ such that $\forall n \geq 1$, $f \in \mathcal{C}_b^4 ( \mathbb{R}^d )$,
    
    \noindent (i) For $\alpha\in (1,2)$,
    \begin{align*}
        \norm{(Q_1 - P_1) f}_\infty
        \leq C \left( \norm{\nabla f}_{\op,\infty} + \norm{\nabla^{2} f}_{\op,\infty} \right) \left[ n^{- \frac{2}{\alpha}} + n^{- 1 -\frac{\gamma}{\alpha}} (\ln n)^{\mathbf{1}_{ \{ \gamma = 2 - \alpha \} }} \right].
    \end{align*}
    
    \noindent (ii) For $\alpha=1$ and $\nu$ satisfying $\int_{\mathbb{S}^{d-1}} \theta \nu ( \dif \theta ) = \mathbf{0}$,
    \begin{align*}
        \norm{(Q_1 - P_1) f}_\infty \leqslant C \left( \sum_{\kappa = 0}^4 \norm{\nabla^\kappa f}_{\op,\infty} \right) \left[ n^{-2} + n^{- 1 - \gamma} (\ln n)^{\mathbf{1}_{ \{ \gamma = 1 \} }} \right].
    \end{align*}

    \noindent (iii) For $\alpha\in (0,1)$,
    \begin{align*}
        \norm{(Q_1 - P_1) f}_\infty
        &\leq C \left( \sum_{\kappa = 0}^2 \norm{\nabla^\kappa f}_{\op,\infty} \right) \left[ n^{-2} + n^{- 1 -\frac{\gamma}{\alpha \lor (1 - \gamma)}} (\ln n)^{\mathbf{1}_{ \{ \gamma = 1 - \alpha \} }} \right] \\
        &\quad + C \norm{\nabla f}_{\op,\infty} n^{-\frac{1}{\alpha}} (\ln n)^{\mathbf{1}_{ \{ \gamma = 1 - \alpha \} }}.
    \end{align*}
    If the distribution of $X_1$ is symmetric in the sense that $\int_{\mathbb{S}^{d-1}} \theta \epsilon (r, \theta) \nu (\dif \theta) = \mathbf{0}$, $\forall r > 0$, we have
    \begin{align*}
        \norm{(Q_1 - P_1) f}_\infty
        \leq C \left( \sum_{\kappa = 0}^2 \norm{\nabla^\kappa f}_{\op,\infty} \right) \left[ n^{-2} + n^{- 1 -\frac{\gamma}{\alpha \lor (1 - \gamma)}} (\ln n)^{\mathbf{1}_{ \{ \gamma = 1 - \alpha \} }} \right].
    \end{align*}
\end{lemma}

The proof of   {the} above lemma is postponed to Appendix. For simplicity, we denote
\begin{align} \label{corollary:one step}
    D (n,\alpha,\gamma, f)
    = \begin{cases}
    \left( \sum_{\kappa = 1}^2 \norm{\nabla^\kappa f}_{\op,\infty} \right) \left[ n^{- \frac{2}{\alpha}} + n^{- 1 -\frac{\gamma}{\alpha}} (\ln n)^{\mathbf{1}_{ \{ \gamma = 2 - \alpha \} }} \right], &\alpha \in (1, 2); \\
    \left( \sum_{\kappa = 0}^4 \norm{\nabla^\kappa f}_{\op,\infty} \right) \left[ n^{-2} + n^{- 1 - \gamma} (\ln n)^{\mathbf{1}_{ \{ \gamma = 1 \} }} \right], &\alpha = 1; \\
    \left( \sum_{\kappa = 0}^2 \norm{\nabla^\kappa f}_{\op,\infty} \right) \left[ n^{-2} + n^{- 1 -\frac{\gamma}{\alpha \lor (1 - \gamma)}} (\ln n)^{\mathbf{1}_{ \{ \gamma = 1 - \alpha \} }} \right] \\
    \phantom{XXXXXXXXXX} + \norm{\nabla f}_{\op,\infty} n^{-\frac{1}{\alpha}} (\ln n)^{\mathbf{1}_{ \{ \gamma = 1 - \alpha \} }}, &\alpha \in (0, 1).
    \end{cases}
\end{align}
Then Lemma \ref{lemma:one step} shows that
\begin{align*}
    \norm{(Q_1 - P_1) f}_\infty \leq C D (n,\alpha,\gamma, f), \quad \forall n \geq 1, \, f \in \mathcal{C}_b^4 ( \mathbb{R}^d ).
\end{align*}

\begin{lemma} \label{lemma4}
    Suppose that the common distribution of $X_1, X_2, \dots$ satisfies Assumption \ref{A:mX} with parameters $\alpha \in (0, 2)$, $\gamma > 0$, and $\nu$ satisfies Assumption \ref{A:nu}. Then there exists a constant $C > 0$ not depending on $n$ and $k$, such that for any $f \in \mathcal{C}_{b}^{4}(\R^d)$,
    \begin{align*}
        \norm{( Q_1 - P_1 ) P_k f}_\infty \leq C \norm{f}_\infty \Lambda_1 (k, n, \alpha, \gamma), \quad 1 \leq k\leq n,
    \end{align*}
    where 
    \begin{align*}
        \Lambda_1 (k, n, \alpha, \gamma) = \begin{cases}
            k^{-\frac{2}{\alpha}} n^{\frac{2}{\alpha} - \frac{(\alpha + \gamma) \land 2}{\alpha}} \ln n, &\alpha \in (1, 2); \\
            k^{-4} n^{3 - (\gamma \land 1)} \ln n, &\alpha = 1;\\
            k^{-\frac{2}{\alpha}} n^{\frac{2}{\alpha} - 1 - (\gamma \land 1)} \ln n + k^{-\frac{1}{\alpha}} \ln n, &\alpha \in (0, 1).
        \end{cases}
    \end{align*}
\end{lemma}

\begin{proof}
    According to Lemma \ref{lemma:one step},
    \begin{align*}
        \norm{( Q_1 - P_1 ) P_k f}_\infty
        \leq C D (n, \alpha, \gamma, P_k f).
    \end{align*}
    By Lemma \ref{lemma:P}, we have
    \begin{align*}
        \norm{\nabla^{\kappa} P_{k}f}_{\op,\infty}\le C\norm{f}_{\infty} n^{\frac{\kappa}{\alpha}} k^{-\frac{\kappa}{\alpha}}, \quad \kappa= 0, 1, \dots, 4.
    \end{align*}
    So straightforward calculations derive that
    \begin{align*}
        D (n, \alpha, \gamma, P_k f)
        \leq C \norm{f}_{\infty} \Lambda_1 (k, n, \alpha, \gamma), \quad 1 \leq k \leq n,
    \end{align*}
    which completes the proof.
\end{proof}

\begin{lemma} \label{lemma5-0}
     Assume the common distribution of $X_1, X_2, \dots$ is locally lower bounded by the Lebesgue measure, and satisfies Assumption \ref{A:mX} with $\alpha \in (0, 2)$, $\gamma > 0$. Suppose that $\nu$ satisfies Assumption \ref{A:nu}. Then there exists a constant $C > 0$ not depending on $n$ and $k$, such that for any $f \in \mathcal{C}_{b}^{4}(\R^d)$,
    \begin{align*}
        \norm{( Q_1 - P_1 ) Q_k f}_\infty \leq C \norm{f}_\infty \Lambda_2 (n, \alpha, 2, \gamma), \quad \frac{n}{4}\leq k\leq n,
    \end{align*}
    where
    \begin{align} \label{eq:Lambda_2}
        \Lambda_2 (n, \alpha, \beta, \gamma) = \begin{cases}
            n^{\frac{2}{\alpha} - \frac{2}{\beta} - \frac{(\alpha + \gamma) \land 2}{\alpha}} \ln n, &\alpha \in (1, 2); \\
            n^{3 - \frac{4}{\beta} - (\gamma \land 1)} \ln n, &\alpha = 1; \\
            n^{\frac{2}{\alpha} - \frac{2}{\beta} - 1 - (\gamma \land 1)} \ln n + n^{- \frac{1}{\beta}} \ln n, &\alpha \in (0, 1).
        \end{cases}
    \end{align}
\end{lemma}
\begin{proof}
For any $\frac{n}{4} \leq k \leq n$, Lemma \ref{lemma:Q1} shows that there exists an operator $\whQ_k$ such that
\begin{gather}
        \norm{(Q_k - \whQ_k) f}_\infty \leq \norm{f}_\infty \eup^{- c k}, \label{eq:QQ}\\
        \norm{\nabla^{\kappa} \whQ_k f}_{\op,\infty} \leq C \norm{f}_\infty n^{\frac{\kappa}{\alpha}} k^{-\frac{\kappa}{2}} \leq C \norm{f}_\infty n^{\frac{\kappa}{\alpha} - \frac{\kappa}{2}}, \label{eq:gradQ}
\end{gather}
for $\kappa = 0, 1, \dots, 4$. Combining Lemma \ref{lemma:one step}, (\ref{eq:QQ}) and the commutativity of $P_1$, $Q_k$ and $\whQ_k$ gives us
\begin{align} \label{eq:PQ1}
    \begin{split}
    &\mathrel{\phantom{=}} \norm{(Q_1 - P_1) Q_k f}_\infty \\
    &\leq \norm{(Q_1 - P_1) (Q_k - \whQ_k)  f}_\infty + \norm{(Q_1 - P_1) \whQ_k f}_\infty \\
    &\leq \norm{(Q_k - \whQ_k) (Q_1 - P_1) f}_\infty + C D (n, \alpha, \gamma, \whQ_k f) \\
    &\leq \norm{(Q_1 - P_1) f}_\infty \eup^{-c k} + C D (n, \alpha, \gamma, \whQ_k f).
    \end{split}
\end{align}
By definitions of $P_1$ and $Q_1$, it is clear that
\begin{align} \label{eq:PQ2}
    \norm{(Q_1 - P_1) f}_\infty \eup^{-c k}
    \leq 2 \norm{f}_\infty \eup^{-c k}
    \leq 2 \norm{f}_\infty \eup^{- \frac{c}{4} n}, \quad \frac{n}{4} \leq k \leq n.
\end{align}
On the other hand, (\ref{eq:gradQ}) implies that
\begin{align} \label{eq:PQ3}
    D (n, \alpha, \gamma, \whQ_k f) \leq C \norm{f}_\infty \Lambda_2 (n, \alpha, 2, \gamma), \quad \frac{n}{4} \leq k \leq n.
\end{align}
Combining \eqref{eq:PQ1}, \eqref{eq:PQ2} and \eqref{eq:PQ3} derives that for $\frac{n}{4} \leq k \leq n$,
\begin{align*}
    \norm{(Q_1 - P_1) Q_k f}_\infty
    &\leq 2 \norm{f}_\infty \eup^{- \frac{c}{4} n} + C \norm{f}_\infty \Lambda_2 (n, \alpha, 2, \gamma)
    \leq C \norm{f}_\infty \Lambda_2 (n, \alpha, 2, \gamma). \qedhere
\end{align*} 
\end{proof} 

\begin{lemma}  \label{lemma5}
Assume the same conditions as in Lemma \ref{lemma5-0}. 
    Furthermore, if
    \begin{align} \label{eq:inhyp0}
        \dtv (\mu_n^{\wtalpha}, S_{\wtalpha} (\nu)) \leq c n^{- \delta}, \quad \forall n \geq 1,
    \end{align}
    holds for some $\wtalpha \in (\alpha, 2)$ and $c, \delta > 0$,   {then} there exists a constant $C>0$ not depending on $n$ and $k$, such that for any $f \in \mathcal{C}_{b}^{4}(\R^d)$,
    \begin{align*}
        \norm{( Q_1 - P_1 ) Q_k f}_\infty \leq C \norm{f}_\infty \Lambda_2 (n, \alpha, \wtalpha, \gamma), \quad \frac{n}{2} \leq k \leq n, 
    \end{align*}
 where $ \Lambda_2$ is defined in \eqref{eq:Lambda_2}. 
\end{lemma}

\begin{proof}
We apply Lemma \ref{lemma:Q2} to derive a better estimate. We denote $m = \left[ \frac{n}{4r} \right] \geq 1$ with a constant $r = \left[ \frac{1}{\delta} (\frac{4}{\alpha} - 2) \right] + 1$ not depending on $n$ and $k$. Notice that for $n < 4r$ we can choose $C$ large enough.

According to Lemma \ref{lemma:Q2}, there exists an operator $\whP_m$ such that
\begin{gather}
    \norm{(Q_m - \whP_m) f}_\infty \leq C \norm{f}_\infty m^{-\delta} \leq C \norm{f}_\infty n^{-\delta}, \label{lemma5pr1} \\
    \norm{\nabla^{\kappa} \whP_m f}_{\op,\infty} \leq C \norm{f}_\infty n^{\frac{\kappa}{\alpha}} m^{-\frac{\kappa}{\wtalpha}} \leq C \norm{f}_\infty n^{\frac{\kappa}{\alpha} - \frac{\kappa}{\wtalpha}}, \label{lemma5pr2}
\end{gather}
holds for $\kappa = 0, 1, \dots, 4$, where we use the facts that $m \geq \frac{n}{4r} - 1$ and $r$ does not depend on $n$. Applying $Q_m = (Q_m - \whP_m) + \whP_m$ repeatedly implies that
\begin{align} \label{lemma5pr3}
    \begin{split}
    &\mathrel{\phantom{=}} (Q_1 - P_1) Q_k f \\
    &= (Q_1 - P_1) (Q_m - \whP_m) Q_{k - m} f + (Q_1 - P_1) \whP_m Q_{k - m} f \\
    &= (Q_1 - P_1) (Q_m - \whP_m)^{\otimes 2} Q_{k - 2 m} f \\
    &\phantom{X} + (Q_1 - P_1) (Q_m - \whP_m) \whP_m Q_{k - 2 m} f + (Q_1 - P_1) \whP_m Q_{k - m} f \\
    &= (Q_1 - P_1) (Q_m - \whP_m)^{\otimes r} Q_{k - rm} f + \sum_{j=0}^{r-1} (Q_1 - P_1) (Q_m - \whP_m)^{\otimes j} \whP_m Q_{k - (j+1) m} f \\
    &= (Q_m - \whP_m)^{\otimes r} (Q_1 - P_1) Q_{k - rm} f + \sum_{j=0}^{r-1} (Q_m - \whP_m)^{\otimes j} Q_{k - (j+1) m} (Q_1 - P_1) \whP_m f,
    \end{split}
\end{align}
where the last equality follows from the commutativity of $P_1$, $\whP_m$ and $Q_l$, $l = 1, \dots, k$. Because $k - rm \geq \frac{n}{2} - r \left[ \frac{n}{4r} \right] \geq \frac{n}{4}$, it is shown in Lemma \ref{lemma5-0} that
\begin{align*}
    \norm{(Q_1 - P_1) Q_{k-rm} f}_\infty \leq C \norm{f}_\infty \Lambda_2 (n, \alpha, 2, \gamma).
\end{align*}
This, together with applying \eqref{lemma5pr1} repeatedly, derives that
\begin{align} \label{lemma5pr4}
    \begin{split}
    \norm{(Q_m - \whP_m)^{\otimes r} (Q_1 - P_1) Q_{k - rm} f}_\infty
    &\leq C n^{-r\delta} \norm{(Q_1 - P_1) Q_{k-rm} f}_\infty \\
    &\leq C \norm{f}_\infty n^{2 - \frac{4}{\alpha}} \Lambda_2 (n, \alpha, 2, \gamma),
    \end{split}
\end{align}
where in the last inequality we use $r\delta \geq \frac{4}{\alpha} - 2$. For $j = 0, 1, \dots, r-1$, Lemma \ref{lemma:one step} shows
\begin{align} \label{lemma5pr5}
    \norm{(Q_m - \whP_m)^{\otimes j} Q_{k - (j+1) m} (Q_1 - P_1) \whP_m f}_\infty
    \leq 2^j \norm{(Q_1 - P_1) \whP_m f}_\infty
    \leq C D(n, \alpha, \gamma, \whP_m f),
\end{align}
where $D (n,\alpha,\gamma, f )$ is defined in \eqref{corollary:one step}. It can be derived by \eqref{lemma5pr2} that
\begin{align} \label{lemma5pr6}
    D(n, \alpha, \gamma, \whP_m f) \leq C \norm{f}_\infty \Lambda_2 (n, \alpha, \wtalpha, \gamma).
\end{align}

Combining \eqref{lemma5pr3}, \eqref{lemma5pr4}, \eqref{lemma5pr5} and \eqref{lemma5pr6}, we have
\begin{align*}
    \norm{(Q_1 - P_1) Q_k f}_\infty 
    &\leq \norm{(Q_m - \whP_m)^{\otimes r} (Q_1 - P_1) Q_{k - rm} f}_\infty \\
    &\phantom{X} + \sum_{j=0}^{r-1} \norm{(Q_m - \whP_m)^{\otimes j} Q_{k - (j+1) m} (Q_1 - P_1) \whP_m f}_\infty \\
    &\leq C \norm{f}_\infty n^{2 - \frac{4}{\alpha}} \Lambda_2 (n, \alpha, 2, \gamma) + \sum_{j=0}^{r-1} C \norm{f}_\infty \Lambda_2 (n, \alpha, \wtalpha, \gamma).
\end{align*}
Notice that
\begin{align*}
    \Lambda_2 (n, \alpha, 2, \gamma) \leq n^{4 (\frac{1}{\wtalpha} - \frac{1}{2})} \Lambda_2 (n, \alpha, \wtalpha, \gamma),
\end{align*}
so
\begin{align*}
    n^{2 - \frac{4}{\alpha}} \Lambda_2 (n, \alpha, 2, \gamma) \leq n^{4 (\frac{1}{\wtalpha} - \frac{1}{\alpha})} \Lambda_2 (n, \alpha, \wtalpha, \gamma) \leq \Lambda_2 (n, \alpha, \wtalpha, \gamma).
\end{align*}
Since $r$ does not depend on $n$ and $k$, we get
\begin{align*}
    \norm{(Q_1 - P_1) Q_k f}_\infty &\leq C \norm{f}_\infty \Lambda_2 (n, \alpha, \wtalpha, \gamma). \qedhere
\end{align*}
\end{proof}

\section{Proof of Main Results} \label{sec3}

We denote $\Lambda (n, \alpha, \gamma)$ as the desired convergence rate in the main result, i.e.
\begin{align*}
    \Lambda (n, \alpha, \gamma) = \begin{cases}
        n^{- \frac{2 - \alpha}{\alpha}} + n^{-\frac{\gamma}{\alpha}} (\ln n)^{\mathbf{1}_{ \{ \gamma = 2 - \alpha \} }}, &\alpha \in (1, 2); \\
        n^{-1} + n^{- \gamma} (\ln n)^{\mathbf{1}_{ \{ \gamma = 1 \} }}, &\alpha = 1; \\
        n^{-1} + n^{-\rho_{\alpha, \gamma}} (\ln n)^{\mathbf{1}_{ \{ \gamma = 1 - \alpha \} }}, &\alpha \in (0, 1),
    \end{cases}
\end{align*}
where
\begin{align*}
    \rho_{\alpha, \gamma} = \begin{cases}
        \frac{1 - \alpha}{\alpha}, &\gamma \in (1 - \alpha, +\infty); \\
        \frac{\gamma}{1 - \gamma}, &\gamma \in (0, 1 - \alpha].
    \end{cases}
\end{align*}
Recall
\begin{align*}
    Q_m f (x) = \E f \left( x + \frac{1}{n^{1 / \alpha}\sigma} \sum_{i=1}^{m} \left( X_{i} - \omega_{n,\alpha} \right) \right), \qquad
    P_m f (x) = \E f \left( x + \frac{1}{n^{1 / \alpha}} \sum_{i=1}^{m} Y_i \right),
\end{align*}
where $f \in \mcl B_b(\R^d)$, we shall frequently use the following decomposition in this section: 
\begin{equation} \label{e:OneStepDec}
    \E f(S_{n}) - \E f(Y_{1})
    = Q_n f (\mathbf{0}) - P_n f (\mathbf{0})
    = \sum_{k = 0}^{n-1} Q_{n-k-1} (Q_1 - P_1) P_k f (\mathbf{0}).
\end{equation}

Lemma \ref{lemma:one step} shows that
\begin{align*}
    \norm{(Q_1 - P_1) f}_\infty \leq C \left( \sum_{\kappa = 0}^4 \norm{\nabla^{\kappa} f}_{\op,\infty} \right) n^{-1} \Lambda (n, \alpha, \gamma), \quad \forall n \geq 1, \; f \in \mathcal{C}_{b}^{4}(\R^d),
\end{align*}
in which $\rho_{\alpha, \gamma}$ can be replaced by $\frac{\gamma}{\alpha \lor (1 - \gamma)}$ if $X_1$ has a symmetric distribution. In order to prove the main result $\dtv (\mu_n^{\alpha}, S_\alpha (\nu) ) \leq C \Lambda (n, \alpha, \gamma)$, the following lemma shows that it suffices to prove $\dtv (\mu_n^{\alpha}, S_\alpha (\nu) ) \leq c n^{-\delta}$ for any $n \geq 1$ and some $c, \delta > 0$.

\begin{lemma} \label{lemma6}
    Assume the common distribution of $X_1, X_2, \dots$ is locally lower bounded by the Lebesgue measure, and satisfies Assumption \ref{A:mX} with $\alpha \in (0, 2)$, $\gamma > 0$. Suppose that $\nu$ satisfies Assumption \ref{A:nu} and
    \begin{align} \label{eq:lemma6hyp}
        \abs{\E f(S_{n}) - \E f(Y_1)} \leq c \norm{f}_\infty n^{-\delta}, \quad \forall n \geq 1, \; f \in \mathcal{C}_{b}^{4}(\R^d),
    \end{align}
    holds for some $c, \delta > 0$. Then there exists a constant $C > 0$ such that
    \begin{align} \label{eq:lemma6con1}
        \dtv (\mu_n^{\alpha}, S_\alpha (\nu) ) \leq C \Lambda (n, \alpha, \gamma), \quad \forall n \geq 1.
    \end{align}
    If the distribution of $X_1$ is symmetric in the sense that $\int_{\mathbb{S}^{d-1}} \theta \epsilon (r, \theta) \nu (\dif \theta) = \mathbf{0}$, $\forall r > 0$, then for $\alpha \in (0, 1)$,
    \begin{align} \label{eq:lemma6con2}
        \dtv (\mu_n^\alpha, S_\alpha (\nu)) \leqslant C \left[ n^{-1} + n^{-\frac{\gamma}{\alpha \lor (1 - \gamma)}} (\ln n)^{\mathbf{1}_{ \{ \gamma = 1 - \alpha \} }} \right], \quad \forall n \geq 1.
    \end{align}
\end{lemma}

\begin{proof}
We only prove the general conclusion \eqref{eq:lemma6con1}. On condition that $X_1$ has a symmetric distribution, \eqref{eq:lemma6con2} follows from a completely analogous proof with $\rho_{\alpha, \gamma}$ in $\Lambda (n, \alpha, \gamma)$ replaced by $\frac{\gamma}{\alpha \lor (1 - \gamma)}$.

By Kantorovich-Rubinstein theorem \cite{Villani2003}, \eqref{eq:lemma6con1} is equivalent to
\begin{align*}
    \abs{\E f(S_{n}) - \E f(Y_1)} \leq C \norm{f}_\infty \Lambda (n, \alpha, \gamma), \quad \forall n \geq 1, \; f \in \mathcal{B}_{b} (\R^d).
\end{align*}
According to \cite[Theorem 7.10, 8.14]{Folland1999}, any $f \in \mathcal{B}_{b} (\R^d)$ can be approximated almost everywhere by a sequence of $f_i \in \mathcal{C}_{b}^{4}(\R^d)$ satisfying $\norm{f_i}_\infty \leq 2 \norm{f}_\infty$, so it is enough to show
\begin{align*}
    \abs{\E f(S_{n}) - \E f(Y_1)} \leq C \norm{f}_\infty \Lambda (n, \alpha, \gamma), \quad \forall n \geq 1, \; f \in \mathcal{C}_{b}^{4}(\R^d).
\end{align*}
Recall \eqref{e:OneStepDec}, to prove the lemma, we need to bound $Q_{n-k-1} (Q_1 - P_1) P_k f$ according to $k \ge m$ and $k<m$ with $m = \left[ \frac{n}{4r} \right] \geq 1$ and $r = \left[ \frac{1}{\delta} (\frac{4}{\alpha} - 2) \right] + 1$. 

\noindent (i) For $k \ge m$, Lemma \ref{lemma:one step} shows that
\begin{align*}
    \norm{(Q_1 - P_1) P_k f}_\infty \leq C \left( \sum_{\kappa = 0}^4 \norm{\nabla^{\kappa} P_k f}_{\op,\infty} \right) n^{-1} \Lambda (n, \alpha, \gamma).
\end{align*}
According to Lemma \ref{lemma:P},
\begin{align*}
    \norm{\nabla^{\kappa} P_k f}_{\op,\infty}
    \leq C \norm{f}_\infty n^{\frac{\kappa}{\alpha}} k^{-\frac{\kappa}{\alpha}}
    \leq C \norm{f}_\infty n^{\frac{\kappa}{\alpha}} \left[ \frac{n}{4r} \right]^{- \frac{\kappa}{\alpha}}
    \leq C \norm{f}_\infty,
\end{align*}
so we get
\begin{align} \label{eq:lemma6pr1}
    \norm{Q_{n-k-1} (Q_1 - P_1) P_k f}_\infty
    \leq \norm{(Q_1 - P_1) P_k f}_\infty \leq C \norm{f}_\infty n^{-1} \Lambda (n, \alpha, \gamma), \quad k \geq m.
\end{align}

\noindent (ii) For $k < m$, we shall use the assumption \eqref{eq:lemma6hyp} and need some preparations. Denote 
\begin{equation*}
\hP_m f (x) = P_m f \left( x + \frac{\omega_{m,\alpha} - \omega_{n,\alpha}}{n^{1 / \alpha}\sigma} \right),
\end{equation*}  
then we have
\begin{align} \label{eq:lemma6pr7}
    \norm{\hP_m (Q_1 - P_1) f}_\infty
    = \norm{P_m (Q_1 - P_1) f}_\infty
    = \norm{(Q_1 - P_1) P_m f}_\infty
    \leq C \norm{f}_\infty n^{-1} \Lambda (n, \alpha, \gamma).
\end{align}
Furthermore, for any fixed $x \in \R^d$,
\begin{align*}
    &\mathrel{\phantom{=}} \abs{(Q_m - \hP_m) f (x)} \\
    &= \abs{\E f \left( x + \frac{1}{n^{1 / \alpha}\sigma} \sum_{i=1}^m \left( X_{i} - \omega_{n,\alpha} \right) \right) - \E f \left( x + \frac{\omega_{m,\alpha} - \omega_{n,\alpha}}{n^{1 / \alpha}\sigma} + \frac{1}{n^{1 / \alpha}} \sum_{i=1}^m Y_i \right)} \\
    &= \abs{\E f \left( x + \frac{\omega_{m,\alpha} - \omega_{n,\alpha}}{n^{1 / \alpha}\sigma} + \frac{m^{1 / \alpha}}{n^{1 / \alpha}} S_m \right) - \E f \left( x + \frac{\omega_{m,\alpha} - \omega_{n,\alpha}}{n^{1 / \alpha}\sigma} + \frac{m^{1 / \alpha}}{n^{1 / \alpha}} Y_1 \right)} \\
    &= \abs{\E \bar{f}_x (S_m) - \E \bar{f}_x (Y_1)},
\end{align*}
where $\bar{f}_x (y) = f ( x + \frac{\omega_{m,\alpha} - \omega_{n,\alpha}}{n^{1 / \alpha}\sigma} + \frac{m^{1 / \alpha}}{n^{1 / \alpha}} y )$. According to \eqref{eq:lemma6hyp},
\begin{align*}
    \abs{ \E \bar{f}_x (S_m) - \E \bar{f}_x (Y_1) }
    \leq C \norm{\bar{f}_x}_\infty m^{-\delta}
    = C \norm{f}_\infty m^{-\delta}
    \leq C \norm{f}_\infty n^{-\delta},
\end{align*}
so it follows that
\begin{align} \label{eq:lemma6pr3}
    \norm{(Q_m - \hP_m) f}_\infty
    = \sup_{x \in \R^d} \abs{(Q_m - \hP_m) f (x)}
    \leq C \norm{f}_\infty n^{-\delta}.
\end{align}

Now we are ready to estimate $Q_{n-k-1} (Q_1 - P_1) P_k f$ for $k < m$. Analogous to the operator decomposition in \eqref{lemma5pr3}, applying $Q_m = (Q_m - \hP_m) + \hP_m$ repeatedly implies
\begin{align} \label{eq:lemma6pr4}
    \begin{split}
    Q_{n-k-1} (Q_1 - P_1) P_k f 
    &= (Q_m - \hP_m)^{\otimes r} Q_{n-k-1-rm} (Q_1 - P_1) P_k f \\
    &\quad + \sum_{j=0}^{r-1} (Q_m - \hP_m)^{\otimes j} \hP_m Q_{n-k-1-(j+1)m} (Q_1 - P_1) P_k f.
    \end{split}
\end{align}
Notice that $n - k - 1 - rm \geq \frac{n}{4} \geq m$, by \eqref{eq:lemma6pr3} and the following estimation (to be proved below)
\begin{align} \label{eq:lemma6pr2}
    \begin{split}
    &\mathrel{\phantom{=}} \norm{Q_l (Q_1 - P_1) f}_\infty \leq C \norm{f}_\infty n^{\frac{4}{\alpha} - 3} \Lambda (n, \alpha, \gamma), \ \ \ m \le l \le n,
    \end{split}
\end{align}
 we have
\begin{align} \label{eq:lemma6pr5}
    \begin{split}
    &\mathrel{\phantom{=}} \norm{(Q_m - \hP_m)^{\otimes r} Q_{n- rm -k-1} (Q_1 - P_1) P_k f}_\infty \\
    &\leq C \norm{Q_{n- rm -k-1} (Q_1 - P_1) P_k f}_\infty n^{-r \delta} \leq C \norm{P_k f}_\infty n^{-r \delta + \frac{4}{\alpha} - 3} \Lambda (n, \alpha, \gamma) \\
    &\leq C \norm{f}_\infty n^{-r \delta + \frac{4}{\alpha} - 3} \Lambda (n, \alpha, \gamma) \leq C \norm{f}_\infty n^{- 1} \Lambda (n, \alpha, \gamma),
    \end{split}
\end{align}
where we use the fact that $r \delta \geq \frac{4}{\alpha} - 2$ in the last inequality.
 According to \eqref{eq:lemma6pr7}, for $j = 0, 1, \dots, r -1$,
\begin{align} \label{eq:lemma6pr6}
    \begin{split}
    &\mathrel{\phantom{=}} \norm{ (Q_m - \hP_m)^{\otimes j} \hP_m Q_{n- (j+1) m -k-1} (Q_1 - P_1) P_k f}_\infty \\
    &= \norm{ (Q_m - \hP_m)^{\otimes j} Q_{n- (j+1) m -k-1} P_k \hP_m (Q_1 - P_1) f}_\infty \\
    &\leq 2^j \norm{\hP_m (Q_1 - P_1) f }_\infty \leq C \norm{f}_\infty n^{- 1} \Lambda (n, \alpha, \gamma).
    \end{split}
\end{align}
Combining \eqref{eq:lemma6pr4}, \eqref{eq:lemma6pr5} and \eqref{eq:lemma6pr6}, we get the following estimate for $k < m$,
\begin{align} \label{eq:lemma6pr8}
    \begin{split}
    \norm{Q_{n-k-1} (Q_1 - P_1) P_k f}_\infty
    &\leq C \norm{f}_\infty n^{- 1} \Lambda (n, \alpha, \gamma) + \sum_{j=0}^{r-1} C \norm{f}_\infty n^{- 1} \Lambda (n, \alpha, \gamma) \\
    &\leq C \norm{f}_\infty n^{- 1} \Lambda (n, \alpha, \gamma).
    \end{split}
\end{align}
Then desired result follows from \eqref{e:OneStepDec}, \eqref{eq:lemma6pr1}, and \eqref{eq:lemma6pr8}.

It remains to prove \eqref{eq:lemma6pr2}.
According to Lemma \ref{lemma:Q1}, there exist operators $\whQ_l$, $m \leq l \leq n$ such that
\begin{align*}
        \norm{(Q_l - \whQ_l) f}_\infty \leq \norm{f}_\infty \eup^{- C' l}, \quad
        \norm{\nabla^{\kappa} \whQ_l f}_{\op,\infty} \leq C \norm{f}_\infty n^{\frac{\kappa}{\alpha}} l^{-\frac{\kappa}{2}} \leq C \norm{f}_\infty n^{\frac{\kappa}{\alpha} - \frac{\kappa}{2}},
\end{align*}
for $\kappa = 0, 1, \dots, 4$. Combining Lemma \ref{lemma:one step} and the commutativity of $P_1$, $Q_1$ and $\whQ_l$, we have
\begin{align*} 
    \begin{split}
    &\mathrel{\phantom{=}} \norm{Q_l (Q_1 - P_1) f}_\infty \\
    &\leq \norm{(Q_l - \whQ_l) (Q_1 - P_1) f}_\infty + \norm{\whQ_l (Q_1 - P_1) f}_\infty \\
    &\leq \norm{(Q_1 - P_1) f}_\infty \eup^{-C' l} + \norm{(Q_1 - P_1) \whQ_l f}_\infty \\
    &\leq 2 \norm{f}_\infty \eup^{-C' l} + C \left( \sum_{\kappa = 0}^4 \norm{\nabla^{\kappa} \whQ_l f}_{\op,\infty} \right) n^{-1} \Lambda (n, \alpha, \gamma) \\
    &\leq 2 \norm{f}_\infty \eup^{-C' [ \frac{n}{4r} ]} + C \norm{f}_\infty \left( \sum_{\kappa = 0}^4 n^{\frac{\kappa}{\alpha} - \frac{\kappa}{2} - 1} \right) \Lambda (n, \alpha, \gamma) \\
    &\leq C \norm{f}_\infty n^{\frac{4}{\alpha} - 3} \Lambda (n, \alpha, \gamma),
    \end{split}
\end{align*}
where in the third inequality we use Lemma \ref{lemma:one step} to derive
\begin{align*}
    \norm{(Q_1 - P_1) \whQ_l f}_\infty &\leq C \left( \sum_{\kappa = 0}^4 \norm{\nabla^{\kappa} \whQ_l f}_{\op,\infty} \right) n^{-1} \Lambda (n, \alpha, \gamma). \qedhere
\end{align*}
\end{proof}

\begin{proof}[Proof of Theorem \ref{theorem1}]
Thanks to Lemma \ref{lemma6}, to prove the theorem we only need to verify \eqref{eq:lemma6hyp}, i.e.,
\begin{align*}
    \abs{\E f(S_{n}) - \E f(Y_1)} \leq C \norm{f}_\infty n^{-\delta}, \quad \forall n \geq 1, \; f \in \mathcal{C}_{b}^{4}(\R^d),
\end{align*}
for some $C, \delta > 0$. Recall \eqref{e:OneStepDec}, it suffice to bound the following sum on its right-hand side
\begin{align*}
    \sum_{k = 0}^{n-1} Q_{n-k-1} (Q_1 - P_1) P_k f (\mathbf{0}).
\end{align*}
For any fixed $\gamma > 0$, the proof is divided into five parts according to $\alpha$: 
\begin{enumerate}
    \item[(i)] for $\alpha \in ( (2 - \sqrt{2 \gamma}) \lor 1, 2)$;
    \item[(ii)] for $\alpha \in (1, 2 - \sqrt{2 \gamma} \, ]$ if $\gamma < \frac{1}{2}$;
    \item[(iii)] for $\alpha = 1$;
    \item[(iv)] for $\alpha \in [\frac{10}{10 + (\gamma \land 1)}, 1)$;
    \item[(v)] for $\alpha \in (0, \frac{10}{10 + (\gamma \land 1)})$.
\end{enumerate}
\vspace{8 pt}
\noindent (i) For $\alpha \in ( (2 - \sqrt{2 \gamma}) \lor 1, 2)$, as $0 \leq k \leq \left[ n^{\alpha / 2} \right]$, the commutativity of $P_k$ and $Q_l$ implies that
\begin{align*}
    \norm{Q_{n-k-1} (Q_1 - P_1) P_k f}_\infty
    = \norm{P_k (Q_1 - P_1) Q_{n-k-1} f}_\infty
    \leq \norm{(Q_1 - P_1) Q_{n-k-1} f}_\infty.
\end{align*}
Since $n - k - 1 \geq \frac{n}{4}$, according to Lemma \ref{lemma5-0},
\begin{align*}
    \norm{(Q_1 - P_1) Q_{n-k-1} f}_\infty
    \leq C \norm{f}_\infty \Lambda_2 (n, \alpha, 2, \gamma)
    = C \norm{f}_\infty n^{\frac{2}{\alpha} - 1 - \frac{(\alpha + \gamma) \land 2}{\alpha}} \ln n.
\end{align*}
So
\begin{align*}
    \norm{Q_{n-k-1} (Q_1 - P_1) P_k f}_\infty
    \leq C \norm{f}_\infty n^{\frac{2}{\alpha} - 1 - \frac{(\alpha + \gamma) \land 2}{\alpha}} \ln n, \quad 0 \leq k \leq \left[ n^{\frac{\alpha}{2}} \right].
\end{align*}

For $\left[ n^{\alpha / 2} \right] + 1 \leq k \leq n-1$, Lemma \ref{lemma4} derives
\begin{align*}
    \norm{Q_{n-k-1} (Q_1 - P_1) P_k f}_\infty
    \leq \norm{(Q_1 - P_1) P_k f}_\infty
    \leq C \norm{f}_\infty k^{- \frac{2}{\alpha}} n^{\frac{2}{\alpha} - \frac{(\alpha + \gamma) \land 2}{\alpha}} \ln n.
\end{align*}
Notice that
\begin{align*}
    \sum_{k = \left[ n^{\alpha / 2} \right] + 1}^{n-1} k^{- \frac{2}{\alpha}}
    \leq \int_{\left[ n^{\alpha / 2} \right]}^{n-1} x^{- \frac{2}{\alpha}} \, \dif  x
    \leq \frac{\alpha}{2 - \alpha} \left[ n^{\frac{\alpha}{2}} \right]^{1 - \frac{2}{\alpha}}
    \leq C n^{\frac{\alpha}{2} - 1},
\end{align*}
by \eqref{e:OneStepDec} we have
\begin{align*}
    \abs{\E f(S_{n}) - \E f(Y_1)}
    &\leq \sum_{k = 0}^{\left[ n^{\alpha / 2} \right]} \norm{Q_{n-k-1} (Q_1 - P_1) P_k f}_\infty + \sum_{k = \left[ n^{\alpha / 2} \right] + 1}^{n-1} \norm{Q_{n-k-1} (Q_1 - P_1) P_k f}_\infty \\
    &\leq C \norm{f}_\infty n^{\frac{\alpha}{2} + \frac{2}{\alpha} - 1 - \frac{(\alpha + \gamma) \land 2}{\alpha}} \ln n + C \norm{f}_\infty n^{\frac{2}{\alpha} - \frac{(\alpha + \gamma) \land 2}{\alpha}} \ln n \sum_{k = \left[ n^{\alpha / 2} \right] + 1}^{n-1} k^{- \frac{2}{\alpha}} \\
    &\leq C \norm{f}_\infty n^{\frac{\alpha}{2} + \frac{2}{\alpha} - 1 - \frac{(\alpha + \gamma) \land 2}{\alpha}} \ln n.
\end{align*}
Since $\alpha \in ( (2 - \sqrt{2 \gamma}) \lor 1, 2)$ implies that $\frac{\alpha}{2} + \frac{2}{\alpha} - 1 - \frac{(\alpha + \gamma) \land 2}{\alpha} < 0$, there exists $\delta > 0$ such that
\begin{align*}
    \abs{\E f(S_{n}) - \E f(Y_1)} \leq C \norm{f}_\infty n^{-\delta}, \quad \forall n \geq 1, \; f \in \mathcal{C}_{b}^{4}(\R^d).
\end{align*}

\vspace{8 pt}
\noindent (ii) For $\alpha \in (1, 2 - \sqrt{2 \gamma} \, ]$ as $\gamma < \frac{1}{2}$, we prove the conclusion by backward induction on $\alpha$. We first show that whenever
\begin{align} \label{eq:inhyp1}
    \dtv (\mu_n^{\wtalpha}, S_{\wtalpha} (\nu)) \leq c n^{-\delta}, \quad \forall n \geq 1,
\end{align}
holds for some $\wtalpha \in (1, 2)$ and $c, \delta > 0$, we have
\begin{align} \label{eq:incon1}
    \dtv (\mu_n^{\alpha}, S_\alpha (\nu)) \leq C n^{-\delta'}, \quad \forall n \geq 1,
\end{align}
for any $\alpha \in ( (1 - \frac{\gamma}{2}) \wtalpha, \wtalpha) \cap (1, 2)$, where the constants $C, \delta' > 0$ do not depend on $n$.

For $0 \leq k \leq \left[ \frac{n}{3} \right]$, the commutativity of $P_k$ and $Q_l$ implies that
\begin{align*}
    \norm{Q_{n-k-1} (Q_1 - P_1) P_k f}_\infty
    = \norm{P_k (Q_1 - P_1) Q_{n-k-1} f}_\infty
    \leq \norm{(Q_1 - P_1) Q_{n-k-1} f}_\infty.
\end{align*}
Combining \eqref{eq:inhyp1} and $n-k-1 \geq \frac{n}{2}$, Lemma \ref{lemma5} shows
\begin{align*}
        \norm{(Q_1 - P_1) Q_{n-k-1} f}_\infty \leq C \norm{f}_\infty \Lambda_2 (n, \alpha, \wtalpha, \gamma) = C \norm{f}_\infty n^{\frac{2}{\alpha} - \frac{2}{\wtalpha} - \frac{(\alpha + \gamma) \land 2}{\alpha}} \ln n,
\end{align*}
so it follows that
\begin{align*}
        \norm{Q_{n - k - 1} ( Q_1 - P_1 ) P_k f}_\infty \leq C \norm{f}_\infty n^{\frac{2}{\alpha} - \frac{2}{\wtalpha} - \frac{(\alpha + \gamma) \land 2}{\alpha}} \ln n, \quad 0 \leq k \leq \left[ \frac{n}{3} \right].
\end{align*}
For $\left[ \frac{n}{3} \right] + 1 \leq k \leq n-1$, according to Lemma \ref{lemma4},
\begin{align*}
    \norm{Q_{n - k - 1} ( Q_1 - P_1 ) P_k f}_\infty 
    \leq \norm{( Q_1 - P_1 ) P_k f}_\infty
    \leq C \norm{f}_\infty n^{- \frac{(\alpha + \gamma) \land 2}{\alpha}} \ln n.
\end{align*}
By \eqref{e:OneStepDec} we have
\begin{align*}
    \abs{\E f(S_{n}) - \E f(Y_1)}
    &\leq \sum_{k = 0}^{\left[ \frac{n}{3} \right]} \norm{Q_{n-k-1} (Q_1 - P_1) P_k f}_\infty + \sum_{k = \left[ \frac{n}{3} \right] + 1}^{n-1} \norm{Q_{n-k-1} (Q_1 - P_1) P_k f}_\infty \\
    &\leq C \norm{f}_\infty n^{\frac{2}{\alpha} - \frac{2}{\wtalpha} + 1 - \frac{(\alpha + \gamma) \land 2}{\alpha}} \ln n + C \norm{f}_\infty n^{1 - \frac{(\alpha + \gamma) \land 2}{\alpha}} \ln n \\
    &\leq C \norm{f}_\infty n^{\frac{2}{\alpha} - \frac{2}{\wtalpha} + 1 - \frac{(\alpha + \gamma) \land 2}{\alpha}} \ln n.
\end{align*}
Since $\alpha \in ( (1 - \frac{\gamma}{2}) \wtalpha, \wtalpha) \cap (1, 2)$ derives that $\frac{2}{\alpha} - \frac{2}{\wtalpha} + 1 - \frac{(\alpha + \gamma) \land 2}{\alpha} < 0$, there exists $\delta' > 0$ such that
\begin{align*}
    \abs{\E f(S_{n}) - \E f(Y_1)} \leq C \norm{f}_\infty n^{-\delta'}, \quad \forall n \geq 1, \; f \in \mathcal{C}_{b}^{4}(\R^d).
\end{align*}
According to \cite[Theorem 7.10, 8.14]{Folland1999}, any $f \in \mathcal{B}_{b} (\R^d)$ can be approximated almost everywhere by a sequence of $f_i \in \mathcal{C}_{b}^{4}(\R^d)$ satisfying $\norm{f_i}_\infty \leq 2 \norm{f}_\infty$, the above inequality implies that
\begin{align*}
    \abs{\E f(S_{n}) - \E f(Y_1)} \leq C \norm{f}_\infty n^{-\delta'}, \quad \forall n \geq 1, \; f \in \mathcal{B}_{b}(\R^d).
\end{align*}
By Kantorovich-Rubinstein theorem \cite{Villani2003}, it is equivalent to
\begin{align*}
    \dtv (\mu_n^{\alpha}, S_\alpha (\nu)) \leq C n^{-\delta'}, \quad \forall n \geq 1.
\end{align*}

For any fixed $\alpha \in (1, 2 - \sqrt{2 \gamma} \, ]$, there exist $\alpha = \alpha_m < \alpha_{m-1} < \dots < \alpha_1 = 2 - \sqrt{\gamma}$ such that $\alpha_i \in ((1 - \frac{\gamma}{2}) \alpha_{i-1}, \alpha_{i-1})$, $i = 2, \dots, m$. It has been shown in (i) that
\begin{align*}
    \dtv (\mu_n^{\alpha_1}, S_{\alpha_1} (\nu)) \leq C_1 n^{-\delta_1}, \quad \forall n \geq 1,
\end{align*}
for $\alpha_1 = 2 - \sqrt{\gamma} \in (2 - \sqrt{2 \gamma}, 2)$ and some $C_1, \delta_1 > 0$. Using the above result \eqref{eq:incon1} repeatedly, we have
\begin{align*}
    \dtv (\mu_n^{\alpha_i}, S_{\alpha_i} (\nu)) \leq C_i n^{-\delta_i}, \quad \forall n \geq 1,
\end{align*}
for some $C_i, \delta_i > 0$, $i = 2, \dots, m$. Since $\alpha_m = \alpha$, Kantorovich-Rubinstein theorem \cite{Villani2003} implies that
\begin{align*}
    \abs{\E f(S_{n}) - \E f(Y_1)} \leq C_m \norm{f}_\infty n^{-\delta_m}, \quad \forall n \geq 1, \; f \in \mathcal{C}_{b}^{4}(\R^d).
\end{align*}

\vspace{8 pt}
\noindent (iii) For $\alpha = 1$, take $\wtalpha_0 = \frac{5}{5 - (\gamma \land 1)}$. We have shown that
\begin{align} \label{eq:inhyp2}
    \dtv (\mu_n^{\wtalpha_0}, S_{\wtalpha_0} (\nu)) \leq c n^{-\delta}, \quad \forall n \geq 1,
\end{align}
for some $c, \delta > 0$.

For $0 \leq k \leq \left[ \frac{n}{3} \right]$, the commutativity of $P_k$ and $Q_l$ implies that
\begin{align*}
    \norm{Q_{n-k-1} (Q_1 - P_1) P_k f}_\infty
    = \norm{P_k (Q_1 - P_1) Q_{n-k-1} f}_\infty
    \leq \norm{(Q_1 - P_1) Q_{n-k-1} f}_\infty.
\end{align*}
Combining \eqref{eq:inhyp2} and $n-k-1 \geq \frac{n}{2}$, Lemma \ref{lemma5} shows
\begin{align*}
    \norm{(Q_1 - P_1) Q_{n-k-1} f}_\infty \leq C \norm{f}_\infty \Lambda_2 (n, 1, \wtalpha_0, \gamma) = C \norm{f}_\infty n^{3 - \frac{4}{\wtalpha_0} - (\gamma \land 1)} \ln n,
\end{align*}
so it follows that
\begin{align*}
        \norm{Q_{n - k - 1} ( Q_1 - P_1 ) P_k f}_\infty \leq C \norm{f}_\infty n^{3 - \frac{4}{\wtalpha_0} - (\gamma \land 1)} \ln n, \quad 0 \leq k \leq \left[ \frac{n}{3} \right].
\end{align*}
For $\left[ \frac{n}{3} \right] + 1 \leq k \leq n-1$, according to Lemma \ref{lemma4},
\begin{align*}
    \norm{Q_{n - k - 1} ( Q_1 - P_1 ) P_k f}_\infty 
    \leq \norm{( Q_1 - P_1 ) P_k f}_\infty
    \leq C \norm{f}_\infty n^{- 1 - (\gamma \land 1)} \ln n.
\end{align*}
By \eqref{e:OneStepDec} we have
\begin{align*}
    \abs{\E f(S_{n}) - \E f(Y_1)}
    &\leq \sum_{k = 0}^{\left[ \frac{n}{3} \right]} \norm{Q_{n-k-1} (Q_1 - P_1) P_k f}_\infty + \sum_{k = \left[ \frac{n}{3} \right] + 1}^{n-1} \norm{Q_{n-k-1} (Q_1 - P_1) P_k f}_\infty \\
    &\leq C \norm{f}_\infty n^{4 - \frac{4}{\wtalpha_0} - (\gamma \land 1)} \ln n + C \norm{f}_\infty n^{- (\gamma \land 1)} \ln n \\
    &\leq C \norm{f}_\infty n^{4 - \frac{4}{\wtalpha_0} - (\gamma \land 1)} \ln n.
\end{align*}
Since $\wtalpha_0 = \frac{5}{5 - (\gamma \land 1)} \in (1, \frac{4}{4 - (\gamma \land 1)})$ derives that $4 - \frac{4}{\wtalpha_0} - (\gamma \land 1) < 0$, there exists $\delta > 0$ such that
\begin{align*}
    \abs{\E f(S_{n}) - \E f(Y_1)} \leq C \norm{f}_\infty n^{-\delta}, \quad \forall n \geq 1, \; f \in \mathcal{C}_{b}^{4}(\R^d).
\end{align*}

\vspace{8 pt}
\noindent (iv) For $\alpha \in [\frac{10}{10 + (\gamma \land 1)}, 1)$, we have
\begin{align*}
    \frac{2}{\alpha} + \frac{\alpha}{\wtalpha_0}
    \leq \frac{10 + (\gamma \land 1)}{5} + \frac{5 - (\gamma \land 1)}{5} \alpha
    < 3.
\end{align*}
Recall that $\wtalpha_0 = \frac{5}{5 - (\gamma \land 1)}$ and
\begin{align} \label{eq:inhyp3}
    \dtv (\mu_n^{\wtalpha_0}, S_{\wtalpha_0} (\nu)) \leq c n^{-\delta}, \quad \forall n \geq 1,
\end{align}
holds for some $c, \delta > 0$.

For $0 \leq k \leq \left[ n^{\alpha / \wtalpha_0} \right]$, the commutativity of $P_k$ and $Q_l$ implies that
\begin{align*}
    \norm{Q_{n-k-1} (Q_1 - P_1) P_k f}_\infty
    = \norm{P_k (Q_1 - P_1) Q_{n-k-1} f}_\infty
    \leq \norm{(Q_1 - P_1) Q_{n-k-1} f}_\infty.
\end{align*}
Combining \eqref{eq:inhyp3} and $n-k-1 \geq \frac{n}{2}$, Lemma \ref{lemma5} shows
\begin{align*}
    \norm{(Q_1 - P_1) Q_{n-k-1} f}_\infty 
    \leq C \norm{f}_\infty \Lambda_2 (n, \alpha, \wtalpha_0, \gamma)
    = C \norm{f}_\infty \left[ n^{\frac{2}{\alpha} - \frac{2}{\wtalpha_0} - 1 - (\gamma \land 1)} + n^{- \frac{1}{\wtalpha_0}} \right] \ln n,
\end{align*}
so it follows that
\begin{align*}
    \norm{Q_{n - k - 1} ( Q_1 - P_1 ) P_k f}_\infty \leq C \norm{f}_\infty \left[ n^{\frac{2}{\alpha} - \frac{2}{\wtalpha_0} - 1 - (\gamma \land 1)} + n^{- \frac{1}{\wtalpha_0}} \right] \ln n, \quad 0 \leq k \leq \left[ n^{\frac{\alpha}{\wtalpha_0}} \right].
\end{align*}
For $\left[ n^{\alpha / \wtalpha_0} \right] + 1 \leq k \leq n-1$, according to Lemma \ref{lemma4},
\begin{align*}
    \norm{Q_{n - k - 1} ( Q_1 - P_1 ) P_k f}_\infty
    \leq \norm{( Q_1 - P_1 ) P_k f}_\infty
    \leq C \norm{f}_\infty \left[ k^{-\frac{2}{\alpha}} n^{\frac{2}{\alpha} - 1 - (\gamma \land 1)} + k^{-\frac{1}{\alpha}} \right] \ln n.
\end{align*}
Notice that
\begin{align*}
    \sum_{k = \left[ n^{\alpha / \wtalpha_0} \right] + 1}^{n-1} k^{- \frac{\kappa}{\alpha}}
    \leq \int_{\left[ n^{\alpha / \wtalpha_0} \right]}^{n-1} x^{- \frac{\kappa}{\alpha}} \, \dif  x
    \leq \frac{\alpha}{\kappa - \alpha} \left[ n^{\frac{\alpha}{\wtalpha_0}} \right]^{1 - \frac{\kappa}{\alpha}}
    \leq C n^{- \frac{\kappa - \alpha}{\wtalpha_0}}, \quad \kappa = 1, 2,
\end{align*}
by \eqref{e:OneStepDec} we have
\begin{align*}
    &\mathrel{\phantom{=}} \abs{\E f(S_{n}) - \E f(Y_1)} \\
    &\leq \sum_{k = 0}^{\left[ n^{\alpha / \wtalpha_0} \right]} \norm{Q_{n-k-1} (Q_1 - P_1) P_k f}_\infty + \sum_{k = \left[ n^{\alpha / \wtalpha_0} \right] + 1}^{n-1} \norm{Q_{n-k-1} (Q_1 - P_1) P_k f}_\infty \\
    &\leq C \norm{f}_\infty \sum_{k = 0}^{\left[ n^{\alpha / \wtalpha_0} \right]} \left[ n^{\frac{2}{\alpha} - \frac{2}{\wtalpha_0} - 1 - (\gamma \land 1)} + n^{- \frac{1}{\wtalpha_0}} \right] \ln n \\
    &\quad + C \norm{f}_\infty \sum_{k = \left[ n^{\alpha / \wtalpha_0} \right] + 1}^{n-1} \left[ k^{-\frac{2}{\alpha}} n^{\frac{2}{\alpha} - 1 - (\gamma \land 1)} + k^{-\frac{1}{\alpha}} \right] \ln n \\
    &\leq C \norm{f}_\infty n^{\frac{\alpha}{\wtalpha_0}} \left[ n^{\frac{2}{\alpha} - \frac{2}{\wtalpha_0} - 1 - (\gamma \land 1)} + n^{- \frac{1}{\wtalpha_0}} \right] \ln n + C \norm{f}_\infty \left[ n^{- \frac{2 - \alpha}{\wtalpha_0}} n^{\frac{2}{\alpha} - 1 - (\gamma \land 1)} + n^{- \frac{1 - \alpha}{\wtalpha_0}} \right] \ln n \\
    &\leq C \norm{f}_\infty \left[ n^{\frac{2}{\alpha} - \frac{2 - \alpha}{\wtalpha_0} - 1 - (\gamma \land 1)} + n^{- \frac{1 - \alpha}{\wtalpha_0}} \right] \ln n.
\end{align*}
Since $\frac{2}{\alpha} + \frac{\alpha}{\wtalpha_0} < 3$ and $\wtalpha_0 = \frac{5}{5 - (\gamma \land 1)} < \frac{2}{2 - (\gamma \land 1)}$, we have
\begin{align*}
    \frac{2}{\alpha} - \frac{2 - \alpha}{\wtalpha_0} - 1 - (\gamma \land 1) < -\frac{2}{\wtalpha_0} + 2 - (\gamma \land 1) < 0.
\end{align*}
Combining $- \frac{1 - \alpha}{\wtalpha_0} < 0$, there exists $\delta > 0$ such that
\begin{align*}
    \abs{\E f(S_{n}) - \E f(Y_1)} \leq C \norm{f}_\infty n^{-\delta}, \quad \forall n \geq 1, \; f \in \mathcal{C}_{b}^{4}(\R^d).
\end{align*}

\vspace{8 pt}
\noindent (v) For $\alpha \in (0, \frac{10}{10 + (\gamma \land 1)})$, we prove the conclusion by backward induction on $\alpha$. We first show that whenever
\begin{align} \label{eq:inhyp4}
    \dtv (\mu_n^{\wtalpha}, S_{\wtalpha} (\nu)) \leq c n^{-\delta}, \quad \forall n \geq 1,
\end{align}
holds for some $\wtalpha \in (0, 1)$ and $c, \delta > 0$, we have
\begin{align} \label{eq:incon4}
    \dtv (\mu_n^{\alpha}, S_\alpha (\nu)) \leq C n^{-\delta'}, \quad \forall n \geq 1,
\end{align}
for any $\alpha \in ( \frac{2 \wtalpha}{2 + \wtalpha (\gamma \land 1)}, \wtalpha) \cap (0, 1)$, where the constants $C, \delta' > 0$ do not depend on $n$.

For $0 \leq k \leq \left[ \frac{n}{3} \right]$, the commutativity of $P_k$ and $Q_l$ implies that
\begin{align*}
    \norm{Q_{n-k-1} (Q_1 - P_1) P_k f}_\infty
    = \norm{P_k (Q_1 - P_1) Q_{n-k-1} f}_\infty
    \leq \norm{(Q_1 - P_1) Q_{n-k-1} f}_\infty.
\end{align*}
Combining \eqref{eq:inhyp4} and $n-k-1 \geq \frac{n}{2}$, Lemma \ref{lemma5} shows
\begin{align*}
        \norm{(Q_1 - P_1) Q_{n-k-1} f}_\infty \leq C \norm{f}_\infty \Lambda_2 (n, \alpha, \wtalpha, \gamma) = C \norm{f}_\infty \left[ n^{\frac{2}{\alpha} - \frac{2}{\wtalpha} - 1 - (\gamma \land 1)} + n^{- \frac{1}{\wtalpha}} \right] \ln n,
\end{align*}
so it follows that
\begin{align*}
        \norm{Q_{n - k - 1} ( Q_1 - P_1 ) P_k f}_\infty \leq C \norm{f}_\infty \left[ n^{\frac{2}{\alpha} - \frac{2}{\wtalpha} - 1 - (\gamma \land 1)} + n^{- \frac{1}{\wtalpha}} \right] \ln n, \quad 0 \leq k \leq \left[ \frac{n}{3} \right].
\end{align*}
For $\left[ \frac{n}{3} \right] + 1 \leq k \leq n-1$, according to Lemma \ref{lemma4},
\begin{align*}
    \norm{Q_{n - k - 1} ( Q_1 - P_1 ) P_k f}_\infty 
    \leq \norm{( Q_1 - P_1 ) P_k f}_\infty
    \leq C \norm{f}_\infty \left[ n^{- 1 - (\gamma \land 1)} + n^{-\frac{1}{\alpha}} \right] \ln n.
\end{align*}
By \eqref{e:OneStepDec} we have
\begin{align*}
    \abs{\E f(S_{n}) - \E f(Y_1)}
    &\leq \sum_{k = 0}^{\left[ \frac{n}{3} \right]} \norm{Q_{n-k-1} (Q_1 - P_1) P_k f}_\infty + \sum_{k = \left[ \frac{n}{3} \right] + 1}^{n-1} \norm{Q_{n-k-1} (Q_1 - P_1) P_k f}_\infty \\
    &\leq C \norm{f}_\infty n \left[ n^{\frac{2}{\alpha} - \frac{2}{\wtalpha} - 1 - (\gamma \land 1)} + n^{- \frac{1}{\wtalpha}} \right] \ln n + C \norm{f}_\infty n \left[ n^{- 1 - (\gamma \land 1)} + n^{-\frac{1}{\alpha}} \right] \ln n \\
    &\leq C \left[ n^{\frac{2}{\alpha} - \frac{2}{\wtalpha} - (\gamma \land 1)} + n^{1 - \frac{1}{\wtalpha}} \right] \ln n.
\end{align*}
Since $\alpha \in ( \frac{2 \wtalpha}{2 + \wtalpha (\gamma \land 1)}, \wtalpha) \cap (0, 1)$ and $\wtalpha \in (0, 1)$, we have $\frac{2}{\alpha} - \frac{2}{\wtalpha} - (\gamma \land 1) < 0$, $1 - \frac{1}{\wtalpha} < 0$, so there exists $\delta' > 0$ such that
\begin{align*}
    \abs{\E f(S_{n}) - \E f(Y_1)} \leq C \norm{f}_\infty n^{-\delta'}, \quad \forall n \geq 1, \; f \in \mathcal{C}_{b}^{4}(\R^d).
\end{align*}
According to \cite[Theorem 7.10, 8.14]{Folland1999}, any $f \in \mathcal{B}_{b} (\R^d)$ can be approximated almost everywhere by a sequence of $f_i \in \mathcal{C}_{b}^{4}(\R^d)$ satisfying $\norm{f_i}_\infty \leq 2 \norm{f}_\infty$, the above inequality implies that
\begin{align*}
    \abs{\E f(S_{n}) - \E f(Y_1)} \leq C \norm{f}_\infty n^{-\delta'}, \quad \forall n \geq 1, \; f \in \mathcal{B}_{b}(\R^d).
\end{align*}
By Kantorovich-Rubinstein theorem \cite{Villani2003}, it is equivalent to
\begin{align*}
    \dtv (\mu_n^{\alpha}, S_\alpha (\nu)) \leq C n^{-\delta'}, \quad \forall n \geq 1.
\end{align*}

For any fixed $\alpha \in (0, \frac{10}{10 + (\gamma \land 1)})$, there exist $\alpha = \alpha_m < \alpha_{m-1} < \dots < \alpha_1 = \frac{10}{10 + (\gamma \land 1)}$ such that $\frac{1}{\alpha_i} \in (\frac{1}{\alpha_{i-1}}, \frac{1}{\alpha_{i-1}} + \frac{\gamma \land 1}{2})$, or equivalently, $\alpha_i \in ( \frac{2 \alpha_{i-1}}{2 + \alpha_{i-1} (\gamma \land 1)}, \alpha_{i-1})$, $i = 2, \dots, m$. It has been shown in (iv) that
\begin{align*}
    \dtv (\mu_n^{\alpha_1}, S_{\alpha_1} (\nu)) \leq C_1 n^{-\delta_1}, \quad \forall n \geq 1,
\end{align*}
for $\alpha_1 = \frac{10}{10 + (\gamma \land 1)}$ and some $C_1, \delta_1 > 0$. Using the above result \eqref{eq:incon4} repeatedly, we have
\begin{align*}
    \dtv (\mu_n^{\alpha_i}, S_{\alpha_i} (\nu)) \leq C_i n^{-\delta_i}, \quad \forall n \geq 1,
\end{align*}
for some $C_i, \delta_i > 0$, $i = 2, \dots, m$. Since $\alpha_m = \alpha$, Kantorovich-Rubinstein theorem \cite{Villani2003} implies that
\begin{align*}
    \abs{\E f(S_{n}) - \E f(Y_1)} &\leq C_m \norm{f}_\infty n^{-\delta_m}, \quad \forall n \geq 1, \; f \in \mathcal{C}_{b}^{4}(\R^d). \qedhere
\end{align*}
\end{proof}

\begin{proof}[Proof of Theorem \ref{corollary1}]
The Pareto distribution with density function
\begin{align*}
    p(x) = \frac{\alpha \Gamma (\frac{d}{2} + 1)}{\pi^{d / 2} d \abs{x}^{d + \alpha}} \mathbf{1}_{[1, +\infty)} (\abs{x}),
\end{align*}
is symmetric, locally lower bounded by Lebesgue measure and satisfies Assumption \ref{A:mX} with $\gamma > 1$. 
And $\nu_0$ satisfies Assumption \ref{A:nu}. Therefore, Theorem \ref{theorem1} shows that
\begin{align} \label{eq:cor1pr1}
    \dtv ( \mu_n^\alpha, S_\alpha (\nu_0) )
    \leq C n^{-\frac{(2 - \alpha) \land \alpha}{\alpha}}, \quad \forall n \geq 1.
\end{align}

According to Kantorovich-Rubinstein theorem \cite{Villani2003}, to prove
\begin{align*}
    \dtv ( \mu_n^\alpha, S_\alpha (\nu_0) )
    = \sup_{\substack{\norm{f}_\infty \leq 1 \\ f\in \mathcal{B}_{b}(\R^d)}} \abs{\E f(S_n) - \E f(Y)}
    \geq c n^{-\frac{(2 - \alpha) \land \alpha}{\alpha}}, \quad \forall n \geq 1,
\end{align*}
it suffices to show that for the specific $f (x) = \cos \langle e_1, x \rangle$,
\begin{align*}
    \abs{\E \cos \langle e_1, S_n \rangle - \E \cos \langle e_1, Y \rangle}
    \geq c n^{-\frac{(2 - \alpha) \land \alpha}{\alpha}}, \quad \forall n \geq 1,
\end{align*}
where $e_1 = (1, 0, \dots, 0) \in \R^d$, $S_n = \sigma^{-1} n^{-\frac{1}{\alpha}} \sum_{i=1}^{n} X_i$ and $Y \sim S_\alpha (\nu_0)$. Because the distributions of $S_n$ and $Y$ are symmetric, their characteristic functions have the following representations 
\begin{align*}
    \varphi_{S_n} (\lambda) = \E \cos \langle \lambda, S_n \rangle, \qquad
    \varphi_{Y} (\lambda) = \E \cos \langle \lambda, Y \rangle.
\end{align*}
So it is enough to show that
\begin{align*}
    \liminf_{n \to \infty} n^{\frac{(2 - \alpha) \land \alpha}{\alpha}} \abs{\varphi_{S_n} (e_1) - \varphi_{Y} (e_1)} > 0.
\end{align*}

The characteristic function of $X_1$ is given by
\begin{align*}
    \varphi_{X} (\lambda)
    &= 1 - \frac{\alpha \Gamma (\frac{d}{2} + 1)}{\pi^{d / 2} d} \int_{\R^d \setminus B (\mathbf{0}, 1)} \frac{1 - \cos \langle \lambda, x \rangle}{\abs{x}^{d + \alpha}} \, \dif x,
\end{align*}
which derives that
\begin{align*}
    \varphi_{S_n} (\lambda)
    = \left( 1 - \frac{\alpha \Gamma (\frac{d}{2} + 1)}{\pi^{d / 2} d \sigma^\alpha n} \int_{\R^d \setminus B (\mathbf{0}, 1 / (\sigma n^{1 / \alpha}))} \frac{1 - \cos \langle \lambda, x \rangle}{\abs{x}^{d + \alpha}} \, \dif x \right)^n.
\end{align*}
By definition, the characteristic function of $Y$ is given by
\begin{align*}
    \varphi_{Y} (\lambda)
    = \exp \left( -\int_{\mathbb{S}^{d-1}} \abs{\langle \lambda,\theta \rangle}^\alpha \nu_0 (\dif \theta) \right)
    = \exp \left( -\frac{\alpha \Gamma (\frac{d}{2} + 1)}{\pi^{d / 2} d \sigma^\alpha} \int_{\R^d} \frac{1 - \cos \langle \lambda, x \rangle}{\abs{x}^{d + \alpha}} \, \dif x \right).
\end{align*}
It can be derived from \eqref{eq:cor1pr1} that $\varphi_{S_n} (e_1) \to \varphi_Y (e_1)$ as $n \to \infty$, so
\begin{align*}
    \lim_{n \to \infty} \frac{\varphi_{S_n} (e_1) - \varphi_{Y} (e_1)}{\ln \varphi_{S_n} (e_1) - \ln \varphi_{Y} (e_1)}
    = \varphi_{Y} (e_1)
    > 0.
\end{align*}
It suffices to estimate
\begin{align*}
    \Delta_n
    &= n^{\frac{(2 - \alpha) \land \alpha}{\alpha}} (\ln \varphi_{S_n} (e_1) - \ln \varphi_{Y} (e_1)) \\
    &= n^{\frac{2(\alpha \land 1)}{\alpha}} \left[ \ln \left( 1 - \frac{C_0}{n} I \left( n^{-\frac{1}{\alpha}} \right) \right) + \frac{C_0}{n} I (0) \right] \\
    &= n^{\frac{2(\alpha \land 1)}{\alpha}} \left[ \frac{C_0}{n} \left( I (0) - I \left( n^{-\frac{1}{\alpha}} \right) \right) - \frac{C_0^2}{2 n^2} I \left( n^{-\frac{1}{\alpha}} \right)^2 + o \left( \frac{1}{n^2} \right) \right],
\end{align*}
as $n \to \infty$, where 
\begin{align*}
    C_0 = \frac{\alpha \Gamma (\frac{d}{2} + 1)}{\pi^{d / 2} d \sigma^\alpha}, \qquad
    I (r) = \int_{\R^d \setminus B (\mathbf{0}, r / \sigma)} \frac{1 - \cos \langle e_1, x \rangle}{\abs{x}^{d + \alpha}} \, \dif x, \quad r \geq 0.
\end{align*}

By L'H{\^o}pital's rule, it can be shown that
\begin{align*}
    \lim_{r \to 0} \frac{I (0) - I (r)}{r^{2 - \alpha}}
    = \frac{1}{(4 - 2 \alpha) \sigma^{2 - \alpha}} \int_{\mathbb{S}^{d-1}} \abs{\langle e_1, \theta \rangle}^2 \dif \theta
    = \frac{1}{(4 - 2 \alpha) d \sigma^{2 - \alpha}}.
\end{align*}
So we have

\noindent (i) For $\alpha \in (1,2)$,
\begin{align*}
    \lim_{n \to \infty} \Delta_n = \frac{C_0}{(4 - 2 \alpha) d \sigma^{2 - \alpha}} = \frac{\alpha \Gamma (\frac{d}{2} + 1)}{(4 - 2 \alpha) \pi^{d / 2} d^2 \sigma^2}.
\end{align*}

\noindent (ii) For $\alpha \in (0,1)$,
\begin{align*}
    \lim_{n \to \infty} \Delta_n = -\frac{1}{2} \left( \frac{\alpha \Gamma (\frac{d}{2} + 1)}{\pi^{d / 2} d \sigma^\alpha} \int_{\R^d} \frac{1 - \cos \langle e_1, x \rangle}{\abs{x}^{d + \alpha}} \, \dif x \right)^2.
\end{align*}

\noindent (iii) For $\alpha = 1$,
\begin{align*}
    \lim_{n \to \infty} \Delta_n
    = \frac{\Gamma (\frac{d}{2} + 1)}{2 \pi^{d / 2} d^2 \sigma^2} - \frac{1}{2} \left( \frac{\Gamma (\frac{d}{2} + 1)}{\pi^{d / 2} d \sigma} \int_{\R^d} \frac{1 - \cos \langle e_1, x \rangle}{\abs{x}^{d + 1}} \, \dif x \right)^2
    = \frac{2 \Gamma (\frac{d}{2} + 1)}{\pi^{d / 2 + 2} d^2} - \frac{2 \Gamma (\frac{d}{2} + 1)^2}{\pi d^2 \Gamma (\frac{d + 1}{2})^2},
\end{align*}
where the right-hand side of above equation is not equal to zero. The desired result follows from
\begin{align*}
    \lim_{n \to \infty} n^{\frac{(2 - \alpha) \land \alpha}{\alpha}} \abs{\varphi_{S_n} (e_1) - \varphi_{Y} (e_1)}
    &= \varphi_Y (e_1) \lim_{n \to \infty} \abs{\Delta_n}
    > 0. \qedhere
\end{align*}
\end{proof}

\section*{Appendix}

\appendix

\section{One step error}

\begin{lemma} \label{mgeneratorerr}
   Let $\alpha\in (0,2)$ and $\mathcal{L}^{\alpha,\nu}$ be defined as in \eqref{mYgenerator}.
   
   \noindent (i) If $\alpha\in (1,2)$, then there exist a positive constant $C$ such that for any $f \in \mathcal{C}^{2}_{b} \left( \mathbb{R}^{d}\right)$ and $x,y\in \mathbb{R}^{d}$, one has
    \begin{align*}
         \left|\mathcal{L}^{\alpha,\nu}f(x)-\mathcal{L}^{\alpha,\nu}f(y)\right|\le C\|\nabla ^{2}f\|_{\op,\infty}|x-y|^{2-\alpha}.
    \end{align*}
    
    \noindent (ii) If $\alpha\in (0,1)$, then there exist a positive constant $C$ such that for any $f\in \mathcal{C}^{2}_{b}\left( \mathbb{R}^{d}\right)$, one has
    \begin{align*}
        \|\mathcal{L}^{\alpha,\nu} \mathcal{L}^{\alpha,\nu}f \|_{\infty}
        \le C\left( \|f\|_{\infty}+\|\nabla f\|_{\op,\infty}+\|\nabla^{2}f\|_{\op,\infty} \right).
    \end{align*}
    
    \noindent (iii) If $\alpha=1$, then there exist a positive constant $C$ such that for any $f\in \mathcal{C}^{4}_{b}\left( \mathbb{R}^{d}\right)$, one has
    \begin{align*}
        \|\mathcal{L}^{\alpha,\nu} \mathcal{L}^{\alpha,\nu}f \|_{\infty}
        \le C\left(\|f\|_{\infty}+\|\nabla^{2}f\|_{\op,\infty}+\|\nabla^{4}f\|_{\op,\infty} \right).
    \end{align*}
\end{lemma}
\begin{proof}
    (i) When $\alpha\in (1,2)$, by (\ref{mYgenerator}) and Taylor expansion, one has
    \begin{align*}
        &\mathrel{\phantom{=}} \mathcal{L}^{\alpha,\nu}f(x)-\mathcal{L}^{\alpha,\nu}f(y)\\
        &=d_{\alpha} \int_{\mathbb{S}^{d-1}} \nu(\dif \theta) \int_{0}^{+\infty} \frac{\int_{0}^{1} \nabla_{r\theta}f(x+zr\theta)-\nabla_{r\theta} f(x)- \nabla_{r\theta}f(y+zr\theta)+\nabla_{r\theta} f(y) \dif z}{r^{1+\alpha}} \dif r\\
        &=d_{\alpha} \int_{\mathbb{S}^{d-1}} \nu(\dif \theta) \int_{|x-y|}^{+\infty} \frac{\int_{0}^{1}\int_{0}^{1} \nabla_{x-y}\nabla_{\theta}f(y+z_{2}(x-y)+z_{1}r\theta)-\nabla_{x-y}\nabla_{\theta} f(y+z_{2}(x-y)) \dif z_{2}\dif z_{1}}{r^{\alpha}} \dif r\\
        &\quad +d_{\alpha} \int_{\mathbb{S}^{d-1}} \nu(\dif \theta) \int_{0}^{|x-y|} \frac{\int_{0}^{1}\int_{0}^{1} \nabla_{z_{1}\theta}\nabla_{\theta}f(x+z_{2}z_{1}r\theta)-\nabla_{z_{1}\theta}\nabla_{\theta}f(y+z_{2}z_{1}r\theta) \dif z_{2}\dif z_{1}}{r^{\alpha-1}} \dif r,
    \end{align*}
    which yields that
    \begin{align*}
        &\mathrel{\phantom{=}} \left|\mathcal{L}^{\alpha,\nu}f(x)-\mathcal{L}^{\alpha,\nu}f(y)\right|\\
        &\le C \|\nabla^{2}f\|_{\op,\infty}\left[|x-y| \int_{\mathbb{S}^{d-1}} \nu(\dif \theta) \int_{|x-y|}^{+\infty} \frac{1}{r^{\alpha}} \dif r
        +\int_{\mathbb{S}^{d-1}} \nu(\dif \theta) \int_{0}^{|x-y|} \frac{1}{r^{\alpha-1}} \dif r \right]\\
        &\le C\|\nabla ^{2}f\|_{\op,\infty}|x-y|^{2\alpha}.
    \end{align*}
    
    \noindent (ii) For $\alpha\in (0,1)$, it follows from (\ref{mYgenerator}) and Taylor expansion that for any $x\in \mathbb{R}^{d}$,
    \begin{align*}
        \mathcal{L}^{\alpha,\nu} f(x)
        &=d_{\alpha} \int_{\mathbb{S}^{d-1}} \nu(\dif \theta) \int_{0}^{+\infty} \frac{f(x+r\theta)-f(x)}{r^{1+\alpha}} \, \dif r\\
        &=d_{\alpha} \int_{\mathbb{S}^{d-1}} \nu(\dif \theta) \left( \int_{1}^{+\infty} \frac{f(x+r\theta)-f(x)}{r^{1+\alpha}} \, \dif r + \int_{0}^{1} \frac{\int_{0}^{1}\nabla_{\theta} f(x+zr\theta) \, \dif z}{r^{\alpha}} \, \dif r \right).
    \end{align*}
    As a consequence, we then have
    \begin{align*}
        \|\mathcal{L}^{\alpha,\nu} f\|_{\infty}\le C\left(\|f\|_{\infty}+\|\nabla f\|_{\op,\infty} \right).
    \end{align*}
    At the same time, since $f\in \mathcal{C}_b^2\left(\mathbb{R}^d\right)$, it can be verified that for any $z,x\in \mathbb{R}^{d}$,
    \begin{align*}
        \nabla_{z}\mathcal{L}^{\alpha,\nu} f(x)=d_{\alpha} \int_{\mathbb{S}^{d-1}} \nu(\dif \theta) \int_{0}^{+\infty} \frac{\nabla_{z}f(x+r\theta)-\nabla_{z}f(x)}{r^{1+\alpha}} \dif r.
    \end{align*}
    Hence, a similar argument yields that
    \begin{align*}
        \|\nabla\mathcal{L}^{\alpha,\nu} f\|_{\op,\infty}\le C\left(\|\nabla f\|_{\infty}+\|\nabla^{2}f\|_{\op,\infty}  \right).
    \end{align*}
    Now, combining these two estimates gives us
    \begin{align*}
        \|\mathcal{L}^{\alpha,\nu}\mathcal{L}^{\alpha,\nu}f\|_{\infty}
        &\le C\left( \|\mathcal{L}^{\alpha,\nu}f\|_{\infty}+\|\nabla\mathcal{L}^{\alpha,\nu}f\|_{\op,\infty} \right) \\
        &\le C\left( \|f\|_{\infty}+\|\nabla f\|_{\op,\infty}+\|\nabla^{2}f\|_{\op,\infty} \right),
    \end{align*}
    which is the desired result.
    
    \noindent (iii) Now, we turn to consider the $\alpha=1$ case. Again, by (\ref{mYgenerator}) and Taylor expansion, one can deduce that
    \begin{align*}
        \mathcal{L}^{1,\nu} f(x)
        &=d_{\alpha} \int_{\mathbb{S}^{d-1}} \nu(\dif \theta) \int_{0}^{+\infty} \frac{f(x+r\theta)-f(x)-\nabla_{r\theta}f(x) \mathbf{1}_{(0,1]}(r)}{r^{2}} \dif r \\
        &=d_{\alpha} \int_{\mathbb{S}^{d-1}} \nu(\dif \theta) \int_{0}^{1} \dif r \int_{0}^{1} \dif z_1 \int_{0}^{1} \nabla_{z_{1}\theta}\nabla_{\theta}f(x+z_{1}z_{2}r\theta) \, \dif z_{2} \\
        &\quad +d_{\alpha} \int_{\mathbb{S}^{d-1}} \nu(\dif \theta) \int_{1}^{+\infty} \frac{f(x+r\theta)-f(x)}{r^{2}} \dif r.
    \end{align*}
    So it follows that
    \begin{align*}
        \| \mathcal{L}^{1,\nu} f\|_{\infty}\le  C\left(\|f\|_{\infty}+\|\nabla^{2}f\|_{\op,\infty}  \right).
    \end{align*}
    What's more, a similar argument also implies that for $\kappa =1,2$,
    \begin{align*}
        \|\nabla^{\kappa}\mathcal{L}^{1,\nu} f\|_{\op,\infty} \le C\left(\|\nabla^{\kappa}f\|_{\op,\infty}+\|\nabla^{\kappa+2}f\|_{\op,\infty}  \right).
    \end{align*}
    Then combining the above results gives us
    \begin{align*}
        \| \mathcal{L}^{1,\nu}\mathcal{L}^{1,\nu} f\|_{\infty}
        &\le C\left(\|\mathcal{L}^{1,\nu}f\|_{\infty}+\frac{1}{n} \|\nabla^{2}\mathcal{L}^{1,\nu}f\|_{\op,\infty}  \right)\le C\left(\|f\|_{\infty}+\|\nabla^{2}f\|_{\op,\infty}+\|\nabla^{4}f\|_{\op,\infty} \right),
    \end{align*}
    and the proof is complete.
\end{proof}

\begin{lemma}\label{corollary:m one step}
    Under the setting of Lemma \ref{mgeneratorerr}, we have the following result
    \begin{align*}
        &\mathrel{\phantom{=}} \left|\mathbb{E} \int_{0}^{\frac{1}{n}} \mathcal{L}^{\alpha,\nu}f(\widehat{Y}_{s}+x)-\mathcal{L}^{\alpha,\nu}f(x) \dif s \right| \\
        &\le \begin{cases}
            C\|\nabla^{2}f\|_{\op,\infty} n^{-\frac{2}{\alpha}}, & \alpha\in (1,2); \\            C\left(\|f\|_{\infty}+\|\nabla^{2}f\|_{\op,\infty}+\|\nabla^{4}f\|_{\op,\infty} \right) n^{-2}, & \alpha=1;\\
            C\left(\|f\|_{\infty}+\|\nabla f\|_{\op,\infty}+\|\nabla^{2}f\|_{\op,\infty} \right)n^{-2}, &\alpha\in (0,1).
        \end{cases}
    \end{align*}
\end{lemma}

\begin{proof}
    For $\alpha\in (1,2)$, using (i) of Lemma \ref{mgeneratorerr} yields that
    \begin{align*}
        \mathrel{\phantom{=}} \left|\mathbb{E} \int_{0}^{\frac{1}{n}} \mathcal{L}^{\alpha,\nu}f(\widehat{Y}_{s}+x)-\mathcal{L}^{\alpha,\nu}f(x) \dif s \right| 
        &\le \mathbb{E} \int_{0}^{\frac{1}{n}} \left|\mathcal{L}^{\alpha,\nu}f(\widehat{Y}_{s}+x)-\mathcal{L}^{\alpha,\nu}f(x)\right| \dif s \\
        &\le   C\|\nabla^{2}f\|_{\op,\infty}\int_{0}^{\frac{1}{n}} \mathbb{E}\left|s^{\frac{1}{\alpha}}\hat{Y_{1}}\right|^{2-\alpha} \dif s\\
        &\le C\|\nabla^{2}f\|_{\op,\infty} \int_{0}^{\frac{1}{n}} s^{\frac{2-\alpha}{\alpha}} \dif s \le C\|\nabla^{2}f\|_{\op,\infty}n^{-\frac{2}{\alpha}}.
    \end{align*}
    On the other hand, it is clear that
    \begin{align*}
        \left|\mathbb{E} \int_{0}^{\frac{1}{n}} \mathcal{L}^{\alpha,\nu}f(\widehat{Y}_{s}+x)-\mathcal{L}^{\alpha,\nu}f(x) \dif s \right| 
        =\left| \int_{0}^{\frac{1}{n}} \dif s \int_{0}^{s} \mathbb{E}\mathcal{L}^{\alpha,\nu}\mathcal{L}^{\alpha,\nu}f(\widehat{Y}_{u}+x) \dif u \right| \le n^{-2} \|\mathcal{L}^{\alpha,\nu}\mathcal{L}^{\alpha,\nu}f\|_{\infty}.
    \end{align*}
    Hence the desired result then follows from an application of Lemma \ref{mgeneratorerr} (ii) and (iii).
\end{proof}

\begin{lemma} \label{Lemma:m one step}
    Let $X$ be a $d$-dimensional random vector satisfying Assumption \ref{A:mX} with $\alpha\in (0,2)$, $\gamma > 0$. Then there exists a constant $C$ such that for any $a\in (0,A^{-1 / \alpha}\wedge 1), x\in \mathbb{R}^{d}$ and $f\in \mathcal{C}_b^2\left(\mathbb{R}^d\right)$, it holds that:
    
    \noindent (i) If $\alpha\in (1,2)$, we have
    \begin{align*}
        &\mathrel{\phantom{=}} \left| \mathbb{E}\left[f(x+aX)-f(x)-\langle aX,\nabla f(x) \rangle\right]-\frac{A\alpha a^{\alpha}}{d_{\alpha}}\mathcal{L}^{\alpha,\nu}f(x) \right|\\
        &\le C\left(a^{2} \|\nabla^{2} f\|_{\op,\infty}\int_{\mathbb{S}^{d-1}} \nu(\dif \theta) \int_{A^{\frac{1}{\alpha}}}^{a^{-1}} \frac{|\epsilon(r,\theta)|}{r^{\alpha-1}} \dif r + a^{\alpha} \|\nabla f\|_{\op,\infty} \sup_{\theta\in \mathbb{S}^{d-1},r\ge a^{-1}}\left|\epsilon(r,\theta) \right| \right).
    \end{align*}
    
    \noindent (ii) If $\alpha\in (0,1)$ and $\gamma\in (1-\alpha,+\infty)$, we have
    \begin{align*}
        &\mathrel{\phantom{=}} \left| \mathbb{E}\left[f(x+aX)-f(x)\right]-\frac{A\alpha a^{\alpha}}{d_{\alpha}}\mathcal{L}^{\alpha,\nu}f(x) \right|\\
        &\le  Ca^{2}\|\nabla^{2}f\|_{\op,\infty}\left(1+ \int_{\mathbb{S}^{d-1}} \nu(\dif \theta) \int_{A^{\frac{1}{\alpha}}}^{a^{-1}} \frac{|\epsilon(r,\theta)|}{r^{\alpha-1}} \dif r \right)\\
        &\quad +Ca\|\nabla f\|_{\op,\infty}\left( \left| \int_{\mathbb{S}^{d-1}} \nu(\dif \theta) \int_{0}^{a^{-1}}  \frac{\theta \epsilon(r,\theta)}{r^{\alpha}}\dif r \right| + \int_{\mathbb{S}^{d-1}} \nu(\dif \theta) \int_{a^{-1}}^{+\infty} \frac{|\epsilon(r,\theta)|}{r^{\alpha}} \dif r \right),
    \end{align*}
    while if $\gamma\in (0,1-\alpha]$, we have
    \begin{align*}
        &\mathrel{\phantom{=}} \left| \mathbb{E}\left[f(x+aX)-f(x)\right]-\frac{A\alpha a^{\alpha}}{d_{\alpha}}\mathcal{L}^{\alpha,\nu}f(x) \right|\\
        &\le Ca^{2}\|\nabla^{2}f\|_{\op,\infty}\left(1+ \int_{\mathbb{S}^{d-1}} \nu(\dif \theta) \int_{A^{\frac{1}{\alpha}}}^{a^{-1}} \frac{|\epsilon(r,\theta)|}{r^{\alpha-1}} \dif r \right)\\
        &\quad +Ca\|\nabla f\|_{\op,\infty} \left|\int_{\mathbb{S}^{d-1}} \nu(\dif \theta) \int_{0}^{a^{-1}} \frac{\theta \epsilon(r,\theta)}{r^{\alpha}} \dif r \right| \\
        &\quad + C a^{\frac{\alpha}{1 - \gamma}} \norm{f}_{\infty} + C a \norm{\nabla f}_{\op,\infty} \int_{\mathbb{S}^{d-1}} \nu(\dif \theta) \int_{a^{-1}}^{a^{-\frac{1}{1 - \gamma}}} \frac{|\epsilon(r,\theta)|}{r^{\alpha}} \dif r.
    \end{align*}
    
    \noindent (iii) If $\alpha=1$, we have
    \begin{align*}
    &\mathrel{\phantom{=}} \left| \mathbb{E}\left[f(x+aX)-f(x)- \mathbf{1}_{(0,1]}(a|X|)\langle aX,\nabla f(x) \rangle\right]-\frac{A\alpha a^{\alpha}}{d_{\alpha}}\mathcal{L}^{\alpha,\nu}f(x) \right|\\
    &\le C a \left(\|f\|_{\infty} +\|\nabla f\|_{\op,\infty}\right) \left|\int_{\mathbb{S}^{d-1}} \epsilon(a^{-1},\theta) \nu(\dif \theta) \right| \\
    &\quad + C a \|\nabla f\|_{\op,\infty} \int_{\mathbb{S}^{d-1}} \nu(\dif \theta) \int_{a^{-1}}^{+\infty} \frac{|\epsilon(r,\theta)|}{r} \dif r \\
    &\quad +C a^2 \|\nabla^{2}f\|_{\op,\infty} \left( 1 + \int_{\mathbb{S}^{d-1}} \nu(\dif \theta) \int_{A}^{a^{-1}} \left|\epsilon(r,\theta)\right| \dif r \right).
    \end{align*}
\end{lemma}

\begin{proof}
    By (\ref{mYgenerator}), it is clear that
    \begin{align*}
        \frac{A\alpha a^{\alpha}}{d_{\alpha}}\mathcal{L}^{\alpha,\nu}f(x)
        &=A \alpha  \int_{\mathbb{S}^{d-1}} \nu(\dif \theta) \int_{0}^{+\infty} \frac{f(x+ar\theta)-f(x)-k_{\alpha}(ar)\langle \nabla f(x), ar\theta\rangle}{r^{1+\alpha}} \dif r\\
        &=A \alpha  \int_{\mathbb{S}^{d-1}} \nu(\dif \theta) \int_{A^{\frac{1}{\alpha}}}^{+\infty} \frac{f(x+ar\theta)-f(x)-k_{\alpha}(ar)\langle \nabla f(x), ar\theta\rangle}{r^{1+\alpha}} \dif r+\mathcal{R},
    \end{align*}
    where the first equality is a consequence of change of variable and the remainder term $\mathcal{R}$ is given by
    \begin{align*}
        \mathcal{R}=A \alpha  \int_{\mathbb{S}^{d-1}} \nu(\dif \theta) \int_{0}^{A^{\frac{1}{\alpha}}} \frac{f(x+ar\theta)-f(x)-k_{\alpha}(ar)\langle \nabla f(x), ar\theta\rangle}{r^{1+\alpha}} \dif r.
    \end{align*}
    Notice that 
$A \alpha  \int_{\mathbb{S}^{d-1}} \nu(\dif \theta) \int_{A^{\frac{1}{\alpha}}}^{+\infty} \frac{1}{r^{1+\alpha}} \dif r=1$,
    we can consider a $d$-dimensional random vector $\Tilde{X}$ which satisfies 
    \begin{align} \label{def:mtX}
        \mathbb{P}\left(|\Tilde{X}|>r,\frac{\Tilde{X}}{|\Tilde{X}|}\in B \right)
        &= \frac{A}{r^{\alpha}} \int_{B} \nu(\dif \theta) = \frac{A\nu(B)}{r^{\alpha}}, \quad \forall r\ge A^{\frac{1}{\alpha}}, B\in \mathscr{B}\left(\mathbb{S}^{d-1} \right) .
    \end{align}
    And it follows that
    \begin{align*}
        \frac{A\alpha a^{\alpha}}{d_{\alpha}}\mathcal{L}^{\alpha,\nu}f(x)=\mathbb{E}\left[f(x+a\Tilde{X})-f(x)-k_{\alpha}(a|\Tilde{X}|)\langle a\Tilde{X},\nabla f(x) \rangle\right]+\mathcal{R}.
    \end{align*}
    We denote $F_{\theta,X} (r) = 1 -\frac{A + \epsilon(r,\theta)}{r^{\alpha}}$ and $F_{\Tilde{X}}(r) = (1 - \frac{A}{r^{\alpha}}) \mathbf{1}_{[A^{1 / \alpha}, +\infty)} (r)$ for $r > 0$. Then by (\ref{def:mX}) and (\ref{def:mtX}), we can deduce that
    \begin{align*}
        &\mathrel{\phantom{=}} \mathbb{E} \left[f(x+aX)-f(x)-k_{\alpha}(a|X|)\langle aX,\nabla f(x) \rangle\right]-\frac{A\alpha a^{\alpha}}{d_{\alpha}} \mathcal{L}^{\alpha,\nu} f(x) \\
        &= \int_{\mathbb{S}^{d-1}} \nu(\dif \theta) \int_{0}^{+\infty} \left[f(x+ar\theta)-k_{\alpha}(ar)\langle ar\theta, \nabla f(x) \rangle\right] \dif \left[F_{\theta,X}(r)-F_{\Tilde{X}}(r) \right] - \mathcal{R}.
    \end{align*}
    
    \noindent (i) For $\alpha\in (1,2)$, since $k_{\alpha}(r)=1$ in this case, using integration by parts yields that
    \begin{align*}
        &\mathrel{\phantom{=}} \left| \int_{\mathbb{S}^{d-1}} \nu(\dif \theta) \int_{0}^{+\infty} \left[f(x+ar\theta)-\langle ar\theta, \nabla f(x) \rangle\right] \dif \left[F_{\theta,X}(r)-F_{\Tilde{X}}(r) \right] \right|\\
        &= \left| \int_{\mathbb{S}^{d-1}} \nu(\dif \theta) \int_{0}^{+\infty} \left[\nabla_{a\theta}f(x+ar\theta)-\nabla_{a\theta}f(x) \right]\left[F_{\theta,X}(r)-F_{\Tilde{X}}(r) \right] \dif r \right| \\
        &\le \left| \int_{\mathbb{S}^{d-1}} \nu(\dif \theta) \int_{0}^{A^{\frac{1}{\alpha}}} \left[ \nabla_{a\theta} f(x+ar\theta) - \nabla_{a\theta}f(x) \right]\left[F_{\theta,X}(r)-F_{\Tilde{X}}(r) \right] \dif r \right| \\
        &\quad +\left| \int_{\mathbb{S}^{d-1}} \nu(\dif \theta) \int_{A^{\frac{1}{\alpha}}}^{+\infty} \left[ \nabla_{a\theta} f(x+ar\theta) - \nabla_{a\theta}f(x) \right]\frac{\epsilon(r,\theta)}{r^{\alpha}} \dif r \right|.
    \end{align*}
    For the first term, it is clear that $\left|F_{\theta,X}(r)-F_{\Tilde{X}}(r) \right|\le 1$, so it follows from Taylor expansion that
    \begin{align*}
        \left| \int_{\mathbb{S}^{d-1}} \nu(\dif \theta) \int_{0}^{A^{\frac{1}{\alpha}}} \left[ \nabla_{a\theta}f(x+ar\theta)-\nabla_{a\theta}f(x) \right]\left[F_{\theta,X}(r)-F_{\Tilde{X}}(r) \right] \dif r \right|
        &\le Ca^{2}\|\nabla^{2} f\|_{\op,\infty}.
    \end{align*}
    On the other hand, we have
    \begin{align*}
        \left| \int_{\mathbb{S}^{d-1}} \nu(\dif \theta) \int_{a^{-1}}^{+\infty} \left[ \nabla_{a\theta}f(x+ar\theta)-\nabla_{a\theta}f(x) \right]\frac{\epsilon(r,\theta)}{r^{\alpha}} \dif r \right|
        &\le Ca^{\alpha} \|\nabla f\|_{\op,\infty} \sup_{\theta\in \mathbb{S}^{d-1}, r\ge a^{-1}}\left|\epsilon(r,\theta) \right|, \\
        \left| \int_{\mathbb{S}^{d-1}} \nu(\dif \theta) \int_{A^{\frac{1}{\alpha}}}^{a^{-1}} \left[\nabla_{a\theta}f(x+ar\theta)-\nabla_{a\theta}f(x) \right]\frac{\epsilon(r,\theta)}{r^{\alpha}} \dif r \right|
        &\le Ca^{2}\|\nabla^{2} f\|_{\op,\infty} \int_{\mathbb{S}^{d-1}} \nu(\dif \theta) \int_{A^{\frac{1}{\alpha}}}^{a^{-1}} \frac{|\epsilon(r,\theta)|}{r^{\alpha-1}} \dif r.
    \end{align*}
    What's more, for the remainder term $\mathcal{R}$, we also have
    \begin{align*}
        \left|\mathcal{R} \right|
        &=A \alpha  \left|\int_{\mathbb{S}^{d-1}} \nu(\dif \theta) \int_{0}^{A^{\frac{1}{\alpha}}} \frac{f(x+ar\theta)-f(x)-\langle \nabla f(x), ar\theta\rangle}{r^{1+\alpha}} \dif r \right|\\
        &=A \alpha  \left|\int_{\mathbb{S}^{d-1}} \nu(\dif \theta) \int_{0}^{A^{\frac{1}{\alpha}}} \dif r \int_{0}^{1} \dif z_1 \int_{0}^{1}  \frac{\nabla_{z_1 a r \theta}\nabla_{a r \theta} f ( x + z_{1} z_{2} a r \theta)}{r^{1 + \alpha}} \dif z_2 \right| \le Ca^{2} \|\nabla^{2} f\|_{\op,\infty}.
    \end{align*}
    Now, combining the estimates for each part gives us
    \begin{align*}
        &\mathrel{\phantom{=}} \left| \mathbb{E}\left[f(x+aX)-f(x)-\langle aX,\nabla f(x) \rangle\right]-\frac{A\alpha a^{\alpha}}{d_{\alpha}}\mathcal{L}^{\alpha,\nu}f(x) \right|\\
        &\le C\left(a^{2} \|\nabla^{2} f\|_{\op,\infty}\int_{\mathbb{S}^{d-1}} \nu(\dif \theta) \int_{A^{\frac{1}{\alpha}}}^{a^{-1}} \frac{|\epsilon(r,\theta)|}{r^{\alpha-1}} \dif r + a^{\alpha} \|\nabla f\|_{\op,\infty} \sup_{\theta\in \mathbb{S}^{d-1},r\ge a^{-1}}\left|\epsilon(r,\theta) \right| \right).
    \end{align*}
    
    \noindent (ii) Now we consider $\alpha\in (0,1)$ and $k_{\alpha}(r)=0$. Using integration by parts as before shows that 
    \begin{align} \label{eq:lemma4.4pr1}
        \begin{split}
        &\mathrel{\phantom{=}} \left| \int_{\mathbb{S}^{d-1}} \nu(\dif \theta) \int_{0}^{+\infty} f(x+ar\theta) \dif \left[F_{\theta,X}(r)-F_{\Tilde{X}}(r) \right] - \mathcal{R} \right|\\
        &\le \left| \int_{\mathbb{S}^{d-1}} \nu(\dif \theta) \int_{0}^{a^{-1}} \nabla_{a\theta}f(x+ar\theta)\left[F_{\theta,X}(r)-F_{\Tilde{X}}(r) \right] \dif r - \mathcal{R} \right|\\
        &\quad +\left| \int_{\mathbb{S}^{d-1}} \nu(\dif \theta) \int_{a^{-1}}^{+\infty} \nabla_{a\theta}f(x+ar\theta)\left[F_{\theta,X}(r)-F_{\Tilde{X}}(r) \right] \dif r \right|.
        \end{split}
    \end{align}
    For the first term on the right-hand side,
    \begin{align*}
        &\mathrel{\phantom{=}} \int_{\mathbb{S}^{d-1}} \nu(\dif \theta) \int_{0}^{a^{-1}} \nabla_{a\theta} f( x + a r \theta) \left[F_{\theta,X}(r)-F_{\Tilde{X}}(r) \right] \dif r - \mathcal{R} \\
        &= \int_{\mathbb{S}^{d-1}} \nu(\dif \theta) \int_{0}^{a^{-1}} [\nabla_{a\theta} f(x+ar\theta) -\nabla_{a\theta} f(x)] \left[F_{\theta,X}(r)-F_{\Tilde{X}}(r) \right] \dif r \\
        &\quad + a \left\langle \nabla f(x), \int_{\mathbb{S}^{d-1}} \nu(\dif \theta) \int_{0}^{a^{-1}}  \frac{\theta \epsilon(r,\theta)}{r^{\alpha}}\dif r \right\rangle \\
        &\quad + \int_{\mathbb{S}^{d-1}} \nabla_{a\theta} f(x) \nu(\dif \theta) \int_{0}^{A^{\frac{1}{\alpha}}} \left(1-\frac{A}{r^{\alpha}} \right) \dif r - \mathcal{R}.
    \end{align*}
    Combining $\mathcal{R} = A \alpha  \int_{\mathbb{S}^{d-1}} \nu(\dif \theta) \int_{0}^{A^{1 / \alpha}} \frac{f(x+ar\theta)-f(x)}{r^{1+\alpha}} \, \dif r$ and
 $\int_{0}^{A^{\frac{1}{\alpha}}} \left( 1-\frac{A}{r^{\alpha}} \right) \dif r
        = -\frac{\alpha A^{\frac{1}{\alpha}}}{1-\alpha}
        = -A \alpha \int_{0}^{A^{\frac{1}{\alpha}}} \frac{1}{r^{\alpha}} \dif r$,
    we have
    \begin{align*}
        &\mathrel{\phantom{=}} \abs{\int_{\mathbb{S}^{d-1}} \nabla_{a\theta} f(x) \nu(\dif \theta) \int_{0}^{A^{\frac{1}{\alpha}}} \left(1-\frac{A}{r^{\alpha}} \right) \dif r - \mathcal{R}} \\
        &= \abs{A \alpha \int_{\mathbb{S}^{d-1}} \nu(\dif \theta) \int_{0}^{A^{\frac{1}{\alpha}}} \frac{f(x+ar\theta)-f(x)-\nabla_{ar\theta}f(x)}{r^{1+\alpha}} \dif r} \\
        &= A \alpha a^2 \abs{\int_{\mathbb{S}^{d-1}} \nu(\dif \theta) \int_{0}^{A^{\frac{1}{\alpha}}} \dif r \int_{0}^{1} \dif z_1 \int_{0}^{1} \frac{\nabla_{z_{1}\theta}\nabla_{\theta}f(x + z_1 z_2 ar\theta)}{r^{\alpha-1}} \dif z_2} \leq C a^2 \norm{\nabla^2 f}_{\op, \infty}.
    \end{align*}
    Together with
    \begin{gather*}
        \begin{split}
        &\mathrel{\phantom{=}} \abs{ \int_{\mathbb{S}^{d-1}} \nu(\dif \theta) \int_{0}^{a^{-1}} [\nabla_{a\theta} f(x+ar\theta) -\nabla_{a\theta} f(x)] \left[F_{\theta,X}(r)-F_{\Tilde{X}}(r) \right] \dif r} \\
        &\leq C a^2 \norm{\nabla^2 f}_{\op, \infty} \int_{\mathbb{S}^{d-1}} \nu(\dif \theta) \int_{0}^{a^{-1}} r \abs{F_{\theta,X}(r)-F_{\Tilde{X}}(r)} \dif r \\
        &\leq C a^2 \norm{\nabla^2 f}_{\op, \infty} \left( 1 + \int_{\mathbb{S}^{d-1}} \nu(\dif \theta) \int_{A^{\frac{1}{\alpha}}}^{a^{-1}} \frac{\abs{\epsilon(r,\theta)}}{r^{\alpha - 1}} \dif r \right),
        \end{split} \\
        \abs{a \left\langle \nabla f(x), \int_{\mathbb{S}^{d-1}} \nu(\dif \theta) \int_{0}^{a^{-1}}  \frac{\theta \epsilon(r,\theta)}{r^{\alpha}}\dif r \right\rangle}
        \leq a \norm{\nabla f}_{\op, \infty} \abs{\int_{\mathbb{S}^{d-1}} \nu(\dif \theta) \int_{0}^{a^{-1}}  \frac{\theta \epsilon(r,\theta)}{r^{\alpha}}\dif r},
    \end{gather*}
    the first term on the right-hand side of \eqref{eq:lemma4.4pr1} can be estimated as
    \begin{align*}
        &\mathrel{\phantom{=}} \abs{\int_{\mathbb{S}^{d-1}} \nu(\dif \theta) \int_{0}^{a^{-1}} \nabla_{a\theta}f(x+ar\theta)\left[F_{\theta,X}(r)-F_{\Tilde{X}}(r) \right] \dif r - \mathcal{R}} \\
        &\leq C a^2 \norm{\nabla^2 f}_{\op, \infty} \left( 1 + \int_{\mathbb{S}^{d-1}} \nu(\dif \theta) \int_{A^{\frac{1}{\alpha}}}^{a^{-1}} \frac{\abs{\epsilon(r,\theta)}}{r^{\alpha - 1}} \dif r \right) \\
        &\quad + a \norm{\nabla f}_{\op, \infty} \abs{\int_{\mathbb{S}^{d-1}} \nu(\dif \theta) \int_{0}^{a^{-1}}  \frac{\theta \epsilon(r,\theta)}{r^{\alpha}}\dif r}.
    \end{align*}
    
    Now we turn to the second term on the right-hand side of \eqref{eq:lemma4.4pr1}. If $\gamma\in (1-\alpha, +\infty)$, it follows that
    \begin{align*}
        \left| \int_{\mathbb{S}^{d-1}} \nu(\dif \theta) \int_{a^{-1}}^{+\infty} \nabla_{a\theta}f(x+ar\theta)\left[F_{\theta,X}(r)-F_{\Tilde{X}}(r) \right] \dif r  \right|
        \le a\|\nabla f\|_{\op,\infty} \int_{\mathbb{S}^{d-1}} \nu(\dif \theta) \int_{a^{-1}}^{+\infty} \frac{|\epsilon(r,\theta)|}{r^{\alpha}} \dif r.
    \end{align*}
    If $\gamma\in (0,1-\alpha]$, using integration by parts and $|\epsilon(r,\theta)|\le K$ yields that
    \begin{align*}
        &\mathrel{\phantom{=}} \left| \int_{\mathbb{S}^{d-1}} \nu(\dif \theta) \int_{a^{-\frac{1}{1 - \gamma}}}^{+\infty} \nabla_{a\theta}f(x+ar\theta)\left[F_{\theta,X}(r)-F_{\Tilde{X}}(r) \right] \dif r  \right|\\
        &\le \norm{f}_{\infty} \int_{\mathbb{S}^{d-1}} \left|F_{\theta,X}(a^{-\frac{1}{1 - \gamma}}) - F_{\Tilde{X}}(a^{-\frac{1}{1 - \gamma}}) \right| \nu(\dif \theta) \\
        & \quad + \norm{f}_{\infty} \int_{\mathbb{S}^{d-1}} \nu(\dif \theta) \int_{a^{-\frac{1}{1 - \gamma}}}^{+\infty}\left[\dif F_{\theta,X}(r)+\dif F_{\Tilde{X}}(r) \right] \\
        &\le C a^{\frac{\alpha}{1 - \gamma}} \norm{f}_{\infty}.
    \end{align*}
    On the other hand,
    \begin{align*}
        &\mathrel{\phantom{=}} \left| \int_{\mathbb{S}^{d-1}} \nu(\dif \theta) \int_{a^{-1}}^{a^{-\frac{1}{1 - \gamma}}} \nabla_{a\theta}f(x+ar\theta)\left[F_{\theta,X}(r)-F_{\Tilde{X}}(r) \right] \dif r  \right|\\
        &\le C a\|\nabla f\|_{\op,\infty} \int_{\mathbb{S}^{d-1}} \nu(\dif \theta) \int_{a^{-1}}^{a^{-\frac{1}{1 - \gamma}}} \frac{|\epsilon(r,\theta)|}{r^{\alpha}} \dif r.
    \end{align*} 
    So we have
    \begin{align*}
        &\mathrel{\phantom{=}} \left| \int_{\mathbb{S}^{d-1}} \nu(\dif \theta) \int_{a^{-1}}^{+\infty} \nabla_{a\theta}f(x+ar\theta)\left[F_{\theta,X}(r)-F_{\Tilde{X}}(r) \right] \dif r \right| \\
        &\leq C a^{\frac{\alpha}{1 - \gamma}} \norm{f}_{\infty} + C a \|\nabla f\|_{\op,\infty} \int_{\mathbb{S}^{d-1}} \nu(\dif \theta) \int_{a^{-1}}^{a^{-\frac{1}{1 - \gamma}}} \frac{|\epsilon(r,\theta)|}{r^{\alpha}} \dif r,
    \end{align*}
    for $\gamma\in (0,1-\alpha]$. The desired result follows from the estimates of two terms on the right-hand side of \eqref{eq:lemma4.4pr1}.
    
    \noindent (iii) When $\alpha=1$ and $k_1 (r)=\mathbf{1}_{(0,1]}(r)$, recall that we impose $\int_{\mathbb{S}^{d-1}} \theta \nu(\dif \theta) = \mathbf{0}$ in this case. The same as before, we have
    \begin{align*}
        &\mathrel{\phantom{=}} \left| \int_{\mathbb{S}^{d-1}} \nu(\dif \theta) \int_{0}^{+\infty} \left[f(x+ar\theta)- \mathbf{1}_{(0,1]}(ar)\langle ar\theta, \nabla f(x) \rangle\right] \dif \left[F_{\theta,X}(r)-F_{\Tilde{X}}(r) \right] \right|\\
        &\le \left| \int_{\mathbb{S}^{d-1}} \nu(\dif \theta) \int_{0}^{a^{-1}} \left[f(x+ar\theta)-\langle ar\theta, \nabla f(x) \rangle\right] \dif \left[F_{\theta,X}(r)-F_{\Tilde{X}}(r) \right] \right|\\
        &\quad +\left| \int_{\mathbb{S}^{d-1}} \nu(\dif \theta) \int_{a^{-1}}^{+\infty} f(x+ar\theta)\dif \left[F_{\theta,X}(r)-F_{\Tilde{X}}(r) \right] \right|.
    \end{align*}
    Using integration by parts as in (i) yields that
    \begin{align*}
        &\mathrel{\phantom{=}} \left| \int_{\mathbb{S}^{d-1}} \nu(\dif \theta) \int_{0}^{a^{-1}} \left[f(x+ar\theta)-\langle ar\theta, \nabla f(x) \rangle\right] \dif \left[F_{\theta,X}(r)-F_{\Tilde{X}}(r) \right] \right|\\
        &\le C a \left(\|f\|_{\infty} +\|\nabla f\|_{\op,\infty}\right) \left|\int_{\mathbb{S}^{d-1}} \epsilon(a^{-1},\theta) \nu(\dif \theta) \right| \\
        &\quad + C a^2 \|\nabla^{2}f\|_{\op,\infty} \left( 1 + \int_{\mathbb{S}^{d-1}} \nu(\dif \theta) \int_{A}^{a^{-1}} \left|\epsilon(r,\theta)\right| \dif r \right),
    \end{align*}
    while
    \begin{align*}
        &\mathrel{\phantom{=}} \left| \int_{\mathbb{S}^{d-1}} \nu(\dif \theta) \int_{a^{-1}}^{+\infty} f(x+ar\theta)\dif \left[F_{\theta,X}(r)-F_{\Tilde{X}}(r) \right] \right|\\
        &\le C a \|f\|_{\infty} \left|\int_{\mathbb{S}^{d-1}} \epsilon(a^{-1},\theta) \nu(\dif \theta) \right| + C a \|\nabla f\|_{\op,\infty} \int_{\mathbb{S}^{d-1}} \nu(\dif \theta) \int_{a^{-1}}^{+\infty} \frac{|\epsilon(r,\theta)|}{r} \dif r.
    \end{align*}
    And the desired result follows.
\end{proof}

\begin{proof}[Proof of Lemma \ref{lemma:one step}]
    For any fixed $x \in \R^d$, notice that
    \begin{align*}
        Q_{1}f(x)-f(x)
        =\mathbb{E}f\left(x+\frac{X_1 -\omega_{n,\alpha}}{n^{1/\alpha}\sigma}  \right)-f(x)
        =\frac{1}{n} \mathcal{L}^{\alpha,\nu}f(x)+I+II,
    \end{align*}
    where $I,II$ is given by
    \begin{align*}
        I&=  \mathbb{E}\left[f \left(x+\frac{X_1}{n^{1/\alpha}\sigma} \right)-f(x)-k_{\alpha}\left(\frac{|X_1|}{n^{1/\alpha}\sigma} \right)\left\langle \frac{X_1}{n^{1/\alpha}\sigma},\nabla f(x) \right\rangle\right]-\frac{1}{n}\mathcal{L}^{\alpha,\nu}f(x)  ,\\
        II&= \mathbb{E}\left[f\left(x+\frac{X_1-\omega_{n,\alpha}}{n^{1/\alpha}\sigma}  \right)-f\left(x+\frac{X_1}{n^{1/\alpha}\sigma} \right)+k_{\alpha}\left(\frac{|X_1|}{n^{1/\alpha}\sigma} \right)\left\langle \frac{X_1}{n^{1/\alpha}\sigma},\nabla f(x) \right\rangle \right].
    \end{align*}
    On the other hand, for the $d$-dimensional $\alpha$-stable process $\widehat{Y}_{t}$ satisfying (\ref{def:mY}), one has
    \begin{align*}
        P_{1}f(x)-f(x)
        &=\mathbb{E}\left[f(x+\widehat{Y}_{\frac{1}{n}}) \right]-f(x)\\
        &=\frac{1}{n}\mathcal{L}^{\alpha,\nu}f(x) + \int_{0}^{\frac{1}{n}} \mathbb{E}\left[\mathcal{L}^{\alpha,\nu}f(x+\widehat{Y}_{s})-\mathcal{L}^{\alpha,\nu}f(x) \right]\dif s.
    \end{align*}
    Hence, it follows that
    \begin{align*}
        Q_{1}f(x)-P_{1}f(x) =-\int_{0}^{\frac{1}{n}} \mathbb{E}\left[\mathcal{L}^{\alpha,\nu}f(x+\widehat{Y}_{s})-\mathcal{L}^{\alpha,\nu}f(x) \right]\dif s+I+II.
    \end{align*}

    \noindent (i) For $\alpha\in (1,2)$, it is clear that $\abs{\omega_{n,\alpha}} \leq \mathbb{E} \abs{X_1} < +\infty$ and $k_{\alpha}=1$, so it follows from an application of Taylor's expansion that
    \begin{align*}
        \left|II \right|
        &=\left|\mathbb{E}\left[f\left(x+\frac{X_1-\omega_{n,\alpha}}{n^{1/\alpha}\sigma}  \right)-f\left(x+\frac{X_1}{n^{1/\alpha}\sigma} \right)+\left\langle \frac{\omega_{n,\alpha}}{n^{1/\alpha}\sigma},\nabla f(x) \right\rangle \right] \right| \le C\|\nabla^{2}f\|_{\op,\infty} n^{-\frac{2}{\alpha}}.
    \end{align*}
    Taking $a = \sigma^{-1} n^{-1 / \alpha}$ in Lemma \ref{Lemma:m one step}, we have
    \begin{align*}
        \abs{I} \leq C \left( \norm{\nabla f}_{\op,\infty} + \norm{\nabla^{2} f}_{\op,\infty} \right) n^{- 1 -\frac{(2 - \alpha) \land \gamma}{\alpha}} (\ln n)^{\mathbf{1}_{ \{ \gamma = 2 - \alpha \} }},
    \end{align*}
    in which we use the assumption $\abs{\epsilon(r,\theta)} \le K (1\wedge r^{-\gamma})$. The desired result follows from Lemma \ref{corollary:m one step} and above estimates of $I$, $II$.

    \noindent (ii) For $\alpha=1$, $\omega_{n,1} = \mathbb{E} \left[X_1 \mathbf{1}_{(0,\sigma n]}(|X_1|) \right]$ and $k_{1}(r)= \mathbf{1}_{(0,1]}(r)$. Notice that
    \begin{align*}
        \mathbb{E}\left[X_1 \mathbf{1}_{(1,\sigma n]}(|X_1|) \right]
        = \int_{\mathbb{S}^{d-1}} \nu (\dif \theta) \int_1^{\sigma n} r \theta \, \dif F_{\theta, X} (r)
        = \int_{\mathbb{S}^{d-1}} \theta \nu (\dif \theta) \int_1^{\sigma n} r \, \dif F_{\theta, X} (r),
    \end{align*}
    where $F_{\theta,X}(r)=1-\frac{A+\epsilon(r,\theta)}{r}$. By integration by parts,
    \begin{align*}
        \int_1^{\sigma n} r \, \dif F_{\theta, X} (r)
        &= \sigma n F_{\theta, X} (\sigma n) - F_{\theta, X} (1) - \int_1^{\sigma n} F_{\theta, X} (r) \, \dif r \\
        &= A \ln (\sigma n) + \epsilon (1, \theta) - \epsilon (\sigma n, \theta) + \int_1^{\sigma n} \frac{\epsilon (r, \theta)}{r} \, \dif r.
    \end{align*}
    Combining $\int_{\mathbb{S}^{d-1}} \theta \nu (\dif \theta) = \mathbf{0}$ and $\abs{\epsilon(r,\theta)} \le K (1\wedge r^{-\gamma})$, we have
    \begin{align*}
        \abs{\mathbb{E}\left[X_1 \mathbf{1}_{(1,\sigma n]}(|X_1|) \right]}
        = \abs{\int_{\mathbb{S}^{d-1}} \left( \epsilon (1, \theta) - \epsilon (\sigma n, \theta) + \int_1^{\sigma n} \frac{\epsilon (r, \theta)}{r} \, \dif r \right) \theta \nu (\dif \theta)}
        \leq C,
    \end{align*}
    which derives
    \begin{align*}
        \abs{\omega_{n,1}}
        \leq \abs{\mathbb{E} \left[X_1 \mathbf{1}_{(0,1]}(|X_1|) \right]} + \abs{\mathbb{E} \left[X_1 \mathbf{1}_{(1,\sigma n]}(|X_1|) \right]}
        \leq C, \quad \forall n \geq 1.
    \end{align*}
    What's more, \eqref{def:mX} implies the following results,
    \begin{gather*}
        \mathbb{P} (|X_1| > \sigma n ) \leq \int_{\mathbb{S}^{d-1}} \frac{A+\epsilon(\sigma n,\theta)}{\sigma n} \nu(\dif \theta)\le Cn^{-1}, \\
        \mathbb{E}\left[|X_1| \mathbf{1}_{[0,\sigma n]}(|X_1|) \right]
        \le \int_{0}^{\sigma n} \mathbb{P} (|X_1| \geq t) \dif t
        =\int_{0}^{\sigma n} \dif t \int_{\mathbb{S}^{d-1}} \frac{A+\epsilon(t,\theta)}{t} \nu(\dif \theta)
        \le C \ln n, \\
        \mathbb{E}\left[|X_1|^2 \mathbf{1}_{[0,\sigma n]}(|X_1|) \right]
        \le 2 \int_{0}^{\sigma n} t \mathbb{P} (|X_1| \geq t) \dif t
        = 2 \int_{0}^{\sigma n} \dif t \int_{\mathbb{S}^{d-1}} ( A+\epsilon(t,\theta) ) \nu(\dif \theta)
        \le C n.
    \end{gather*}
    
   Now we are ready to estimate $II$ for $\alpha = 1$. Consider that
    \begin{align} \label{eq:lemma2.1pr1}
        \begin{split}
        II
        &= \mathbb{E}\left[\int_{0}^{1} \left\langle \frac{\omega_{n,1}}{n\sigma},\nabla f(x)-\nabla f\left(x+ \frac{X_1-r\omega_{n,1}}{n\sigma}\right)\right\rangle \dif r \right] \\
        &= \int_{0}^{1} \mathbb{E} \left[\left\langle \frac{\omega_{n,1}}{n\sigma},\nabla f(x)-\nabla f\left(x+ \frac{X_1-r\omega_{n,1}}{n\sigma}\right)\right\rangle \mathbf{1}_{[0,\sigma n]}(|X_1|) \right] \dif r \\
        & \quad + \int_{0}^{1} \mathbb{E} \left[ \left\langle \frac{\omega_{n,1}}{n\sigma},\nabla f(x)-\nabla f\left(x+ \frac{X_1-r\omega_{n,1}}{n\sigma}\right)\right\rangle \mathbf{1}_{(\sigma n,+\infty)}(|X_1|) \right] \dif r.
        \end{split}
    \end{align}
    For any fixed $r \in [0,1]$, we denote $Z = X_1-r\omega_{n,1}$ for simplicity, then
    \begin{align*}
        &\mathrel{\phantom{=}} \mathbb{E} \left[\left\langle \frac{\omega_{n,1}}{n\sigma},\nabla f(x)-\nabla f\left(x+ \frac{X_1-r\omega_{n,1}}{n\sigma}\right)\right\rangle \mathbf{1}_{[0,\sigma n]}(|X_1|) \right] \\
        &= \frac{1}{n \sigma} \mathbb{E} \left[ \left( \nabla_{\omega_{n,1}} f(x) - \nabla_{\omega_{n,1}} f\left(x+ \frac{Z}{n\sigma}\right) \right) \mathbf{1}_{[0,\sigma n]}(|X_1|) \right] \\
        &= -\frac{1}{n \sigma} \mathbb{E} \left[ \int_0^1 \left\langle \frac{Z}{n\sigma}, \nabla \nabla_{\omega_{n,1}} f\left(x+ \frac{s Z}{n\sigma}\right) \right\rangle \dif s \mathbf{1}_{[0,\sigma n]}(|X_1|) \right] \\
        &= -\frac{1}{n \sigma} \mathbb{E} \left[ \int_0^1 \left\langle \frac{Z}{n\sigma}, \nabla \nabla_{\omega_{n,1}} f\left(x+ \frac{s Z}{n\sigma}\right) - \nabla \nabla_{\omega_{n,1}} f (x) \right\rangle \dif s \mathbf{1}_{[0,\sigma n]}(|X_1|) \right] \\
        & \quad - \frac{1}{n^2 \sigma^2} \left\langle \E \left[ Z \mathbf{1}_{[0,\sigma n]}(|X_1|) \right], \nabla \nabla_{\omega_{n,1}} f (x) \right\rangle.
    \end{align*}
    It follows that
    \begin{align} \label{eq:lemma2.1pr2}
        \begin{split}
        &\mathrel{\phantom{=}} \abs{\mathbb{E} \left[\left\langle \frac{\omega_{n,1}}{n\sigma},\nabla f(x)-\nabla f\left(x+ \frac{X_1-r\omega_{n,1}}{n\sigma}\right)\right\rangle \mathbf{1}_{[0,\sigma n]}(|X_1|) \right]} \\
        &\leq C \norm{\nabla^3 f}_{\op,\infty} n^{-3} \E \left[ \abs{Z}^2 \mathbf{1}_{[0,\sigma n]}(|X_1|) \right] + C \norm{\nabla^2 f}_{\op,\infty} n^{-2} \abs{\E \left[ Z \mathbf{1}_{[0,\sigma n]}(|X_1|) \right]} \\
        &= C \norm{\nabla^3 f}_{\op,\infty} n^{-3} \left\{ \E \left[ \abs{X_1}^2 \mathbf{1}_{[0,\sigma n]}(|X_1|) \right] + 2r \abs{\omega_{n,1}} \E \left[ \abs{X_1} \mathbf{1}_{[0,\sigma n]}(|X_1|) \right] + r^2 \abs{\omega_{n,1}}^2 \right\} \\
        &\quad + C \norm{\nabla^2 f}_{\op,\infty} n^{-2} \abs{\omega_{n,1}} \left( 1 - r \mathbb{P} (|X_1| \leq \sigma n ) \right) \\
        &\leq C \left( \norm{\nabla^3 f}_{\op,\infty} + \norm{\nabla^2 f}_{\op,\infty} \right) n^{-2}.
        \end{split}
    \end{align}
    On the other hand,
    \begin{align} \label{eq:lemma2.1pr3}
        \begin{split}
        &\mathrel{\phantom{=}} \abs{\mathbb{E} \left[ \left\langle \frac{\omega_{n,1}}{n\sigma},\nabla f(x)-\nabla f\left(x+ \frac{X_1-r\omega_{n,1}}{n\sigma}\right)\right\rangle \mathbf{1}_{(\sigma n,+\infty)}(|X_1|) \right]} \\
        & \leq C \norm{\nabla f}_{\op,\infty} n^{-1} \mathbb{P}\left(|X_1| > \sigma n \right) \leq C \norm{\nabla f}_{\op,\infty} n^{-2}.
        \end{split}
    \end{align}
    Combining \eqref{eq:lemma2.1pr1}, \eqref{eq:lemma2.1pr2} and \eqref{eq:lemma2.1pr3} shows that
    \begin{align*}
        \abs{II}
        \leq C \left( \sum_{\kappa = 1}^3 \norm{\nabla^\kappa f}_{\op,\infty} \right) n^{-2}.
    \end{align*}

    Taking $a = (\sigma n)^{-1}$ in Lemma \ref{Lemma:m one step}, we have
    \begin{align*}
        \abs{I} \leq C \left( \sum_{\kappa = 0}^2 \norm{\nabla^\kappa f}_{\op,\infty} \right) n^{- 1 - (\gamma \land 1)} (\ln n)^{\mathbf{1}_{ \{ \gamma = 1 \} }},
    \end{align*}
    in which we use the assumption $\abs{\epsilon(r,\theta)} \le K (1\wedge r^{-\gamma})$. The desired result follows from Lemma \ref{corollary:m one step} and above estimates of $I$, $II$.
    
    \noindent (iii) For $\alpha \in (0, 1)$, $\omega_{n,\alpha} = \mathbf{0}$ and $k_{\alpha} = 0$,
    so $II=0$. Combining $\abs{\epsilon(r,\theta)} \le K (1\wedge r^{-\gamma})$ and Lemma \ref{Lemma:m one step} with $a = \sigma^{-1} n^{-1 / \alpha}$, it can be derived that
    \begin{align*}
        \abs{I} &\leq C \left( \sum_{\kappa = 0}^2 \norm{\nabla^\kappa f}_{\op,\infty} \right) \left[ n^{-2} + n^{- 1 -\frac{\gamma}{\alpha \lor (1 - \gamma)}} (\ln n)^{\mathbf{1}_{ \{ \gamma = 1 - \alpha \} }} \right]+ C \norm{\nabla f}_{\op,\infty} n^{-\frac{1}{\alpha}} (\ln n)^{\mathbf{1}_{ \{ \gamma = 1 - \alpha \} }}.
    \end{align*}
    Furthermore, on condition that $\int_{\mathbb{S}^{d-1}} \theta \epsilon (r, \theta) \nu (\dif \theta) = \mathbf{0}$, $\forall r > 0$, we have
    \begin{align*}
        \abs{I}
        \leq C \left( \sum_{\kappa = 0}^2 \norm{\nabla^\kappa f}_{\op,\infty} \right) \left[ n^{-2} + n^{- 1 -\frac{\gamma}{\alpha \lor (1 - \gamma)}} (\ln n)^{\mathbf{1}_{ \{ \gamma = 1 - \alpha \} }} \right].
    \end{align*}
    The desired result follows from Lemma \ref{corollary:m one step} and above estimate of $I$.
\end{proof}

\section{Measure decomposition}
\begin{proof}[Proof of Lemma \ref{lemma:decom1}]
By definition, the distribution of $X$ is locally lower bounded by Lebesgue measure means there exist $\varepsilon_{0}$, $\tau>0$ and $a \in \mathbb{R}^d$ such that
\begin{align}
    \PP (X \in E) \geq \varepsilon_0 \Leb (E \cap B (a, \tau) ), \quad \forall E \in \mathscr{B} (\mathbb{R}^d ).
\end{align}
Consider a random vector $\whX$ with density function
\begin{align*}
    p_{\whX} (z) = c \eup^{- \frac{1}{\tau^2 - \abs{z-a}^2}} \mathbf{1}_{B (a, \tau)} (z), \qquad
    c = \left( \int_{B (a, \tau)} \eup^{- \frac{1}{\tau^2 - \abs{z-a}^2}} \, \dif z \right)^{-1},
\end{align*}
then for any $E \in \mathscr{B} (\mathbb{R}^d )$,
\begin{align*}
    \PP (\whX \in E)
    = c \int_{E \cap B (a, \tau)} \eup^{- \frac{1}{\tau^2 - \abs{z-a}^2}} \, \dif z
    \leq c \Leb (E \cap B (a, \tau) )
    \leq c \varepsilon_0^{-1} \PP (X \in E).
\end{align*}
Notice that
\begin{align*}
    0 < \varepsilon_0 c^{-1} = \varepsilon_0 \int_{B (a, \tau)} \eup^{- \frac{1}{\tau^2 - \abs{z-a}^2}} \, \dif z < \varepsilon_0 \Leb (B (a, \tau) ) \leq \PP (X \in B (a, \tau)) \leq 1.
\end{align*}
For $p = 1 - \varepsilon_0 c^{-1} \in (0, 1)$, we take $\chi \sim B(1, p)$ and $U$ independent of $\whX$ satisfying
\begin{align*}
    \PP ( U \in E ) = \frac{1}{p} \left[ \PP (X \in E) - (1-p) \PP ( \whX \in E ) \right],
\end{align*}
which implies that
\begin{align*}
    \PP (X \in E) = (1-p) \PP ( \whX \in E ) + p \PP ( U \in E ) = \PP \left( (1 - \chi) \whX + \chi U \in E \right),
\end{align*}
i.e., $X$ and $(1 - \chi) \whX + \chi U$ possess the same distribution.
\end{proof}

\begin{proof}[Proof of Lemma \ref{lemma:decom2}]
For any fixed $\wtalpha \in (\alpha, 2)$, take $q = 1 - \E [ \abs{X}^{\alpha - \wtalpha} \land 1 ] \in (0, 1)$, and consider independent random variable $\wtchi \sim B(1, q)$ and random vectors $\wtX$, $V$ such that for any $E \in \mathscr{B} (\R^d)$,
\begin{align*}
    \PP ( \wtX \in E ) = \frac{1}{1-q} \E \left[ (\abs{X}^{\alpha - \wtalpha} \land 1) \mathbf{1}_E (X) \right], \quad
    \PP ( V \in E ) = \frac{1}{q} \left[ \PP (X \in E) - (1-q) \PP ( \wtX \in E ) \right],
\end{align*}
which implies that
\begin{align*}
    \PP (X \in E) = (1-q) \PP ( \wtX \in E ) + q \PP ( V \in E ) = \PP \left( (1 - \wtchi) \wtX + \wtchi V \in E \right),
\end{align*}
i.e., $X$ and $(1 - \wtchi) \wtX + \wtchi V$ possess the same distribution. Notice that
\begin{align*}
    \frac{R^{\alpha - \wtalpha} \land 1}{1 - q} \PP ( X \in E )
    \leq \PP ( \wtX \in E )
    \leq \frac{1}{1 - q} \PP ( X \in E ), \quad
    \forall E \subseteq B (\mathbf{0}, R), \; R > 0,
\end{align*}
so the distribution of $\wtX$ is locally lower bounded by the Lebesgue measure if and only if the one of $X$ is. It remains to show that $\wtX$ satisfies Assumption \ref{A:mX} with parameters $\wtalpha$ and $\gamma$. In fact, for $r \geq 1$,
\begin{align*}
    \mathbb{P} \left( \lvert \wtX \rvert \ge r, \, \frac{\wtX}{\lvert \wtX \rvert}\in B \right)
    &= \frac{1}{1-q} \E \left[ \abs{X}^{\alpha - \wtalpha} \mathbf{1}_{\{ \abs{X} \geq r, \, X / \abs{X} \in B \}} \right] \\
    &= \frac{\wtalpha - \alpha}{1-q} \E \left[ \int_r^{+\infty} s^{\alpha - \wtalpha - 1} \mathbf{1}_{\{ r \leq \abs{X} < s, \, X / \abs{X} \in B \}} \, \dif  s \right] \\
    &= \frac{\wtalpha - \alpha}{1-q} \int_r^{+\infty} s^{\alpha - \wtalpha - 1} \, \dif s \int_{B} \left( \frac{A+\epsilon(r,\theta)}{r^{\alpha}} - \frac{A+\epsilon(s,\theta)}{s^{\alpha}} \right) \nu(\dif \theta) \\
    &= \frac{1}{1-q} \int_B \frac{1}{r^{\wtalpha}} \left[ \frac{A \alpha}{\wtalpha} + \epsilon (r, \theta) - (\wtalpha - \alpha) r^{\wtalpha} \int_r^{+\infty} \frac{\epsilon (s, \theta)}{s^{\wtalpha + 1}} \, \dif  s \right] \nu (\dif  \theta).
\end{align*}
Together with
\begin{align*}
    \mathbb{P} \left( \lvert \wtX \rvert \ge r, \, \frac{\wtX}{\lvert \wtX \rvert}\in B \right)
    = \mathbb{P} \left( \lvert \wtX \rvert \ge 1, \, \frac{\wtX}{\lvert \wtX \rvert}\in B \right) + \frac{1}{1-q} \mathbb{P} \left( r \leq \abs{X} < 1, \, \frac{X}{\abs{X}}\in B \right),
\end{align*}
holds for $0 < r < 1$, we have
\begin{align*}
    \mathbb{P} \left( \lvert \wtX \rvert \ge r, \, \frac{\wtX}{\lvert \wtX \rvert}\in B \right)=\int_{B} \frac{\wtA + \wtepsilon(r,\theta)}{r^{\wtalpha}} \nu(\dif \theta), \quad \forall r > 0, \; B \in \mathscr{B}(\mathbb{S}^{d-1}),
\end{align*}
where $\wtA = (A \alpha) / [(1-q) \wtalpha]$, and
\begin{align*}
    \wtepsilon (r, \theta) = \frac{(r \land 1)^{\wtalpha - \alpha}}{1-q} \epsilon (r, \theta) - \frac{(\wtalpha - \alpha) r^{\wtalpha}}{1-q} \int_{r \lor 1}^{+\infty} \frac{\epsilon (s, \theta)}{s^{\wtalpha + 1}} \, \dif  s + \frac{A [\wtalpha r^{\wtalpha - \alpha} - (\wtalpha - \alpha) r^{\wtalpha} - \alpha]}{(1 - q) \wtalpha} \mathbf{1}_{(0, 1)} (r).
\end{align*}
Furthermore,
\begin{align*}
    \abs{\frac{(\wtalpha - \alpha) r^{\wtalpha}}{1-q} \int_r^{+\infty} \frac{\epsilon (s, \theta)}{s^{\wtalpha + 1}} \, \dif  s}
    \leq \frac{K (\wtalpha - \alpha) r^{\wtalpha}}{1-q} \int_r^{+\infty} \frac{1}{s^{\gamma + \wtalpha + 1}} \, \dif  s
    = \frac{K (\wtalpha - \alpha)}{(1-q) (\gamma + \wtalpha) r^\gamma},
\end{align*}
which implies that
\begin{align*}
    \sup_{\theta\in \mathbb{S}^{d-1}} \abs{\wtepsilon(r,\theta)} &\leq \left[ \frac{2A }{(1-q) } + \frac{K (\gamma + 2\wtalpha - \alpha)}{(1-q) (\gamma + \wtalpha)} \right] \left(1\wedge \frac{1}{r^{\gamma}}\right). \qedhere
\end{align*}
\end{proof}

\bibliographystyle{amsplain}

\begin{thebibliography}{10}

\bibitem{Applebaum2009}
David Applebaum, \emph{L{\'e}vy processes and stochastic calculus}, Cambridge
  university press, 2009.

\bibitem{BC16}
Vlad Bally and Lucia Caramellino, \emph{{Asymptotic development for the CLT in
  total variation distance}}, Bernoulli \textbf{22} (2016), no.~4, 2442 --
  2485.

\bibitem{Bhattacharya1968}
R.~N. Bhattacharya, \emph{Berry-esseen bounds for the multi-dimensional central
  limit theorem}, Bulletin (New Series) of the American Mathematical Society
  \textbf{74} (1968), no.~6, 285--287.

\bibitem{BR2010}
Rabi~N. Bhattacharya and R.~Ranga Rao, \emph{Normal approximation and
  asymptotic expansions}, vol.~64, Society for Industrial and Applied
  Mathematics, 2010.

\bibitem{Bogdan2020}
Krzysztof Bogdan, Pawe{\l} Sztonyk, and Victoria Knopova, \emph{Heat kernel of
  anisotropic nonlocal operators}, Documenta Mathematica \textbf{25} (2020),
  1--54.

\bibitem{Xu2021}
Peng Chen, Ivan Nourdin, and Lihu Xu, \emph{{Stein's Method for Asymmetric
  $\alpha$-stable Distributions, with Application to the Stable CLT}}, Journal
  of Theoretical Probability \textbf{34} (2021), no.~3, 1382--1407.

\bibitem{Xu2023}
Peng Chen, Ivan Nourdin, Lihu Xu, and Xiaochuan Yang, \emph{Multivariate stable
  approximation by stein’s method}, Journal of Theoretical Probability
  (2023), 1--43.

\bibitem{Xu2022}
Peng Chen, Ivan Nourdin, Lihu Xu, Xiaochuan Yang, and Rui Zhang,
  \emph{Non-integrable stable approximation by stein’s method}, Journal of
  Theoretical Probability \textbf{35} (2022), no.~2, 1137--1186.

\bibitem{Chen2016}
Zhenqing Chen and Xicheng Zhang, \emph{Heat kernels and analyticity of
  non-symmetric jump diffusion semigroups}, Probability Theory and Related
  Fields \textbf{165} (2016), 267--312.

\bibitem{Folland1999}
Gerald~B Folland, \emph{Real analysis: modern techniques and their
  applications}, vol.~40, John Wiley \& Sons, 1999.

\bibitem{H8102}
Peter Hall, \emph{{On the Rate of Convergence to a Stable Law}}, Journal of the
  London Mathematical Society \textbf{s2-23} (1981), no.~1, 179--192.

\bibitem{H8109}
Peter Hall, \emph{Two-sided bounds on the rate of convergence to a stable law},
  Probability Theory and Related Fields \textbf{57} (1981), no.~3, 349--364.

\bibitem{Ibragimov1971}
I~Ibragimov, \emph{Independent and stationary sequences of random variables},
  Groningen, 1975.

\bibitem{Ken1999}
Sato Ken-Iti, \emph{L{\'e}vy processes and infinitely divisible distributions},
  Cambridge university press, 1999.

\bibitem{KK2000}
Rachel Kuske and Joseph~B. Keller, \emph{Rate of convergence to a stable law},
  SIAM Journal on Applied Mathematics \textbf{61} (2000), no.~4, 1308--1323.

\bibitem{Prohorov1952}
Yu.V. Prohorov, \emph{A local theorem for densities}, Doklady Akad. Nauk SSSR
  (NS) \textbf{83} (1952), 797--800.

\bibitem{SM1962}
S.Kh. Sirazhdinov and M.~Mamatov, \emph{On convergence in the mean for
  densities}, Theory of Probability \& Its Applications \textbf{7} (1962),
  no.~4, 424--428.

\bibitem{Villani2003}
C\'{e}dric Villani, \emph{Topics in optimal transportation}, Graduate Studies
  in Mathematics, vol.~58, American Mathematical Society, Providence, RI, 2003.
  \MR{1964483}

\bibitem{Xu201902}
Lihu Xu, \emph{Approximation of stable law in wasserstein-1 distance by
  stein’s method}, The Annals of Applied Probability \textbf{29} (2019),
  no.~1, 458--504.

\end{thebibliography}

\end{document}